\documentclass[11 pt]{article}

\usepackage{amsmath,amsthm,amsfonts,amssymb,csquotes,epsfig, titlesec, epsfig, float, caption,wrapfig, mathrsfs}
\usepackage[hidelinks]{hyperref}
\usepackage{tikz-cd}
\usepackage{tikz}
\usepackage{xcolor}
\usepackage{verbatim}
\usepackage{accents}
\usepackage{marvosym}
\usepackage[citestyle=alphabetic,bibstyle=alphabetic,backend=bibtex, maxbibnames=10,minbibnames=5]{biblatex}
\usepackage[]{enumerate}
\usepackage[american]{babel}

\usetikzlibrary{calc, decorations.pathreplacing,shapes.misc}
\usetikzlibrary{decorations.pathmorphing}
\usepackage[left=1in,top=1in,right=1in]{geometry}
\usepackage[capitalize]{cleveref}
\usepackage{subcaption}
\usetikzlibrary{shapes.geometric}

\newtheorem{thm}{Theorem}[section]
\newtheorem*{specialThm}{Theorem 3.1\'{}}
\newtheorem{lem}[thm]{Lemma}

\newtheorem{prop}[thm]{Proposition}
\newtheorem{qs}[thm]{Question}
\newtheorem{cor}[thm]{Corollary}
\newtheorem{conj}[thm]{Conjecture}

\theoremstyle{remark} 
\newtheorem{rem}[thm]{Remark}
\crefname{rem}{Remark}{Remarks}
\Crefname{rem}{Remark}{Remarks}
\crefname{lem}{Lemma}{Lemmas}
\Crefname{lem}{Lemma}{Lemmas}
\newtheorem{example}[thm]{Example}
\newtheorem{notation}[thm]{Notation}

\theoremstyle{definition} 
\newtheorem{df}[thm]{Definition} 

\titleformat*{\section}{\normalsize \bfseries \filcenter}
\titleformat*{\subsection}{\normalsize \bfseries }
\newtheorem{mainthm}{Theorem}
\Crefname{mainthm}{Theorem}{Theorems}
\newtheorem{maincor}[mainthm]{Corollary}
\Crefname{maincor}{Corollary}{Corollaries}

\crefformat{equation}{(#2#1#3)}

\makeatletter
\def\namedlabel#1#2{\begingroup
   \def\@currentlabel{#2}\label{#1}\endgroup
}
\makeatother

\DeclareFontFamily{U}{mathx}{}
\DeclareFontShape{U}{mathx}{m}{n}{<-> mathx10}{}
\DeclareSymbolFont{mathx}{U}{mathx}{m}{n}
\DeclareMathAccent{\widehat}{0}{mathx}{"70}
\DeclareMathAccent{\widecheck}{0}{mathx}{"71}

\renewbibmacro{in:}{}

\title{\normalsize \textbf{Resolutions of toric subvarieties by line bundles and applications}}
\author{\normalsize Andrew Hanlon, Jeff Hicks, Oleg Lazarev}
\date{}

\newcommand{\eps}{\varepsilon}

\newcommand{\bb}{\mathbb}

\newcommand{\RR}{\mathbb R}
\newcommand{\ZZ}{\mathbb Z}

\renewcommand{\AA}{\mathbb A}
\newcommand{\NN}{\mathbb N}
\newcommand{\PP}{\mathbb P}
\newcommand{\LL}{\mathbb L}

\newcommand{\TT}{\mathbb T}
\newcommand{\KK}{\mathbb K}
\newcommand{\kk}{\mathbb K}
\newcommand{\GG}{\mathbb G}

\newcommand{\sS}{S}

\newcommand{\into}{\hookrightarrow}
\newcommand{\tensor}{\otimes}

\renewcommand{\Im}{\text{Im}}
\newcommand{\ul}{\underline}
\newcommand{\HE}{\Psi}
\newcommand{\sF}{F}
\newcommand{\ot}{\leftarrow}
\newcommand{\ott}{{\color{red}\leftarrow}}
\newcommand{\tto}{{\color{blue}\rightarrow}}

\newcommand{\bF}{\textbf{F}}

\newcommand{\strata}{\sigma}
\newcommand{\stratb}{\tau}
\newcommand{\stratc}{\upsilon}

\newcommand{\X}{\mathcal X}
\newcommand{\Y}{\mathcal Y}
\newcommand{\bm}{\delta}
\newcommand{\fF}{F}

\newcommand{\gentime}{\text{\ClockLogo}}

\DeclareMathOperator{\id}{id}
\DeclareMathOperator{\Sh}{Sh}
\DeclareMathOperator{\EP}{EP}
\DeclareMathOperator{\cone}{cone}

\DeclareMathOperator{\Hom}{Hom}

\DeclareMathOperator{\sgn}{sgn}

\DeclareMathOperator{\Coh}{Coh}
\DeclareMathOperator{\Crit}{Crit}
\DeclareMathOperator{\ind}{ind}
\DeclareMathOperator{\Pic}{Pic}
\DeclareMathOperator{\st}{\; |\; }

\DeclareMathOperator{\coker}{coker}
\DeclareMathOperator{\str}{star}
\DeclareMathOperator{\Ob}{Ob}

\DeclareMathOperator{\Spec}{Spec}

\DeclareMathOperator{\Frob}{Frob}

\DeclareMathOperator{\Rdim}{Rdim}
\newcommand{\Addresses}{{\bigskip
  \footnotesize

  \noindent A.~Hanlon, \textsc{Department of Mathematics, Dartmouth College}\par\nopagebreak
  \noindent \textit{E-mail address}: \texttt{ahanlon.math@gmail.com}

  \medskip

  \noindent J.~Hicks, \textsc{School of Mathematics,  University of Edinburgh}\par\nopagebreak
  \noindent \textit{E-mail address}: \texttt{jeff.hicks@ed.ac.uk}

  \medskip

  \noindent O.~Lazarev, \textsc{Department of Mathematics,  University of Massachusetts Boston}\par\nopagebreak
  \noindent \textit{E-mail address}: \texttt{oleg.lazarev@umb.edu}

  \medskip

}}

\begin{document}

\clearpage
\setcounter{page}{1}
\maketitle

\begin{abstract}
	Given any toric subvariety $Y$ of a smooth toric variety $X$ of codimension $k$, we construct a length $k$ resolution of $\mathcal O_Y$  by line bundles on $X$. Furthermore, these line bundles can all be chosen to be direct summands of the pushforward of $\mathcal O_X$ under the map of toric Frobenius. The resolutions are built from a stratification of a real torus that was introduced by Bondal and plays a role in homological mirror symmetry.
 
    As a corollary, we obtain a virtual analogue of Hilbert's syzygy theorem for smooth projective toric varieties conjectured by Berkesch, Erman, and Smith. Additionally, we prove that the Rouquier dimension of the bounded derived category of coherent sheaves on a toric variety is equal to the dimension of the variety, settling a conjecture of Orlov for these examples. We also prove Bondal's claim that the pushforward of the structure sheaf under toric Frobenius generates the derived category of a smooth toric variety and formulate a refinement of Uehara's conjecture that this remains true for arbitrary line bundles.
\end{abstract}

\section{Introduction}
A fundamental idea in algebraic geometry is to describe a sheaf via a resolution by sheaves in some well-behaved class. For example, cohomology and other derived functors can be defined in terms of projective or injective resolutions. Moreover, resolutions can contain information not captured by the result of a homological algebra computation as a chain complex often contains more data than its cohomology. 

This paper studies resolutions of sheaves by line bundles on a smooth toric variety $X$. As $X$ admits a description by a combinatorial object, namely a fan $\Sigma$ in the real vector space $N_\RR$ generated by the cocharacter lattice $N$ of torus $\TT$ acting on $X$, it is reasonable to expect that many sheaves have explicit, even combinatorial, resolutions. We produce such resolutions for toric subvarieties using a finite set of line bundles. The line bundles are canonically characterized in terms of the toric Frobenius morphism, and the resolution is encoded by a stratification of a real torus introduced by Bondal \cite{bondal2006derived}. These resolutions additionally have minimal length. In particular, we exhibit an explicit resolution of the diagonal that should be fruitful for computational applications.

Classically, a locally free resolution of any coherent sheaf on a smooth projective variety $X$ of length $\dim(X)$ is guaranteed by Hilbert's syzygy theorem. Such minimal free resolutions have found various striking applications in the study of coherent sheaves on projective space, where these resolutions often involve only direct sums of line bundles. However, resolutions by direct sums of line bundles have greater potential than minimal free resolutions when generalizing to other smooth projective toric varieties \cite{berkesch2020virtual}. One way to produce such a resolution is to apply Hilbert's syzygy theorem instead to the Cox ring of $X$, but the resulting resolution is typically much longer than the dimension of $X$. The resolution of the diagonal constructed here remedies this situation by implying that any coherent sheaf on a smooth projective toric variety $X$ has a length at most $\dim(X)$ resolution by direct sums of line bundles as conjectured by Berkesch, Erman, and Smith \cite{berkesch2020virtual}.

As the derived category was developed as a natural setting for homological algebra that keeps track of resolutions up to quasi-isomorphism, our resolutions also have implications for derived categories of toric varieties. 
Much work on the derived category of a toric variety has focused on generalizing Beilinson's resolution of the diagonal on projective space \cite{beilinson1978coherent} and the resulting full strong exceptional collection of line bundles in $D^b\Coh(\PP^n)$.
For a general toric variety $X$, $D^b\Coh(X)$ always admits a full exceptional collection \cite{kawamata2006derived,kawamata2013derived}. Although there are algorithmic approaches to finding such a collection \cite{ballard2019variation} and thus algebraically describing $D^b\Coh(X)$, there does not appear to be a single canonical choice with geometric meaning. Alternatively, for smooth $X$, $D^b\Coh(X)$ is known to be generated by line bundles, but it is also known that even for smooth toric varieties it is not possible to produce an exceptional collection of line bundles \cite{hille2006counterexample,efimov2014maximal}. Instead, the resolution of the diagonal presented here implies that there is a canonically characterized finite set of line bundles that generate $D^b\Coh(X)$ in minimal time.   

\subsection{Results}

We now more carefully describe our setup and results. Although it would seem \emph{a priori} simpler to work entirely within the framework of toric varieties, many of our arguments are cleaner when using the language of toric stacks. Thus, from this point forward we will generalize to toric stacks in the sense of \cite{geraschenko2015toric}, and $\X$ will denote a toric stack over an algebraically closed field determined by a stacky fan $(\Sigma, \ul\beta)$. We will additionally assume that all toric stacks can be covered with smooth stacky charts. In particular, the toric stacks that we consider are smooth. See \cref{subsec:stackbackground} for some background on toric stacks and \cref{subsec:charts} for the definition of a smooth stacky chart.  We would like to point out that the toric stack associated with a smooth stacky chart is $[\AA^n/G]$ where $G$ is a finite group whose order does not divide the characteristic of the ground field $\kk$ and $n$ is the dimension of $\X$. To be clear, stacks will appear as a convenient language for discussing equivariant sheaves and resolutions in our proof, and we do not use any deep theory of stacks. For those so inclined, every instance of stack appearing in this section could be replaced by variety. For simplicity, all results stated in this section are over an algebraically closed field of characteristic zero. However, we prove versions of these results for algebraically closed fields of positive characteristic in the text.

On a toric stack $\X$ covered by smooth stacky charts, we define a finite set $\Theta$ of line bundles on $\X$, closely related to the toric Frobenius morphisms, in \cref{subsec:Thomsen}. We call $\Theta$ the \emph{Thomsen collection}. For an inclusion $i\colon \Y \hookrightarrow \X$ of a closed toric substack, we will abuse notation and write $\mathcal{O}_\Y$ for the pushforward $i_* \mathcal{O}_\Y$. Our main result is a finite resolution of $\mathcal O_\Y$ by elements of $\Theta$. 

\begin{mainthm}[Restatement of \cref{thm:resolutionOfSubstacks}] \label{mainthm:res}
	Let $\X$ be a toric stack that admits a cover by smooth stacky charts and let $\Y$ be a closed toric substack of $\X$. There exists an explicit resolution of $\mathcal{O}_\Y$ by a complex of direct sums of line bundles in the Thomsen collection on $\X$ with length equal to the codimension of $\Y$. 
\end{mainthm}

As far as the authors are aware, \cref{mainthm:res} produces new explicit resolutions even when $\X$ is projective space and $\Y$ is a projective toric variety embedded via a toric morphism.
The resolution in \cref{mainthm:res} is constructed geometrically. Namely, we produce this resolution from a stratification by subtori of a real torus whose dimension coincides with the codimension of $\Y$. When $\Y$ is a point, that is, the unit in $\TT$, the real torus is naturally identified with $M_\RR/M$ where $M$ is the character lattice of $\TT$. In that case, the stratification was introduced by Bondal in \cite{bondal2006derived} and is given simply by the subtori orthogonal to the rays of the fan (see also \cite[Section 5]{favero2022homotopy}). The proof of \cref{mainthm:res} is outlined in \cref{subsec:strategy} where the general description of the stratification is also given. 

\begin{rem} 
    The proof of \cref{mainthm:res} uses a few functorial properties of the constructed resolution. However, we expect that the resolution enjoys stronger functoriality that would be interesting to study. In particular, we expect that our assumption that $\X$ ``be covered by smooth stacky charts'' can be weakened to ``be covered by finite quotient stacks of toric varieties.'' In fact, previous work \cite{bruns2003divisorial, bruns2005conic, faber2019non} on the Thomsen collection $\Theta$ and its homological properties for general affine toric varieties, where $\Theta$ is the set of conic modules, may be helpful in such a generalization. Moreover, the shape of the resolution depends only on the choice of toric subvariety and the one-dimensional cones of $\Sigma$. 
\end{rem}

\begin{rem}
    When $\kk$ has positive characteristic, we are only able to prove \cref{mainthm:res} for subvarieties whose local models in charts have a nice behavior with respect to the characteristic (\cref{def:admissibleOverK}). The only section of the proof where these conditions are used is in \cref{subsubsec:characterPushforward} where we employ Schur orthogonality to identify the pushforward along a quotient by a finite abelian group $G$ with a direct sum of sheaves twisted by the characters of $G$.
\end{rem}

For describing more general sheaves on toric stacks, \cref{mainthm:res} is of particular interest when applied to the diagonal in $\X \times \X$.  

\begin{maincor} \label{maincor:diagonal}
    Let $\X$ be a toric stack covered by smooth stacky charts and let $\mathcal{O}_\Delta$ be the structure sheaf of the diagonal in $\X\times \X$. There is a resolution of $\mathcal O_\Delta$ by a complex of direct sums of sheaves in the Thomsen collection on $\X \times \X$ of length $\dim (\X)$.
\end{maincor}

\begin{rem} 
    The Thomsen collection on $\X\times \X$ consists of products of sheaves in the Thomsen collection $\Theta$ on $\X$, so all sheaves that appear in the resolution of \cref{maincor:diagonal} are of the form $\mathcal{F} \boxtimes \mathcal{F'}$ for $\mathcal F, \mathcal F' \in \Theta$. 
\end{rem}

For example when $X=\PP^n$, the Thomsen collection $\Theta = \{ \mathcal O(-n), \hdots, \mathcal O(-1), \mathcal O \}$ coincides with Beilinson's exceptional collection from \cite{beilinson1978coherent}.
Thus, \cref{maincor:diagonal} can be viewed as a generalization of Beilinson's resolution of the diagonal to all smooth toric varieties and stacks.
Previously, optimal resolutions of the diagonal generalizing Beilinson's resolution were obtained in \cite{bayer2001syzygies} for unimodular toric varieties\footnote{Here, unimodular is meant in a strong sense. The determinant of any collection of $\dim(X)$ generators of the fan must be $-1, 0$ or $1$.} and more recently in \cite{brown2022short} for smooth projective toric varieties of Picard rank $2$. 
Beilinson's resolution was also generalized to weighted projective spaces in \cite{canonaco2008derived}.
For an alternative generalization for general toric varieties, which is not known to admit a finite subcomplex in general, see \cite{brown2021tate}. 

In fact, a minimal resolution of the diagonal has applications to commutative algebra that motivated the constructions in \cite{brown2021tate,brown2022short}. In algebraic language, a resolution of a coherent sheaf by direct sums of line bundles on a toric variety translates to a virtual resolution of a $B$-saturated module over the Cox ring in the sense of \cite[Definition 1.1]{berkesch2020virtual} where $B$ is the irrelevant ideal. Further, Berkesch, Erman, and Smith suggested \cite[Question 6.5]{berkesch2020virtual} (see also \cite[Conjecture 1.1]{yang2021virtual} and \cite[Conjecture 1.2]{brown2022short}) that virtual resolutions of length at most $\dim(X)$ can be constructed for any such module over the Cox ring of a smooth toric variety $X$. By the argument presented in \cite[proof of Proposition 1.2]{berkesch2020virtual} and \cite[proof of Corollary 1.3]{brown2022short}, the existence of these virtual resolutions follows from \cref{maincor:diagonal} on any smooth projective toric variety.

\begin{maincor} \label{maincor:virtual}
    Let $X$ be a smooth projective toric variety with Cox ring $R$ and irrelevant ideal $B$. Every finitely-generated $\Pic(X)$-graded $B$-saturated $R$-module has a virtual resolution of length at most $\dim(X)$.
\end{maincor}

\begin{rem} The assumption that $X$ is projective in \cref{maincor:virtual} is used to appeal to the Fujita vanishing theorem and to apply a Fourier-Mukai transform, but it is plausible that this assumption can be weakened.
\end{rem}

In \cite{eisenbud2015tate} and \cite{berkesch2020virtual}, \cref{maincor:virtual} was proved for products of projective spaces. Other special cases of the Berkesch-Erman-Smith conjecture were established in \cite{yang2021virtual, brown2022short}. As our resolution of the diagonal is explicit, we expect that \cref{maincor:virtual} can be built upon to obtain further interesting results in computational algebraic geometry.

Now, we turn to applications to derived categories. An object $E \in D^b\Coh(\X)$ is a \emph{classical generator} if $D^b\Coh(\X)$ coincides with the smallest strictly full, saturated, and triangulated subcategory containing $E$. The following is an immediate consequence of \cref{maincor:diagonal}.

\begin{maincor} \label{maincor:generate}
    Let $\X$ be a toric stack covered by smooth stacky charts and let $\Theta$ be the Thomsen collection on $\X$. The sheaf
    $$ E = \bigoplus_{\mathcal{F} \in \Theta} \mathcal{F} $$
    is a classical generator of $D^b\Coh(\X)$.
\end{maincor}

\begin{rem} \label{rem:bondalgeneration}
    In \cite{bondal2006derived}, Bondal claimed without proof that \cref{maincor:generate} holds for smooth and proper toric varieties. For instance, \cref{maincor:generate} is referred to as Bondal's conjecture in \cite{uehara2014exceptional} (see, in particular, Remark 3.7) where it is proved for $3$-dimensional smooth toric Fano varieties. \cref{maincor:generate} was proved in \cite{ohkawa2013frobenius} in dimension $2$ and for smooth toric Fano fourfolds in \cite{prabhu2017tilting}. One of the methods presented in \cite{prabhu2017tilting} is analogous to the methods employed here, but on a clever \emph{ad hoc} case-by-case basis. We note that the smooth case follows from the smooth and proper one, as every toric variety has a toric compactification, and localization of derived categories sends classical generators to classical generators.

    While preparing this paper, we learned that Ballard, Duncan, and McFaddin have a different argument for proving \cref{maincor:generate}. In addition, an independent proof of \cref{maincor:generate} explicitly using homological mirror symmetry appeared in \cite[Corollary 3.5]{favero2023rouquier} while this paper was being completed.
\end{rem} 

\begin{rem} 
    As observed in \cite{bondal2006derived}, there is a class of toric varieties for which $\Theta$ forms a full strong exceptional collection of line bundles (see also \cite[Proposition 5.14]{favero2022homotopy}) given \cref{maincor:generate}. Also, there are instances where a subset of the Thomsen collection gives a full strong exceptional collection of line bundles. For example, such collections have been constructed in \cite{bernardi2010derived,costa2010frobenius,costa2012derived,dey2010derived,lason2011full} using \cref{maincor:generate} and in \cite{uehara2014exceptional, prabhu2017tilting} by proving of \cref{maincor:generate} in their setting. However, in general, smooth toric varieties do not admit full exceptional collections of line bundles \cite{hille2006counterexample} even under Fano assumptions \cite{efimov2014maximal}.
\end{rem} 

In \cite{uehara2014exceptional}, Uehara proposes studying the extension of \cref{maincor:generate} to any line bundle on a smooth and proper toric variety. More precisely, \cite[Conjecture 3.6]{uehara2014exceptional} asks if, for any line bundle $\mathcal{L}$ on a smooth and proper toric variety $X$, there is a positive integer $\ell$ such that the Frobenius pushforward of degree $\ell$ of $\mathcal L$ is a classical generator of $D^b \Coh(X)$. The answer is not always affirmative as demonstrated by $\mathcal{O}(-1)$ on $\PP^1$. In \cref{sec:frobgen}, we illustrate some general obstructions to this claim and conjecture that these are the only obstructions.

The fact that the resolution in \cref{maincor:diagonal} has length $\dim(\X)$ has additional consequences. In \cite{rouquier2008dimensions}, Rouquier introduced an interesting notion of dimension for triangulated categories which we will refer to as the \emph{Rouquier dimension} and denote by $\Rdim$. Roughly, the generation time of a classical generator $E$ is $t$ if every object can be obtained from $E$ by taking summands, shifts, finite direct sums, and at most $t$ cones. The Rouquier dimension is the minimum generation time over all classical generators. See \cref{subsec:Rouquierbackground} for more details. For any smooth quasi-projective scheme $Y$, Rouquier proved in \cite{rouquier2008dimensions} that
$$ \dim(Y) \leq \Rdim(D^b\Coh(Y)) \leq 2 \dim(Y) $$
and that the first inequality is equality when $Y$ is, in addition, affine. In \cite{orlov2009remarks}, after additionally verifying
\begin{equation} \label{eq:orlovconj}
    \Rdim(D^b\Coh(Y)) = \dim(Y)
\end{equation}
holds for all smooth quasi-projective curves, Orlov conjectures that \eqref{eq:orlovconj} holds for all smooth quasi-projective schemes. This conjecture has received considerable attention and been established in a variety of cases (see \cite{bai2021rouquier} for a thorough discussion of known cases).  However, the general result remains elusive. Restricting to the toric setting, a notable result is \cite{ballard2019toric} which verifies the conjecture when a certain subset of $\Theta$ gives rise to a tilting object and for all smooth toric Fano threefolds and fourfolds. 

When $\X$ is proper, the length of a resolution of the diagonal gives an upper bound on the generation time of the corresponding classical generator. Thus, we verify Orlov's conjecture for all toric stacks satisfying our covering condition.\footnote{Technically, \cite{rouquier2008dimensions} does not provide the lower bound for all toric stacks, but it follows from localization that $\dim(\X)=\Rdim(D^b\Coh(\TT)) \leq \Rdim(D^b\Coh(\X))$. See also \cite[Section 2.2]{ballard2012hochschild}.}

\begin{maincor} \label{maincor:rdim}
	If $\X$ is a toric stack covered by smooth stacky charts, $\Rdim(D^b\Coh(\X))=\dim(\X)$.
\end{maincor}

\begin{rem}
    An independent and concurrent proof of \cref{maincor:rdim} using homological mirror symmetry and ideas similar to those explained in \cref{subsec:HMSMotication} was obtained by Favero and Huang in \cite{favero2023rouquier}. We would also like to note that Ballard previously suggested that toric Frobenius should be useful for resolving Orlov's conjecture for toric varieties.
\end{rem}

\begin{rem}
    When $\X$ is not proper, we observe that it has a toric compactification $\bar \X$ for which \cref{maincor:rdim} follows from \cref{maincor:diagonal}.  We then obtain \cref{maincor:rdim} for $\X$ as $\Rdim$ does not increase under localization of derived categories.

    It follows from \cref{maincor:rdim} that $\Rdim(D^b\Coh(X)) = \dim(X)$ for any normal toric variety $X$. This is because the singularities of toric varieties are rather well-behaved. In particular, the bounded derived category of coherent sheaves on any toric variety is a localization of that on a smooth projective toric variety. Other singular examples where \eqref{eq:orlovconj} holds have been found in \cite[Corollary 3.3]{ballard2012hochschild} and \cite[Example 4.22]{bai2021rouquier}, but the smoothness assumption is necessary in general \cite[Example 2.17]{bai2021rouquier}.
\end{rem}

\subsection{Mirror symmetry inspiration}
\label{subsec:HMSMotication}

We now briefly explain how this work is inspired by the homological mirror symmetry (HMS) conjecture \cite{kontsevich1995homological}. 
However, we note that the constructions and proofs in the body of the paper are entirely algebro-geometric and nowhere directly involve HMS. 

Bondal's proposal \cite{bondal2006derived} has had a tremendous impact on understanding HMS for toric varieties. Fang, Liu, Treumann, and Zaslow \cite{fang2011categorification,fang2012t, fang2014coherent} demonstrated the relation of Bondal's stratification to HMS by forming a singular Lagrangian $\LL_{\Sigma, \beta} \in T^* T^n$ where $T^n$ in the real torus $M_\RR/M$ and relating sheaves on $T^n$ with microsupport in $\LL_{\Sigma, \beta}$ to coherent sheaves on $\X$. The most general case of HMS for toric stacks was proved using this approach by Kuwagaki \cite{kuwagaki2020nonequivariant} which after applying the main result of \cite{ganatra2018microlocal}, becomes a  quasi-equivalence
\begin{equation} \label{eq:torichms}
    D_{dg}^b\Coh(\X) \simeq \text{Tw}^\pi \mathcal{W}(T^*T^n, \LL_{\Sigma, \beta})
\end{equation}
where $D_{dg}^b\Coh(\X)$ is a dg-enhancement of $D^b\Coh(\X)$ and $\text{Tw}^\pi \mathcal{W}(T^*T^n, \LL_{\Sigma, \beta})$ is the idempotent completion of the pre-triangulated closure of the partially wrapped Fukaya category of $T^*T^n$ stopped at $\LL_{\Sigma,\beta}$. 

The resolutions in \cref{mainthm:res} are motivated by a construction of resolutions in $\mathcal{W}(T^*T^n, \LL_{\Sigma, \beta})$. Namely, the zero section $T^n \subset T^*T^n$ is an object of $\mathcal{W}(T^*T^n, \LL_{\Sigma, \beta})$. Given a Morse function, well-behaved with respect to the stratification on $T^n$, one can produce a Lagrangian cobordism between $T^n$ and a disjoint union of Lagrangians which are isotopic to cotangent fibers in $T^*T^n$ and correspond to line bundles under \eqref{eq:torichms}. By \cite[Proposition 1.37]{ganatra2018sectorial}, this Lagrangian cobordism induces an isomorphism in the Fukaya category. This procedure can be generalized to apply to conormals of subtori mirror to toric subvarieties. Moreover, the relevant collection of line bundles can be identified with $\Theta$ using \cite{abouzaid2009morse,hanlon2022aspects}. In fact, it is possible to prove \cref{maincor:generate} by itself using variations on \cite[Proposition 5.22]{ganatra2018microlocal} and \cite[Lemma 3.14]{hanlon2022aspects}. Of course, there are various technical aspects of partially wrapped Fukaya categories that need to be wrangled with to make these arguments precise.
Even then, the output would be a quasi-isomorphism in the wrapped Fukaya category or mirror derived category of coherent sheaves with a complex with unspecified morphisms rather than an explicit resolution as constructed in \cref{mainthm:res}.
Our work \cite{hanlon2023rouquier}, which appeared after this article, carries out the strategy outlined above to relate Rouquier dimension of wrapped Fukaya categories to topological invariants.

\subsection{Strategy of Proof}
\begin{wrapfigure}{r}{0.5\textwidth}
    \centering
      \scalebox{.5 }{\begin{tikzpicture}[scale=2]
\newcommand{\OO}{\mathcal O}
\usetikzlibrary{calc, decorations.pathreplacing,shapes.misc}
\usetikzlibrary{decorations.pathmorphing}

\tikzstyle{fuzz}=[red, 
    postaction={draw, decorate, decoration={border, amplitude=0.15cm,angle=90 ,segment length=.15cm}},
]

\draw  (-1.5,3) rectangle (3.5,-2);
\draw[fuzz](1,3) node (v7) {} -- (1,-2);
\draw[fuzz](-1.5,0.5) -- (3.5,0.5) node (v6) {};
\draw[fuzz] (3.5,-2) node (v10) {} -- (-1.5,3);
\node[fill=white, fill opacity=50] (v2) at (1,0.5) {$\OO(0)$};
\node[fill=white, fill opacity=50] (v11) at (-1.5,3) {$\OO(-1)$};
\node[fill=white, fill opacity=50] (v1) at (-1.5,0.5) {$\OO(-1)$};
\node[fill=white, fill opacity=50] (v4) at (1,-2) {$\OO(-1)$};
\node[fill=white, fill opacity=50] (v3) at (-0.5,-1) {$\OO(-1)$};
\node[fill=white, fill opacity=50] (v8) at (2.5,2) {$\OO(-2)$};
\draw  (v1) edge (v2);
\node (v9) at (3.5,3) {};
\node (v5) at (-1.5,-2) {};

\draw[thick,  ->]  (v3) edge  node[midway, fill=white, fill opacity=50]{1} (v1);
\draw[thick,  ->]  (v3) edge  node[midway, fill=white, fill opacity=50]{1} (v4);
\draw[thick,  ->]  (v3) edge  node[midway, fill=white, fill opacity=50]{1} (v5);
\draw[thick,  ->]  (v1) edge  node[midway, fill=white, fill opacity=50]{$-x_2$}(v2);
\draw[thick,  ->]  (v6) edge  node[midway, fill=white, fill opacity=50]{$x_0$}(v2);
\draw[thick,  ->]  (v4) edge  node[midway, fill=white, fill opacity=50]{$x_2$}(v2);
\draw[thick,  ->]  (v7) edge  node[midway, fill=white, fill opacity=50]{$-x_1$}(v2);
\draw[thick,  ->]  (v8) edge  node[midway, fill=white, fill opacity=50]{$-x_0$}(v7);
\draw[thick,  ->]  (v8) edge  node[midway, fill=white, fill opacity=50]{$-x_1$}(v6);
\draw[thick,  ->]  (v8) edge  node[midway, fill=white, fill opacity=50]{$-x_2$}(v9);
\draw[thick,  ->]  (v10) edge node[midway, fill=white, fill opacity=50]{$-x_0$}(v2);
\draw[thick,  ->]  (v11) edge node[midway, fill=white, fill opacity=50]{$x_1$}(v2);
\end{tikzpicture} }
    \caption{The complex of sheaves associated to $\phi\colon e\to \PP^2$. The stratification $S^\phi$ is a stratification of the real 2-torus and indicated by the dashed red lines}
    \label{fig:introFig}
  \end{wrapfigure}
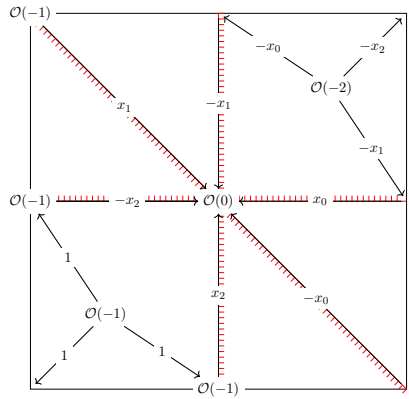
The proof of \cref{mainthm:res} breaks into two portions. First, we prove the special case where we resolve the skyscraper sheaf of the identity $e$ in a toric variety $X$ in \cref{thm:resolutionOfPoints}. The second part of the proof is to bootstrap this result to the resolution of toric subvarieties $Y \subset X$. To pass from the first to the second step, we will need to enhance the first step to \href{thm:resolutionOfPointsInStacks}{Theorem $3.1'$} ---``the resolution of identity in a toric \emph{stack}.'' The second step then naturally generalizes to the statement of 
\cref{mainthm:res} --- ``resolution of toric substacks inside a toric stack.''

To resolve a point $\phi\colon e\to  X$, we first construct an (a priori non-exact) complex of sheaves $C_\bullet(S^\phi, \mathcal O^\phi)$ on $X$ for every such $\phi$.  To each morphism $\phi\colon e\to X$, we associate a stratified space $S^\phi$ equipped with a Morse function and a sheaf $\mathcal O^\phi$ with values in line bundles on $X$. The complex $C_{\bullet}(S^\phi, \mathcal O^\phi)$ is the discrete Morse complex associated with this data. See \cref{fig:introFig} for an example of a stratified space equipped with a sheaf taking values in line bundles on $\PP^2$, whose associated Morse complex computes a resolution of $\phi\colon e\to \PP^2$.
We then prove that these complexes satisfy the following properties.
\begin{description}
    \item[Koszul resolution of points] When $\phi \colon e\to \AA^n$, the complex  $C_\bullet(S^\phi, \mathcal O^\phi)$ is the Koszul resolution of $\phi_*\mathcal O_e$. This follows for checking $\phi\colon e\to \AA^1$ by hand, and proving a K\"unneth-type formula (\cref{lem:kunneth}).
    \item[Pullback up to homotopy] Let $D_\rho$ be a toric divisor and $X_\rho=X\setminus D_\rho$ giving us an open inclusion $i \colon X_\rho\to X$. Let $\phi_\rho \colon e\to X_\rho$ be the inclusion of the identity. \cref{lem:functorialityRestriction} shows that $i^*C_\bullet(S^\phi, \mathcal O^\phi)$ and $C_\bullet(S^{\phi_\rho}, \mathcal O^{\phi_\rho})$ are homotopic. This is done by showing that $S^{\phi_\rho}$ and $S^{\phi}$ are two different stratifications of the same space with the same sheaf so their associated Morse complexes are homotopic.
\end{description}
Using these two properties and that $X$ is covered by smooth charts, we show that $C_\bullet(S^\phi, \mathcal O^\phi)$ restricts to a resolution of $\phi_*\mathcal O_e$ on every toric chart. This proves that the complex is a resolution of $\mathcal O_e$.  

We take a similar approach to resolving subvarieties. To every toric subvariety $\phi\colon Y\to X$ we associate a complex of sheaves $C_\bullet(S^\phi, \mathcal O^\phi)$ using a discrete Morse function. This complex also satisfies the pullback up to homotopy property. On the next step, we arrive at an impasse: on the toric charts, we do not have local models for resolutions of $\phi_*\mathcal O_Y|_{\AA^n}$. To get around this we reduce to the previous setting. For example, to resolve a line $\AA$ in $\AA^2$, one can first resolve a point in $\PP^1$, pull that back to a resolution of a line in $\AA^2\setminus\{0\}$, and then push that forward to a resolution on $\AA^2$. This example is examined in more detail in \cref{ex:resolutionOfLine}.

Generally, we will remove from $\AA^n$ some toric boundary components to obtain $X^\circ\subset \AA^n$ an open toric subvariety. This open subvariety will have the property that $Y\cap X_{\Sigma^\circ}$ is an algebraic torus $\TT_Y$. Then we attempt to take a quotient $X^\circ/\TT_Y$ and resolve a point there by our previous strategy. For this to work, we need to justify that the pullback-pushforward (by $\pi^*\circ i_*$ in \cref{fig:quotientStrategy}) of the resolution of a point $C_\bullet(S^{\phi_e}, \mathcal O^{\phi_e})$ gives us a resolution of $\phi_*\mathcal O_{\Y}$ and that this resolution agrees with $C_\bullet(S^\phi, \mathcal O^\phi)$. 
\begin{figure}[h]\centering
		\begin{tikzcd}
			e \arrow{d}{\phi_e}  & \mathbb{T}_Y  \arrow{l}  \arrow{d}{\phi^\circ}  \arrow{r} & Y \arrow{d}{\phi} \\
			\; [X^\circ/\mathbb{T}_Y] 	\arrow[bend right]{rrr}{i} &  X^\circ \arrow{r}{I} 	\arrow{l}{\pi^\circ}		  & \mathbb{A}^n \arrow{r}{\pi} & \;[\mathbb{A}^n/\mathbb{T}_Y]
		\end{tikzcd}
        \caption{Resolutions of toric subvarieties of $\AA^n$ are obtained by resolving a point on an appropriately chosen quotient stack.}
        \label{fig:quotientStrategy}
\end{figure}
To capture some of the geometric intuition from \cref{ex:resolutionOfLine}, we work with toric stacks. The second part of the argument can be broken into the steps:
\begin{description}
    \item[Points in stacks] Prove \href{thm:resolutionOfPointsInStacks}{Theorem $3.1'$}, which resolves $\phi\colon e\to \X$ for toric stacks $\X$ which are covered by smooth stacky charts. The main difference to the non-stacky version is in constructing the local resolutions: the case of $\phi \colon  e\to [\AA^n/G] $, where $G$ is a finite subgroup of the torus. This follows from \cref{lem:pushforwardquotient}, which shows that our resolution is functorial under some quotients by finite groups.
    \item[Checking \cref{fig:quotientStrategy}] The functor $i_*$ does not send line bundles to line bundles. We define a functor of sheaves $i_\flat:\Sh([X^\circ/\mathbb{T}_Y])\to \Sh([X/\mathbb{T}_Y])$ so that the composition $\pi^*\circ i_\flat$ is exact and sends line bundles to line bundles. Then we show that $C_\bullet(S^\phi, \mathcal O^\phi)$ is functorial under the composition $\pi^*\circ i_\flat$. This is proven in \cref{lem:pullbackquotient,lemma:functorialityCodim2}. 
\end{description}
The proofs of \cref{lem:pushforwardquotient,lem:pullbackquotient,lemma:functorialityCodim2} rely on the machinery for toric stacks developed in \cref{sec:backgroundB}. 

From an expositional viewpoint, it would be desirable to treat the cases of resolving points/ subvarieties/ substacks inside varieties/ stacks separately. However, we have taken the more economical route of defining $C_\bullet(S^\phi, \mathcal O^\phi)$ for the most general case (resolutions for substacks of stacks) for which the other cases are specializations. We suggest that the reader first looks through the construction for resolving points inside toric varieties before moving on to the other cases. For the case of $\phi:e\to X$, one can read the following sections, replacing everywhere the words ``toric substack'' with ``$e$'' and ``toric stack'' with ``smooth toric variety''.

\begin{itemize}
    \item \cref{subsec:stackbackground,subsec:Thomsen} which discuss notation for line bundles on toric varieties and the Thomsen collection.
    \item \cref{subsec:strategy,subsec:pathsandsheaves,subsec:defOfCbullet,ex:pointInA1,ex:pointInP2}, which cover the overall strategy in greater detail, define the complex $C_\bullet(S^\phi, \mathcal O^\phi)$, and provide some examples.
    \item \cref{pf:kunneth,pf:restrictionHomotopy} which prove the K\"unneth formula and pullback up to homotopy property of the complex.
\end{itemize}

\subsection{Outline}
The remainder of the paper is organized as follows. 
\cref{sec:backgroundB} contains some relevant background on line bundles on toric stacks and a plethora of examples motivating the proof of \cref{mainthm:res}. 
In particular, we give a thorough introduction to toric stacks including a discussion of the Thomsen collection. 

In \cref{sec:Resolutionideas}, we present the general outline of the proof of \cref{mainthm:res}. In \cref{subsec:strategy,subsec:pointsinstacks,subsec:stacksinstacks}, we demarcate the steps for proving \cref{mainthm:res} working in increasing generality on the cases of points in toric varieties, points in toric stacks, and substacks of toric stacks. We discuss stratifications of real tori in \cref{subsec:pathsandsheaves}, which leads to the definition of $C_\bullet(S^\phi, \mathcal O^\phi)$ in \cref{subsec:defOfCbullet}.

We then prove the lemmas constituting the proof of \cref{mainthm:res} in \cref{sec:Resolutionproofs}. \cref{sec:frobgen} proposes a generalization of \cref{maincor:generate} inspired by \cite[Conjecture 3.6]{uehara2014exceptional}. 

Finally, \cref{sec:quivers} contains some facts about discrete Morse theory that we use in our proofs. We also include a short discussion of quotient stacks in \cref{subsec:stackBackground} followed by some background information on generation in triangulated categories and Rouquier dimension in \cref{subsec:Rouquierbackground}. 

\subsection{Acknowledgements}
The authors thank Mohammed Abouzaid, Laurent C\^{o}t\'{e}, David Favero, Sheel Ganatra, Jesse Huang, Qingyuan Jiang, Mahrud Sayrafi, Nick Sheridan, and Abigail Ward for useful discussions. 
We are additionally grateful to Michael K. Brown and Daniel Erman for connecting our work to virtual resolutions and sharing their insights with us.
We also thank the anonymous referee for several improvements and for correcting various minor errors.
This project was born out of discussions at the workshop ``Recent developments in Lagrangian Floer theory" at the Simons Center for Geometry and Physics, Stony Brook University, and we thank the Simons Center and the workshop organizers for a stimulating scientific environment.

AH was supported by  NSF RTG grant DMS-1547145 and by the Simons Foundation (Grant Number
814268 via the Mathematical Sciences Research Institute, MSRI). JH was supported by EPSRC Grant EP/V049097/1. 
\section{Line bundles on toric stacks} \label{sec:backgroundB}
This section serves three purposes.
\begin{itemize}
	\item \cref{subsec:stackbackground}: Provide some background on line bundles on toric varieties/stacks from the perspective of support functions.
	\item \cref{subsec:Thomsen}: Introduce the Thomsen collection for toric varieties/stacks. This is a collection of line bundles on $X$ indexed by the strata of a stratification of the torus $M_\RR/M$.
	\item Develop the tools needed to prove \cref{thm:resolutionOfSubstacks}. In \cref{subsec:finiteQuotient}, we describe how line bundles pushforward under quotients. \Cref{subsec:charts} defines a toric stacky chart. In \cref{subsec:equivCodim2}, we examine how to interpret the codimension of the complement of an open immersion between toric stacks, and prove that when this ``equivariant codimension'' is greater than 1, the pushforward map along the inclusion enjoys special properties.
\end{itemize}

Additionally, this section includes many examples of resolutions of toric substacks of toric stacks of dimension 1 or 2, mostly designed to illustrate the arguments in \cref{thm:resolutionOfSubstacks} which allow us to pass from resolutions of points in toric stacks to resolutions of toric substacks. These examples are presented in increasing difficulty as follows.
\begin{itemize}
	\item \cref{ex:resolutionOfLine}: Resolving the line $z_0=z_1\subset \AA^2$.
	\item \cref{ex:nonequivParaResolution}: A non-example: resolving the parabola $z_0^2=z_1\subset \AA^2$ nonequivariantly.
			\item \cref{ex:pointInOrbifoldLine}: Resolving a point in an orbifold line. 
	\item \cref{ex:pointInWeightedP1}: Resolving a point in $\PP(2, 1)$. 
	\item \cref{ex:pointInNonSeparatedLine}: Resolving a point in the non-separated line.
	\item \cref{ex:equivariantResolutionOfParabola}: Resolving the parabola $\{ z_0^2=z_1 \} \subset \AA^2$ equivariantly.
	\item \cref{ex:resolutionOfHyperbola}: Resolving the hyperbola $\{ z_0z_1=1 \}\subset \AA^2$.
\end{itemize}
\subsection{Some resolutions of toric subvarieties in \texorpdfstring{$\AA^2$}{the plane}}
We work over an algebraically  closed field $\kk$ with characteristic $p\geq 0$. We start with some motivating examples and non-examples.  We denote by $\TT^k$ the (split) algebraic torus of dimension $k$. The examples are chosen to highlight the strategy of \cref{thm:resolutionOfSubstacks} of reducing resolutions of toric subvarieties to resolutions of points in a quotient stack.
\begin{example}[Resolving $z_0=z_1$]
	\label{ex:resolutionOfLine}
	Take the line $\phi \colon \AA^1\to \AA^2 $ parameterized by $z\mapsto (z, z)$. Let $i \colon \AA^2\setminus \{(0,0)\}\to \AA^2$ be the inclusion and let $\TT^1= \AA^1 \setminus \{0\}$. The relevant maps of fans are drawn on the left-hand side of \cref{fig:resolutionOfLine}. 
	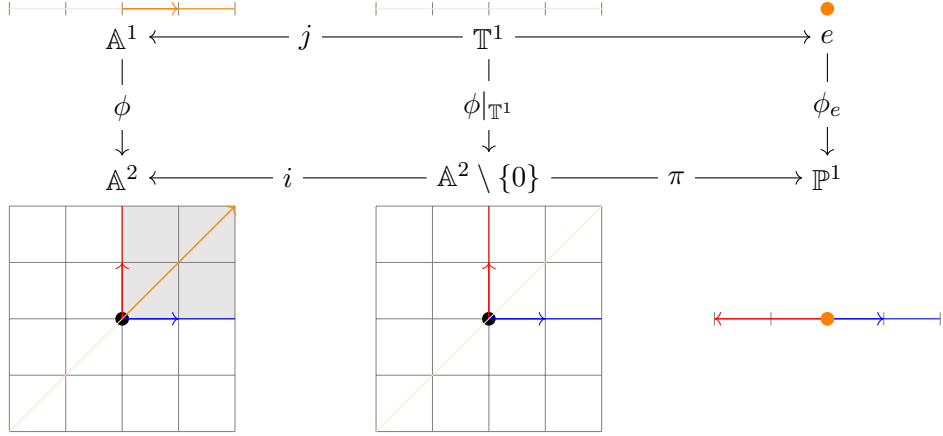
\begin{figure}
		\centering
		\usetikzlibrary{decorations.pathreplacing}
\begin{tikzpicture}[decoration=ticks, segment length=.75cm, scale=.75]

			\begin{scope}[]
				\begin{scope}[]
					\fill[gray!20]  (0,0) rectangle (2,2);
					\draw [help lines, step=1cm] (-2,-2) grid (2,2);
					\node[circle, fill=black, scale=.5] at (0,0) {};
					\draw[blue] (0,0) edge[->] (1,0) edge (2, 0);
					\draw[red] (0,0) edge[->] (0, 1) edge (0, 2);
					\draw[orange!20](-2,-2) edge (2,2);
					\draw[orange] (0,0) edge[->] (2,2) edge (2,2);
				\end{scope}
				\begin{scope}[shift={(0,5.5)}]
					\draw[decorate, help lines](-2,0)-- (2.01,0);
					\draw[orange!20](-2,0) -- (2,0);
					\draw[orange] (0,0) edge[->] (1,0) edge (2,0);
				\end{scope}
			\end{scope}

			\begin{scope}[shift={(6.5,0)}]
				\begin{scope}[]
					\draw [help lines, step=1cm] (-2,-2) grid (2,2);
					\node[circle, fill=black, scale=.5] at (0,0) {};
					\draw[blue] (0,0) edge[->] (1,0) edge (2,0);
					\draw[red] (0,0) edge[->] (0,1) edge (0,2);
					\draw[orange!20](-2,-2) edge (2,2);
				\end{scope}
				\begin{scope}[shift={(0,5.5)}]
					\draw[decorate, help lines](-2,0)-- (2.01,0);
					\draw[orange!20](-2,0) -- (2,0);
				\end{scope}
			\end{scope}

			\begin{scope}[shift={(12.5,0)}]
				\draw[decorate, help lines](-2,0)-- (2.01,0);
				\draw[gray!20](-2,0) -- (2,0);
				\draw[blue] (0,0) edge[->] (1,0) edge (2, 0);
				\draw[red] (0,0) edge[->] (-2,0) edge (-2,0);
				\node[fill=orange, circle, scale=.5] at (0,0) {};
			\end{scope}

			\node[fill=orange, circle, scale=.5] at (12.5,5.5) {};
			\node (v1) at (0,5) {$\mathbb A^1$};
			\node (v2) at (0,2.5) {$\mathbb A^2$};
			\node (v3) at (6.5,5) {$\mathbb T^1$};
			\node (v4) at (6.5,2.5) {$\mathbb A^2\setminus \{0\}$};
			\node (v5) at (12.5,5) {$e$};
			\node (v6) at (12.5,2.5) {$\mathbb P^1$};
			\draw  (v1) edge[->] node[fill=white] {$\phi$} (v2);
			\draw  (v1) edge[<-] node[fill=white] {$j$} (v3);
			\draw  (v2) edge[<-] node[fill=white] {$i$} (v4);
			\draw  (v3) edge[->] node[fill=white] {$\phi|_{\mathbb T^1}$} (v4);
			\draw  (v3) edge[->]  (v5);
			\draw  (v4) edge[->] node[fill=white] {$\pi$}(v6);
			\draw  (v5) edge[->] node[fill=white] {$\phi_e$}(v6);
		\end{tikzpicture} 		\caption{Reducing the resolution of a line in $\AA^2$ to the resolution of a point in $\PP^1$. Compare to \cref{fig:quotientStrategy}.}
		\label{fig:resolutionOfLine}
	\end{figure}
	We have a projection $\pi \colon \AA^2\setminus \{(0,0)\}\to \PP^1 $. We give $\PP^1$ the coordinates $[x_0: x_1]$. Let $e$ denote the toric variety which is a single point, and let $\phi_e \colon  e\to \PP^1 $ be the inclusion whose image is $[1:1]\in \PP^1$, that is, the origin of the algebraic torus. This inclusion is presented on the right-hand side of \cref{fig:resolutionOfLine}. We have that $\pi^{-1}(e)=\TT^1$.

	Now, something a bit unusual happens. Let $C_\bullet(S^{\phi_e}, \mathcal O^{\phi_e})$ be the resolution of $(\phi_*)\mathcal O_e$ on $\PP^2$ given by the complex
	$\mathcal O_\PP(-1)\xrightarrow{x_0-x_1}\mathcal O_{\PP}\dashrightarrow\mathcal \phi'_*O_e$. This resolution can be naturally ``drawn'' on the real torus $M_\RR/M$ by
	\begin{equation}
		\begin{tikzpicture}

			\draw (-2.5,-0.5) -- (1.5,-0.5);
			\node (v1) at (-0.5,-2) {$\mathcal O_{\mathbb P}(-1)$};
			\node (v2) at (-2.5,-1) {$\mathcal O_{\mathbb P}$};
			\node (v3) at (1.5,-1) {};
			\draw[->]  (v1) edge node[red,midway, fill=white]{$x_0$} (v2);
			\draw[->]  (v1) edge node[blue,midway, fill=white]{$-x_1$}(v3);
			\draw[thick, red] (-2.5,-0.5) -- (-2.25,-0.5);
			\draw[thick, blue] (-2.5,-0.5) -- (-2.75,-0.5);
			\node[circle, fill=black, scale=.5] at (-2.5,-0.5) {};
\end{tikzpicture}
		\label{eq:resolutionPtInP1}
	\end{equation}
	This resolution of the point comes from \cref{thm:resolutionOfPoints}.
	To see that this is indeed the desired resolution, we can restrict to smooth charts by specializing to $x_0=1$ or $x_1=1$ where it becomes the Koszul resolution of the point $1 \in \AA^1$. See also \cref{ex:pointInA1}.
	
	Now  pull \cref{eq:resolutionPtInP1} along $\pi$, and push forward along $i$. This need not be an exact resolution (as $i_*$ is not exact) of $\mathcal O_{\AA_1}$ (as $i_*(\phi_{\TT^1})\mathcal O_{\TT^1}$ is not $\phi_*\mathcal O_{\AA^1}$) by line bundles (as pushforward, in general, doesn't send line bundles to line bundles). A priori, we do not even know if $i_*\pi^*C_\bullet(S^{\phi_e}, \mathcal O^{\phi_e})$ is a complex of coherent sheaves (as $i$ is not proper). Disregarding all these red flags, one can compute that $i_*\pi^*C_\bullet(S^{\phi_e}, \mathcal O^{\phi_e})$ is the following complex
	\[
		\begin{tikzpicture}
			\draw (-2.5,-0.5) -- (1.5,-0.5);
			\node (v1) at (-0.5,-2) {$\mathcal O_{\AA^2}$};
			\node (v2) at (-2.5,-1) {$\mathcal O_{\AA^2}$};
			\node (v3) at (1.5,-1) {};
			\draw[->]  (v1) edge node[red,midway, fill=white]{$z_0$} (v2);
			\draw[->]  (v1) edge node[blue,midway, fill=white]{$-z_1$}(v3);
			\draw[thick, red] (-2.5,-0.5) -- (-2.25,-0.5);
			\draw[thick, blue] (-2.5,-0.5) -- (-2.75,-0.5);
			\node[circle, fill=black, scale=.5] at (-2.5,-0.5) {};
		\end{tikzpicture}
	\]
	which matches the usual resolution of $\phi_*\mathcal O_\AA$ as the structure sheaf of a hypersurface in $\AA^2$. We would like to highlight that two unusual things are happening here. Namely, the pushforward $i_*$ is not exact, but it happens to preserve exactness of our resolution and resolve the desired sheaf.  Additionally, since the inclusion $i$ has complement of codimension two, $i_*$ sends line bundles to line bundles (Hartog's principle). In general, $i_*$ will not satisfy the second property, and so we will have to address this issue.
\end{example}
We now provide an example that looks similar to the one before in that we reduce resolving a subvariety to resolving a point on a quotient, but it is \emph{not} a special case of our main theorem.
\begin{example}[Resolving $z_0^2=z_1$]
	\label{ex:nonequivParaResolution}
	We now describe a resolution of the parabola, which \emph{does not quite} fit into our proof strategy.
	Take the parabola $\phi \colon \AA\into \AA^2 $ parameterized by $z\mapsto (z, z^2)$. To construct the map $\pi$, we observe that $\TT^1 = \AA\setminus\{0\}\into \AA^2$ is the inclusion of a subgroup of the algebraic torus. We therefore consider the inclusion $i\colon \AA^2\setminus \{0\}\to \AA^2$. The maps of fans are drawn in \cref{fig:resolutionOfParabola}.
	\begin{figure}
		\centering
		\begin{tikzpicture}[scale=.75, decoration=ticks, segment length=.75cm]

    \begin{scope}[]
        \begin{scope}[]
            \fill[gray!20]  (0,0) rectangle (2,2);
            \draw [help lines, step=1cm] (-2,-2) grid (2,2);
            \node[circle, fill=black, scale=.5] at (0,0) {};
            \draw[blue] (0,0) edge[->] (1,0) edge (2, 0);
            \draw[red] (0,0) edge[->] (0, 1) edge (0, 2);
            \draw[orange!20](-2,-1) edge (2,1);
            \draw[orange] (0,0) edge[->] (2,1) edge (2,1);
        \end{scope}
        \begin{scope}[shift={(0,5.5)}]
            \draw[decorate, help lines](-2,0)-- (2.01,0);
            \draw[orange!20](-2,0) -- (2,0);
            \draw[orange] (0,0) edge[->] (1,0) edge (2,0);
        \end{scope}
    \end{scope}

    \begin{scope}[shift={(6.5,0)}]
        \begin{scope}[]
            \draw [help lines, step=1cm] (-2,-2) grid (2,2);
            \node[circle, fill=black, scale=.5] at (0,0) {};
            \draw[blue] (0,0) edge[->] (1,0) edge (2,0);
            \draw[red] (0,0) edge[->] (0,1) edge (0,2);
            \draw[orange!20](-2,-1) edge (2,1);
        \end{scope}
        \begin{scope}[shift={(0,5.5)}]
            \draw[decorate, help lines](-2,0)-- (2.01,0);
            \draw[orange!20](-2,0) -- (2,0);
        \end{scope}
    \end{scope}

    \begin{scope}[shift={(12.5,0)}]
        \draw[decorate, help lines](-2,0)-- (2.01,0);
        \draw[gray!20](-2,0) -- (2,0);
        \draw[blue] (0,0) edge[->] (1,0) edge (2, 0);
        \draw[red] (0,0) edge[->] (-2,0) edge (-2,0);
        \node[fill=orange, circle, scale=.5] at (0,0) {};
    \end{scope}

    \node[fill=orange, circle, scale=.5] at (12.5,5.5) {};
    \node (v1) at (0,5) {$\mathbb A^1$};
    \node (v2) at (0,2.5) {$\mathbb A^2$};
    \node (v3) at (6.5,5) {$\mathbb T^1$};
    \node (v4) at (6.5,2.5) {$\mathbb A^2\setminus \{0\}$};
    \node (v5) at (12.5,5) {$e$};
    \node (v6) at (12.5,2.5) {$\mathbb P(1,2)$};
    \draw  (v1) edge[->] node[fill=white] {$\phi$} (v2);
    \draw  (v1) edge[<-] node[fill=white] {$j$} (v3);
    \draw  (v2) edge[<-] node[fill=white] {$i$} (v4);
    \draw  (v3) edge[->] node[fill=white] {$\phi|_{\mathbb T^1}$} (v4);
    \draw  (v3) edge[->]  (v5);
    \draw  (v4) edge[->] node[fill=white] {$\pi$}(v6);
    \draw  (v5) edge[->] node[fill=white] {$\phi_e$}(v6); 
\end{tikzpicture} 		\caption{Reducing the resolution of a parabola to the resolution of a point}
		\label{fig:resolutionOfParabola}
	\end{figure}
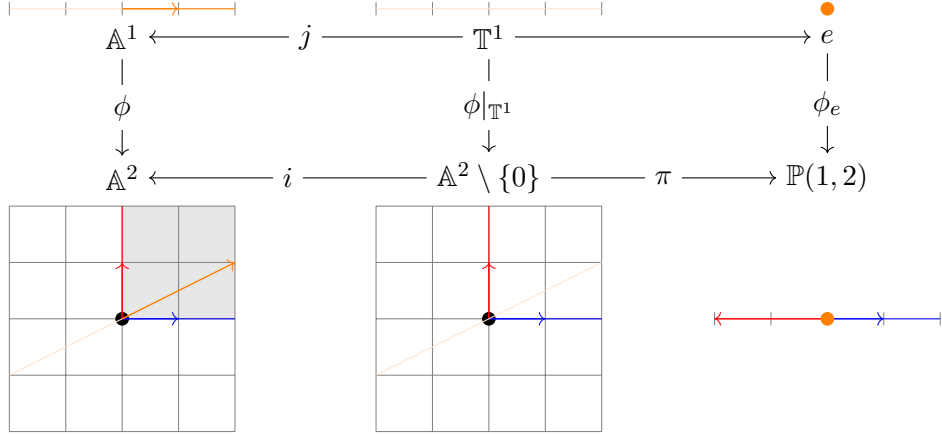
	We now mimic the previous example.
	There is a projection
	\begin{align*}
		\pi \colon \AA^2\setminus \{0\}\to  \PP(1, 2) &  &
		(z_0, z_1)\mapsto [z_0: z_1]_{(1,2)}.
	\end{align*}
	Then, the identity in the algebraic torus $e=[1:1]\in \PP(1, 2)$ satisfies the property that $\pi^{-1}(e)=\TT^1$.
	We now must produce a resolution of $(\phi_e)_*\mathcal O_e$. This is \emph{not} one of the examples covered in \cref{thm:resolutionOfPoints}, which cover resolving points in \emph{smooth} toric varieties. For this particular example, we happen to know a resolution $C_\bullet$ of $e$ given by
	\[
		\begin{tikzpicture}			
			\node (v1) at (0,-2) {$\mathcal O_{\PP(1, 2)}$};
			\node (v2) at (-3,-1) {$\mathcal O_{\PP(1,2)}$};
			\node (v3) at (3,-1) {};
			\draw[->]  (v1) edge node[red,midway, fill=white]{$x_0^2$} (v2);
			\draw[->]  (v1) edge node[blue,midway, fill=white]{$-x_1$}(v3);
		\end{tikzpicture}
	\]
	As before, disregard all warnings and compute $i_*\pi^* C_\bullet$ to obtain
	\[
		\begin{tikzpicture}			
			\node (v1) at (0,-2)  {$\mathcal O_{\AA^2}$};
			\node (v2) at (-3,-1) {$\mathcal O_{\AA^2}$};
			\node (v3) at (3,-1) {};
			\draw[->]  (v1) edge node[red,midway, fill=white]{$z_0^2$} (v2);
			\draw[->]  (v1) edge node[blue,midway, fill=white]{$-z_1$}(v3);

		\end{tikzpicture}
	\]
	the usual resolution of $\phi_*\mathcal O_\AA$.
\end{example}
In both examples, the first step --- producing the map $\pi$ --- is clear. We restrict the subvariety to the open torus which is a subgroup of the larger torus. Then, we study a point in the quotient by that subgroup. However, the next step --- producing a resolution of that point --- can be difficult in general as both the quotient space and morphism $\pi$ may not be particularly nice. In particular, the strategy of restricting to smooth charts in \cref{fig:resolutionOfLine} cannot be immediately applied to \cref{ex:nonequivParaResolution}.

Our approach is to instead take the stack quotient. The stack $[\AA^2\setminus \{0\}/\TT^1]$, where $t \in \TT^1$ acts by $t\cdot (z_1, z_2)=(tz_1, t^2z_2)$, is a smooth stack and can be covered with smooth stacky charts. The resolution of a point on this toric stack is covered by \href{thm:resolutionOfPointsInStacks}{Theorem $3.1'$}, which we work out in \cref{ex:equivariantResolutionOfParabola}. When taking this approach, we obtain the following resolution for the parabola:
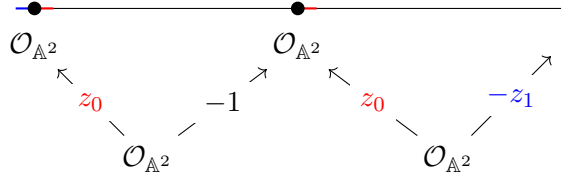
\[\begin{tikzpicture}
	\draw (-3,-0.5) -- (4,-0.5);
	\node (v1) at (-1.5,-2.5) {$\mathcal O_{\AA^2}$};
	\node (v2) at (-3,-1) {$\mathcal O_{\AA^2}$};
	\node (v3) at (0.5,-1) {$\mathcal O_{\AA^2}$};
	\node (v5) at (4,-1) {};
	\draw[->]  (v1) edge node[red,midway, fill=white]{$z_0$} (v2);
	\draw[->]  (v1) edge node[midway, fill=white]{$-1$}(v3);
	\draw[thick, red] (0.5,-0.5) -- (0.75,-0.5);
	\node[circle, fill=black, scale=.5] at (0.5,-0.5) {};
	\draw[thick, red] (-3,-0.5) -- (-2.75,-0.5);
	\draw[thick, blue] (-3,-0.5) -- (-3.25,-0.5);
	\node[circle, fill=black, scale=.5] at (-3,-0.5) {};
	\node (v4) at (2.5,-2.5) {$\mathcal O_{\AA^2}$};
	\draw[->]  (v4) edge  node[red, midway, fill=white]{$z_0$}(v3);
	\draw[->]  (v4) edge  node[blue,midway, fill=white]{$-z_1$}(v5);
\end{tikzpicture}
\]
which is easily seen to be homotopy equivalent to the one produced in \cref{ex:nonequivParaResolution}. This new resolution also has the nice property that all the structure coefficients of the differential are given by primitive monomials, which was not the case with the previous resolution. 
\begin{rem}
	As noted above, we choose to construct resolutions on toric stacks which are covered by smooth stacky charts because we can produce a nice description of the resolution of a point in a smooth stacky chart. However, as \cref{ex:nonequivParaResolution} shows, there are reasonable guesses for what the resolution of a point in a simplicial toric chart should be. As we do not need those particular charts to prove our main theorem, we do not discuss them further in this article.
\end{rem}
\subsection{Line bundles on toric stacks} \label{subsec:stackbackground}

Before going further with resolutions, we will discuss aspects of the theory of toric stacks that we will use. 
There are several notions of toric stacks in the literature. We will use toric stacks in the sense of \cite{geraschenko2015toric} on which the following exposition is based.
 
\begin{notation}
	We now fix the following notation for discussion of toric stacks, following \cite{geraschenko2015toric}. 
	\begin{itemize}
		\item $\kk$ is an algebraically closed field of characteristic $p\geq 0$, and $\GG_m$ is the multiplicative group.
		\item $L$ and $N$ are lattices. $K=L^*$ and $ M=N^*$ are the dual lattices. 
		\item Linear maps between lattices (or their $\RR$-vector spaces $L_\RR = L\tensor_\ZZ \RR$, $N_\RR= N\tensor_\ZZ, \RR$) will be underlined.
		\item $\TT_L$ and $\TT_N$ are the tori whose 1-parameter subgroups are naturally isomorphic to $L$ and $N$, respectively.
		They can be constructed via Cartier duality. Starting with $M$, construct the constant group scheme $\KK^M$, so that $M=\Spec(\KK^M)$. The Cartier dual is $\TT_N=\widehat{M}:=\Spec(\hom(\kk^M, \ZZ))$. 
		The dual lattice $M$ is the group of characters of $\TT_N$, and we will write $\chi^m \colon \TT_N\to \GG_m$ for the character associated with $m\in M$. 
		\item Given $\ul\beta \colon  L\to  N $, we obtain a map $\ul\beta^* \colon M\to K$ inducing a map $\TT_{\beta} \colon \TT_L \to \TT_{N} $. We set $G_\beta=\ker(\TT_{\beta})$.
	\end{itemize}
\end{notation}
 Suppose we have a short exact sequence of groups $0\to M'\xrightarrow{\ul\phi} M\to C_\phi\to 0$. Applying Cartier duality gives us $0\to G_\phi \to \TT_{N'}\to \TT_{N}\to 0$, so that $G_\phi = \widehat{C_\phi}$. Since Cartier duality is a duality, we also have $C_{\phi}=\widehat{G_{\phi}}$. When $G_{\phi}$ is a finite group, its Cartier dual is isomorphic to its group of characters giving us the identification
\begin{equation}
    \coker(\ul\phi^*\colon M'\to M) =\hom_{gp}(G_{\phi}, \GG_m).
    \label{eq:cokerToKer}
\end{equation}

A \emph{stacky fan}\footnote{There are several definitions of stacky fan in the literature. We use the definition in \cite{geraschenko2015toric}} is a pair $(\Sigma, \ul\beta \colon L\to N)$, where  $\Sigma$ is a fan on $L$, and $\ul\beta$ is a morphism of lattices with finite cokernel. The fan $\Sigma$ gives us a toric variety $X_\Sigma$. As $\ul \beta$ has finite cokernel, the induced map $\TT_{\beta} \colon  \TT_L\to \TT_N$ is surjective. 
The toric stack associated with a stacky fan is the quotient stack $\X_{\Sigma, \ul\beta}:=[X_\Sigma/G_\beta]$.
We will frequently drop the subscripts and write $\X$ or $\Y$ for a toric stack.
\begin{example}[Projective line]
	\begin{figure}
		\centering
		\scalebox{.5}{\begin{tikzpicture}[decoration=ticks, segment length=1cm]

			\begin{scope}[]
				\draw [help lines, step=1cm] (-2,-2) grid (2,2);
				\node[circle, fill=black, scale=.5] at (0,0) {};
				\draw[blue] (0,0) edge[->] (1,0) edge (2, 0);
				\draw[red] (0,0) edge[->] (0, 1) edge (0, 2);
			\end{scope}

			\begin{scope}[shift={(0,-4.5)}]
				\draw[decorate, help lines](-2,0)-- (2.01,0);
				\draw[gray!20](-2,0) -- (2,0);
				\draw[blue] (0,0) edge[->] (1,0) edge (2,0);
				\draw[red] (0,0) edge[->] (-1,0) edge (-2,0);
				\node[fill=black, circle, scale=.5] at (0,0) {};
			\end{scope}

			\node (v1) at (0,-2.5) {};
			\node (v2) at (0,-4) {};
			\draw  (v1) edge[->] node[fill=white]{$\beta$} (v2);
		\end{tikzpicture} }
		\caption{A presentation of $\PP^1$ as a toric stack.}
		\label{fig:p2StackyFan}
	\end{figure}
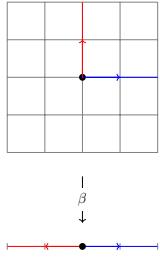
	The projective line can be presented as a toric stack as in \cref{fig:p2StackyFan}. That is, we consider the fan $\Sigma$ on $L_\RR=\RR^2$ whose cones are $\{0, \RR\cdot e_1, \RR\cdot e_2\}$ so that $X_\Sigma= \AA^2\setminus\{0\}$. Let $\ul \beta \colon  L \to N=\ZZ $ be the projection $\begin{pmatrix} 1 & -1 \end{pmatrix}$. $G_\beta= \TT^1$, and it acts by $g\cdot(z_1, z_2)= (gz_1, gz_2)$. 
		This defines a toric stack. Since the action is free, we can identify $\X_{\Sigma, \beta}=X_\Sigma/G_\beta = \PP^1$.
\end{example}

Let $(\Sigma,\ul\beta\colon L\to N)$ and $(\Sigma', \ul\beta^\circ\colon L'\to N')$ be two stacky fans.
A morphism of stacky fans $(\ul\phi,\ul\Phi)\colon (\Sigma', \ul\beta^\circ)\to (\Sigma, \ul\beta)$ is a pair of lattice morphisms making the following diagram commute.
\[\begin{tikzcd}
		L'\arrow{r}{\ul\Phi} \arrow{d}{\ul\beta^\circ} & L \arrow{d}{\ul\beta}\\
		N' \arrow{r}{\ul\phi}& N
	\end{tikzcd}
\]
Additionally, we ask that $\ul\Phi$ be a morphism of fans (so that for all $\sigma'\in \Sigma'$, there exists $\sigma\in \Sigma$ so that $\ul\Phi(\sigma')\subset \sigma$). A morphism of stacky fans induces a toric morphism of the corresponding toric stacks which we denote by $\phi \colon  \X_{\Sigma', \ul\beta'}\to \X_{\Sigma, \ul\beta}$. More descriptively, we obtain from this data a group homomorphism $f_\phi\colon G_\beta\to G_\beta'$ such that the map of toric varieties $\Phi \colon  X_\Sigma\to X_{\Sigma,'}$ is an $f_\phi$-map (see \cref{subsec:stackBackground}). 
\begin{example}
	Let $(\Sigma_e,\ul\beta_e)$ be the stacky fan of a point. Given any stacky fan $(\Sigma, \ul\beta)$, there exists a unique morphism of stacky fans $(\ul\phi_e, \ul\Phi_e)\colon (\Sigma_e, \ul{\beta}_e)\to (\Sigma, \ul\beta)$. 
\end{example}

We will be interested in line bundles on $\X_{\Sigma, \ul\beta}$ that can be described by a line bundle on $X_{\Sigma}$ equipped with a $G_\beta$ action. The sections of a line bundle on $\X_{\Sigma, \ul\beta}$ are the $G_\beta$-equivariant sections of the respective bundle on $X_\Sigma$.

We first give a short and standard description of line bundles on toric varieties in preparation for discussing line bundles on toric stacks.
A torus invariant Cartier divisor on $X_\Sigma$ is uniquely determined by a support function $\sF \colon  |\Sigma| \to \ZZ $, which is a continuous function on the support of the fan $|\Sigma| \subset L_\RR$ whose restriction to every $\sigma\in \Sigma$ is an integral linear function. Namely, the divisor associated to $\sF$ is $D_F =\sum_{\rho\in \Sigma(1)} -\sF(u_\rho) D_\rho$ where $u_\rho$ is the primitive generator of $\rho$ and $D_\rho$ the orbit closure associated to $\rho$. Thus, a support function determines a line bundle $\mathcal{O}_{X_\Sigma}(F) := \mathcal{O}_{X_\Sigma}(D_F)$, and, in fact, every line bundle on $X_\Sigma$ is isomorphic to a line bundle associated to a torus invariant divisor. Two support functions determine isomorphic line bundles when their difference is a linear function.

We can also construct a fan on $L\oplus \ZZ$ from the data of a support function, whose cones are
\[\Sigma_F:=\{(\sigma, \sF(\sigma))\st \sigma\in \Sigma\}\cup \{\langle (0, 1) , (\sigma, \sF(\sigma))\rangle\st \sigma\in \Sigma\}.\]
There is a map of fans $\ul\Pi \colon  \Sigma_{\sF}\to \Sigma $ from projection onto the first factor. The toric variety $X_{\Sigma_\sF}$ is the total space of the corresponding line bundle over $X_{\Sigma}$.

The sections of $\Sigma_\sF\to \Sigma$ are generated by the toric monomial sections. Each of these can be expressed as a map of fans $\ul S \colon \Sigma\to \Sigma_F $ so that $\ul\Pi\circ \ul S=\ul\id_\Sigma$, i.e.,  linear maps $k\colon L\to \ZZ$ with the property that $k(u_\rho)\geq F(u_\rho)$ for every $\rho\in \Sigma(1)$. Typically, the monomial sections are identified with the points of the lattice polytope
\[\Delta(\sF):=\{k\in K \st \forall \rho\in \Sigma(1), k(u_\rho)\geq F(u_\rho)\}.\]
A linear map inducing an isomorphism of line bundles can also be seen as an isomorphism of fans intertwining with the projection $\ul\Pi$. 

\begin{example}[Line bundles on $\PP^1$]
	\label{ex:lineBundlesOnP1}
	Let $X_\Sigma = \PP^1$ with complete fan generated by $u_{\rho_0}= -1$ and $u_{\rho_1}= 1$. Consider the support function $\sF$ determined by 
	\begin{align*}
		\sF(u_{\rho_0})= 0, \sF(u_{\rho_1})=-3.
	\end{align*}
	so that $\mathcal{O}_{\PP^1}(\sF)\simeq \mathcal{O}_{\PP^1}(3)$. We have drawn the fans $\Sigma$ and $\Sigma_{\sF}$ in \cref{fig:sectionsOnP1}.
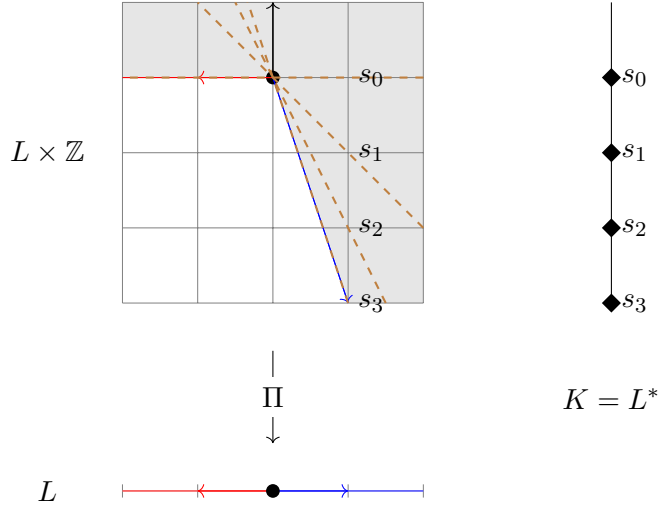
\begin{figure}
    \centering
    \begin{tikzpicture}[decoration=ticks, segment length=1cm]

			\begin{scope}[yscale=1]
				\fill[gray!20] (-2,1) -- (0,1) -- (1,-2) -- (2,-2) -- (2,2) -- (-2,2) -- cycle;
				\draw [help lines, step=1cm] (-2,-2) grid (2,2);
				\node[circle, fill=black, scale=.5] at (0,1) {};
				\draw[blue] (0,1) edge[->] (1,-2) edge (1,-2) ;
				\draw[red] (0,1) edge[->] (-1,1) edge (-2,1);

				\draw[black] (0,1) edge[->] (0,2) edge (0,2);
				\begin{scope}[dashed, brown, thick]
					\clip  (-2,-2) rectangle (2,2);
					\draw (1.5,-2) -- (-0.5,2);
					\draw (2,-1) -- (-1,2);
					\draw (2,1) -- (-2,1);
					\draw (1,-2) -- (-0.5,2.5);
				\end{scope}
				\begin{scope}[]
					\node[right] at (1,-2) {$s_3$};
					\node[right] at (1,-1) {$s_2$};
					\node[right] at (1,0) {$s_1$};
					\node[right] at (1,1) {$s_0$};
				\end{scope}

			\end{scope}

			\begin{scope}[shift={(0,-4.5)}]
				\draw[decorate, help lines](-2,0)-- (2.01,0);
				\draw[gray!20](-2,0) -- (2,0);
				\draw[blue] (0,0) edge[->] (1,0) edge (2,0);
				\draw[red] (0,0) edge[->] (-1,0) edge (-2,0);
				\node[fill=black, circle, scale=.5] at (0,0) {};

			\end{scope}

			\begin{scope}[shift={(3.5,0)},yscale=1]
				\draw (1,2) -- (1,-2);
				\node[below] at (1,-3) {$K=L^*$};
				\node[right] at (1,-2) {$s_3$};
				\node[right] at (1,-1) {$s_2$};
				\node[right] at (1,0) {$s_1$};
				\node[right] at (1,1) {$s_0$};
				\node[diamond, fill, scale=.5] at (1,-2) {};
				\node[diamond, fill, scale=.5] at (1,-1) {};
				\node[diamond, fill, scale=.5] at (1,0) {};
				\node[diamond, fill, scale=.5] at (1,1) {};
			\end{scope}

			\node (v1) at (0,-2.5) {};
			\node (v2) at (0,-4) {};
			\draw  (v1) edge[->] node[fill=white]{$\Pi$} (v2);
			\node at (-3,-4.5) {$L$};
			\node at (-3,0) {$L\times \mathbb Z$};
		\end{tikzpicture}     \caption{The total space of the line bundle $\mathcal O_{\PP^1}(3)$. The three equivariant sections of this map are labeled by dashed lines, which are also described by the polygon on the right-hand side.}
    \label{fig:sectionsOnP1}
\end{figure}
	The dashed lines are the images of the sections $\ul S\colon \Sigma\to \Sigma_{\sF}$. These are in bijection with the black diamonds on the right-hand side of the figure, which are the lattice points in the polytope $\Delta(\sF)=\{k\in K \st k(u_{\rho_0})\geq 0, k(u_{\rho_1})  \geq -3 \}$.
\end{example}

We now apply this construction to toric stacks. As before, to each support function $\sF \colon |\Sigma| \to \RR,$ we can associate the fan $\Sigma_\sF$ on $L\oplus \ZZ$. We define $\ul\beta_\sF:=\ul\beta\oplus \ul\id \colon  L\oplus \ZZ\to N\oplus \ZZ $.

The projection $\ul\Pi \colon \Sigma_\sF\to \Sigma$ extends to a morphism of stacky fans $(\ul\pi, \ul\Pi)\colon (\Sigma_\sF, \ul\beta_\sF)\to (\Sigma, \ul\beta)$ in the obvious way.
The total space $\X_{\Sigma_\sF, \ul\beta_{\sF}}$ now comes with a $G_\beta$ action, making the sheaf of sections of $\pi\colon\X_{\Sigma_\sF, \ul\beta_{\sF}}\to \X_{\Sigma}$ a $G_\beta$-equivariant sheaf. We write $\mathcal O_{\X_{\Sigma, \beta}}(\sF)$ for this sheaf. 

A different way to understand this sheaf is to describe the 
 $G_\beta$ action on fibers (see \cref{subsec:stackBackground} for notation). 
	Let $(\Sigma, \ul \beta)$ be a stacky fan containing a saturated (not necessarily top dimensional) smooth cone $\sigma$ of $\dim(N)$ with the property that $\ul \beta|_{\sigma}$ is injective. Let $\Sigma_{\sigma}$ be the fan of cones subordinate to $\sigma$, so we have an open immersion of stacky fans $(\ul i, \ul I)\colon(\Sigma_{\sigma}, \ul \beta)\to (\Sigma,  \ul\beta)$.

    Let $L'= L/\ker(\ul \beta), N'=N$ and $\ul \beta' \colon  L'\to N' $, and let $\Sigma'_\sigma$ be the image of $\Sigma_\sigma$ in the quotient. This gives us a stacky fan $(\Sigma_\sigma', \ul\beta')$. The map $(\ul\pi, \ul\Pi):(\Sigma_\sigma, \ul \beta)\to (\Sigma_\sigma', \ul \beta')$ induces an equivalence of stacks  $\X_{\Sigma_\sigma, \beta}=[\AA^{|\sigma|}/G_{\beta'}]$.
        
    Let $\sF$ be a support function $\Sigma$. 
    Then, $\mathcal O_{\AA^{|\sigma|}}(i^*\sF)$ has a canonical trivialization given by the linear map $k':L'
\to \ZZ$ satisfying $k'|_{\ul\Pi(\sigma)}=\sF|_{\ul\Pi(\sigma)},$ whose section $s_k\in \Gamma(\mathcal O_{\AA^{|\sigma|}}(i^*\sF))$ we declare to be the constant $e$-section. In coordinates on $\TT_L$, the section is parameterized by
\begin{align*}
	s_{k'}:\TT_{L'}\mapsto \TT_{L'}\times \GG_m\\
	x \mapsto (x, \chi^{k'}(x)).
\end{align*}
    
    We now prove that with respect to this trivialization the $G_{\beta'}$ action on the sheaf $\mathcal O_{\AA^{|\sigma|}}(i^*\sF)$ is given by multiplication on the fibers by the character $\chi^{k'}\in\hom(G_{\beta'}, \GG_m)$.

	Let $x_0$ denote the origin in $\AA^{|\sigma|}$, which is fixed on the $G_{\beta'}$-action on $\AA^{|\sigma|}$. For all $g\in G_{\beta}$, we compare 
	\begin{align}
		s_{k'}(g\cdot x_0) = (x_0, \chi^{k'}(x_0)) && g\cdot s_{k'} (x_0 )= (x_0,\chi^{k'}(g)\cdot \chi^{k'}(x_0)) \label{eq:gbetaAction}
	\end{align}
	It follows that the $G_\beta$ action on the sections of $\TT_{L'}\times \GG_m\to \TT_{L'}$ in this trivialization is given multiplication by the character $\chi^{k'}(g)$.

In this fashion, $\sF$ defines not only a line bundle $\mathcal O_{X_\Sigma}(\sF)$, but also a $G_\beta$ action on the sheaf, making it a $G_\beta$-equivariant sheaf. We define $\mathcal O_{\X_{\Sigma, \ul\beta}}(\sF)$ to be the sheaf of $G_{\beta}$-equivariant sections. 

 Among these $G_\beta$-equivariant sections are those which are additionally monomial sections. Each monomial section corresponds to a map of fans $\ul S \colon \Sigma\to \Sigma_\sF $ that is a section and extends to a map of stacky fans $(\ul s, \ul S)\colon (\Sigma, \ul\beta)\to (\Sigma_\sF, \ul\beta_\sF)$. The set of such sections is
\[\Delta_\beta(\sF):=\{k\in K \st \forall \rho\in \Sigma, k(u_\rho) \geq \sF(u_\rho) \text{ and there exists } m\in M \text{ with } k=\ul\beta^* m)\}, \]
and can also be written as
\[\Delta_\beta(\sF)=\{m\in M \st \forall \rho\in \Sigma, \ul\beta^*m( u_\rho )\geq\sF(u_\rho)\}\]
since $\ul \beta^*$ is injective.
This is a subset of $\Delta(\sF)$ (reflecting that all $G_\beta$-equivariant monomial sections are monomial sections, but not all monomial sections are $G_\beta$-equivariant).

\begin{example}[Non-separated line]
	\label{ex:LineBundlesOnNonSeparatedLine}
	Consider the fan $\Sigma = \{0, \langle e_1\rangle , \langle e_2\rangle\}$, so that $X_\Sigma=\AA^2\setminus \{0\}$. Let $\ul \beta \colon  \ZZ^2\to \ZZ $ be given by the matrix $\begin{pmatrix}1 & 1\end{pmatrix}$. Then the action $G_\beta\curvearrowright X_\Sigma$ is given by $(z_1, z_2)\mapsto (tz_1, t^{-1}z_2)$. The toric stack $\X_{\Sigma, \ul\beta}$ corresponds to the non-separated line (which has 2 origin points).
	
	Now consider the support function $\sF$ defined by $\sF (e_i)=a_i$ for $a_1, a_2 \in \ZZ$. The sections of $\mathcal O(\sF)$ on $\mathcal O_{X_\Sigma}=\AA^2$ are generated by $\Delta(\sF)=\{(k_1, k_2) \st k_i \leq a_i\}$. However, the $G_\beta$-equivariant sections are indexed by the subset $\Delta_\beta(\sF)=\{(k_1, k_2)\st k_1=k_2, k_i \leq a_i\}$. This can also be identified with $\{(m)\in M=N^*\st m\leq \min(a_1, a_2)\}$. 
\end{example}

We base our notation for toric line bundles on support functions (as opposed to toric divisors) for several expositional reasons.
One reason for using the support function notation is that it allows us to compare different $G_\beta$ actions on $\mathcal O_{X_\Sigma}(\sF)$. Since this is a line bundle, the set of $G_\beta$-structures is a torsor of over the character group of $G_\beta$. Moreover, the action of $\hom(G_\beta, \GG_m)$ can be understood via its identification with $\coker(\ul\beta)$ and \cref{eq:gbetaAction}. In summary:
\begin{prop}
    Let $\sF$ be a support function for $\Sigma$, and $\ul\beta \colon L\to N$. Let $\chi^{(-)}: K\to \hom(G_\beta, \GG_m)$ be the identification from \cref{eq:cokerToKer}. For any $k\in K$, we have a canonical isomorphism
	\[\mathcal O_{\X_{\Sigma, \beta}}(\sF+k) = \mathcal O_{\X_{\Sigma, \beta}}(\sF)_{\chi^{k}}\]
	where the latter bundle is the bundle $\mathcal O_{\X_{\Sigma, \beta}}(\sF)$ where the $G_{\beta}$ action has been twisted by multiplication by $\chi^k(g)$ on the fibers. 
	\label{prop:identifyingCharacter}
\end{prop}

Secondly, some functoriality properties of line bundles are easier to state in the language of support functions. For example:
\begin{prop}
	Let $(\ul\phi, \ul\Phi)\colon (\Sigma', \ul\beta')\to(\Sigma, \ul\beta)$ be a morphism of stacky fans, let $\phi \colon  \X'\to \X $ be the associated morphism of stacks. Let $\sF \colon  |\Sigma| \to \RR $ be a support function. Define $\ul\Phi^*F \colon  |\Sigma'|\to \RR $ by pullback. Then $\phi^*\mathcal O_{\X}(\sF)=\mathcal O_{\X'}(\ul\Phi^*F)$.  
	\label{prop:pullbackOfSupportFunction}
\end{prop}
\begin{proof}
	Recall that the line bundle  $\mathcal O_{\X}(\sF)$ is a line bundle $\mathcal O_{X}(\sF)$ on $X$ with a $G_\beta$ action, and the map $\Phi \colon  X'\to X$ is $f_\phi\colon G_{\beta'}\to G_{\beta}$ equivariant. Given a $G_{\beta}$-sheaf on $\X$, the pullback sheaf naturally inherits a $G_{\beta'}$ action via pushforward of the $G_{\beta}$ action via the homomorphism $f_\phi:G_\beta\to G_{\beta'}$. It is a classical fact on toric varieties that $\Phi^*\mathcal O_X(\sF)\cong \mathcal O_{X'}(\ul\Phi^* \sF)$.
 
    We need to also check that the $G_{\beta}$ action on $ \mathcal O_{X'}(\ul\Phi^* \sF)$ arises from the pushforward $f_\phi$. From \cref{prop:identifyingCharacter}, we have the constant $e$ section over a top dimensional cone of $X'$ which is parameterized by the section $s'_{k'}(x')=(x', \chi^{k'}(x'))$. We have that the $G_{\beta'}$ action on this section is given by multiplication with character $\chi^{k'}(g')$.  
    
    Let $k=\ul \Phi^*(k')$; then $s_k$ is the constant $e$ section.
     The $G_\beta$ action on $s_k$ is multiplication by the pullback character
    \begin{align*}
        \chi^{k'}(\TT_\Phi g)=\TT_\Phi^*\chi^{k'}(g)=\chi^{\ul \Phi^* k'}(g)=\chi^k(g)
    \end{align*}
    where the second equality uses that Cartier duality is a contravariant functor. This matches the $G_\beta$ action on $\mathcal O_{\X_{\Sigma,'}}(\ul\Phi^* \sF)$.
\end{proof}

Finally, some of the intuition for our constructions come from homological mirror symmetry (\cref{subsec:HMSMotication}), where support functions make an appearance in the form of a certain Hamiltonian function on the mirror space.

\begin{example}[Orbifold line]
	\label{ex:bundleInOrbifoldLine}
	We describe $[\AA^1/(\ZZ/2\ZZ)]$ as a toric stack, its line bundles, and a resolution of the point $e$ in this stack. We take $L=N=\ZZ$, $\ul\beta:L\xrightarrow{(2)} N$, and $\Sigma=\{0, \langle e_1\rangle \}$. Then, $X_{\Sigma}=\AA^1$ and the toric stack $\X_{\Sigma, \ul\beta}$ is the orbifold line $[\AA^1/(\ZZ/2\ZZ)]$.
 
	A support function $\sF$ on $L$ is determined by its value on $e_1$ so let $F_a$ be the support function such that $F_{a}(e_1)=a$. On $X_{\Sigma}$, the line bundles $\mathcal O(F_a)$ are all isomorphic. However, up to isomorphism, there are two line bundles on $\X_{\Sigma, \ul\beta}$. In particular, observe that when $G: N\to \ZZ$ is a linear function that $\ul\beta^* \colon  L\to \ZZ $ takes values in $2\ZZ$, therefore $\sF_1$ is not linearly equivalent to $\sF_0$. More concretely:
	\begin{itemize}
		\item The sheaf $\mathcal O_{\X_{\Sigma, \ul\beta}}$ corresponds to the $\ZZ/2\ZZ$-invariant sections of $\mathcal O_{\AA^1}$. Under identification of $\mathcal O_{\AA^1}$ with  $\kk[x]$, these sections correspond to the even degree polynomials.
		\item the second sheaf $\mathcal O_{\X_{\Sigma, \ul\beta}}(\sF_1)$ has sections corresponding to the $\ZZ/2\ZZ$ anti-invariant sections of $\mathcal O_{\AA^1}$, i.e., the odd degree polynomials (where the action of $G_\beta$ on sections is multiplication by $-1$).
	\end{itemize}
	A specific example with sections given by maps of fans, and their corresponding points in the lattice polytope is drawn in \cref{fig:sectionsOnOrbifoldLine}.
	\begin{figure}
		\centering
		\begin{tikzpicture}[decoration=ticks, segment length=1cm]

    \begin{scope}[yscale=-1]
        \fill[gray!20] (2,2) -- (0,0) -- (0,-2) -- (2,-2) -- cycle;
        \draw [help lines, step=1cm] (-2,-2) grid (2,2);
        \node[circle, fill=black, scale=.5] at (0,0) {};
        \draw[blue] (0,0) edge[->] (1,1) edge (2,2) ;
        \draw[black] (0,0) edge[->] (0,-1) edge (0,-2);
    \end{scope}

    \begin{scope}[shift={(6,0)},yscale=-1]
        \fill[gray!20] (2,1) -- (0,0) -- (0,-2) -- (2,-2) -- cycle;
        \draw [help lines, step=1cm] (-2,-2) grid (2,2);
        \node[circle, fill=black, scale=.5] at (0,0) {};
        \draw[blue] (0,0) edge[->] (2,1) edge (2,1) ;
        \draw[black] (0,0) edge[->] (0,-1) edge (0,-2);
    \end{scope}

    \begin{scope}[shift={(6,-4.5)}]
        \draw[decorate, help lines](-2,0)-- (2.01,0);
        \draw[gray!20](-2,0) -- (2,0);
        \draw[blue] (0,0) edge[->] (2,0) edge (2,0);
        \node[fill=black, circle, scale=.5] at (0,0) {};
    \end{scope}

    \begin{scope}[shift={(0,-4.5)}]
        \draw[decorate, help lines](-2,0)-- (2.01,0);
        \draw[gray!20](-2,0) -- (2,0);
        \draw[blue] (0,0) edge[->] (1,0) edge (2,0);
        \node[fill=black, circle, scale=.5] at (0,0) {};
    \end{scope}

    \begin{scope}[shift={(8.5,0)}]
        \draw (1,2) -- (1,-2);
        \node[below] at (1,-2) {$M=N^*$};
        \node[right] at (1,-0.5) {$s_0$};
        \node[right] at (1,0) {$s_1$};
        \node[right] at (1,0.5) {$s_2$};
        \node[right] at (1,1) {$s_3$};
        \node[diamond, fill=red, scale=.5] at (1,-0.5) {};
        \node[diamond, fill=green, scale=.5] at (1,0) {};
        \node[diamond, fill=red, scale=.5] at (1,0.5) {};
        \node[diamond, fill=green, scale=.5] at (1,1) {};
    \end{scope}

    \node (v1) at (0,-2.5) {$L_F$};
    \node (v2) at (0,-4) {$L$};
    \node at (-3,-4.5) {$L$};
    \node (v3) at (6,-2.5) {$N_F$};
    \node (v4) at (6,-4) {$N$};
    \draw  (v1) edge[->] node[fill=white]{$\Pi$} (v2);
    \draw  (v3) edge[->] node[fill=white]{$\pi$} (v4);
    \draw  (v1) edge[->] node[fill=white]{$\beta_F$} (v3);
    \draw  (v2) edge[->] node[fill=white]{$\beta$} (v4);
    \begin{scope}[yscale=-1]
        \draw[red, dashed] (-2,2) -- (2,-2);
        \draw[red, dashed] (-2,-2) -- (2,2);
        \draw[green, dashed] (2,0) -- (-2,0);
        \draw[green, dashed] (-1,2) -- (1,-2);
        \node[right] at (1,1) {$S_0$};
        \node[right] at (1,0) {$S_1$};
        \node[right] at (1,-1) {$S_2$};
        \node[right] at (1,-2) {$S_3$};
    \end{scope}

    \begin{scope}[shift={(6,0)},yscale=-1]
        \begin{scope}
            \clip (-2, -2) rectangle (2,1);
            \draw[red, dashed] (-2,1) -- (2,-1);
            \draw[red, dashed] (-2,-1) -- (2,1);
            \draw[green, dashed] (2,0) -- (-2,0);
            \draw[green, dashed] (-2,2) -- (2,-2);
        \end{scope}
        \node[right] at (1,0.5) {$s_0$};
        \node[right] at (1,0) {$s_1$};
        \node[right] at (1,-0.5) {$s_2$};
        \node[right] at (1,-1) {$s_3$};
    \end{scope}

    \begin{scope}[shift={(-5,0)}, yscale=-1]
        \draw (1,2) -- (1,-2);
        \node[below] at (1,2) {$K=L^*$};
        \node[right] at (1,-2) {$S_3$};
        \node[right] at (1,-1) {$S_2$};
        \node[right] at (1,0) {$S_1$};
        \node[right] at (1,1) {$S_0$};
        \node[diamond, fill, scale=.5] at (1,-2) {};
        \node[diamond, fill, scale=.5] at (1,-1) {};
        \node[diamond, fill, scale=.5] at (1,0) {};
        \node[diamond, fill, scale=.5] at (1,1) {};
    \end{scope}
\end{tikzpicture} 		\caption{Sections of a line bundle on $[\AA/(\ZZ/2\ZZ)]$ where the support function takes the value $\sF(u_{\rho_1
		})=-1$. On the left-hand side, the dashed lines represent the (not necessarily $\ZZ/2\ZZ$ equivariant) sections of the bundle $\mathcal O_\AA(\sF)$. The green dashed lines represent the $\ZZ/2\ZZ$ equivariant sections. }
		\label{fig:sectionsOnOrbifoldLine}
	\end{figure}
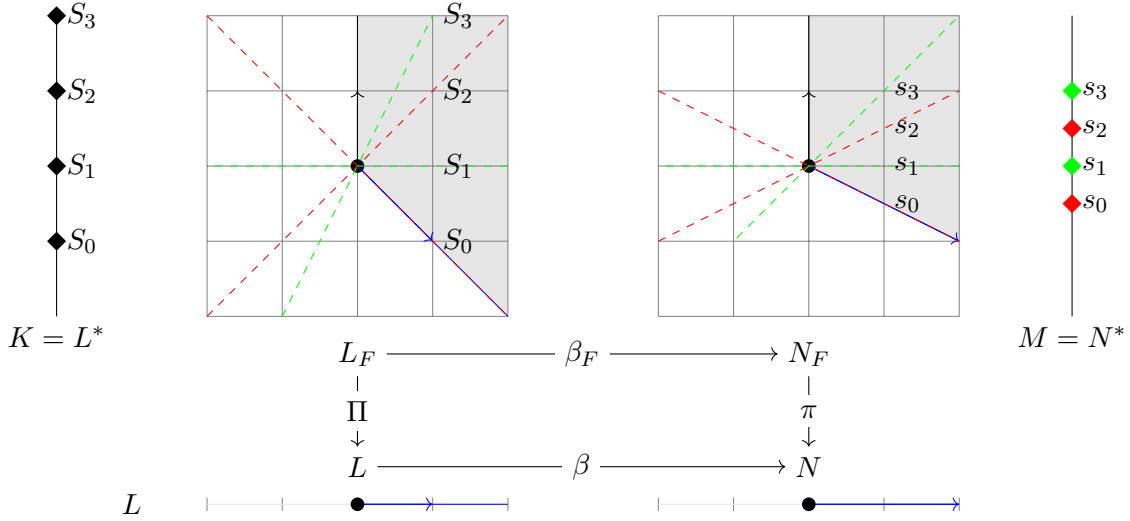
\end{example}

\subsection{Thomsen collection} \label{subsec:Thomsen}
In this section, we describe the set of line bundles from which our resolutions will be built. The construction of these line bundles is based on \cite{bondal2006derived}. We can assign to any $m \in M_\RR$ a $\TT_N$-invariant divisor on $X_\Sigma$ by
\begin{equation} \label{eq:bondalmap}
m \mapsto \sum_{\rho \in \Sigma(1)} \left\lfloor - \ul \beta^* m(u_\rho)  \right\rfloor D_\rho 
\end{equation}
where $u_\rho\in L$ is the primitive generator of the corresponding ray $\rho\in \Sigma(1)$ and $\lfloor \cdot \rfloor$ is the floor function.

For the remainder of this section, we will assume that $\Sigma$ is smooth. Then, we can identify the set of $\TT_N$ invariant divisors on $X_\Sigma$ with the set of support functions SF$(\Sigma)$ on $\Sigma$. Under that identification, \cref{eq:bondalmap} becomes a map
    $$ \bF \colon M_\RR \to \text{ SF}(\Sigma) $$
given by
    $$\bF(m)(u_\rho) = \left\lceil \ul\beta^*m(u_\rho) \right\rceil $$
for all $\rho \in \Sigma(1)$ and where $\lceil \cdot \rceil$ is the ceiling function.  Observe that $\bF(m_\RR)$ and $\bF(m_\RR+m_\ZZ)$ are linearly equivalent support functions for any $m_\ZZ\in M$. As a result, we obtain a map 
\begin{equation}
	\label{eq:lineBundleFromTorusPoints}
	\mathcal O(\bF(-))\colon M_\RR/M \to \Pic(\X_{\Sigma, \ul\beta}).
\end{equation}
where $\Pic(\X_{\Sigma, \ul \beta})$ is the Picard group. We call the image of this map the \emph{Thomsen collection}. This nomenclature is explained by \cref{sec:frobgen}.

We can describe the regions $\bF^{-1}(\sF)\subset M_\RR$ via the hyperplane arrangement $\{\langle m, \beta(u_\rho)\rangle \in \ZZ\}_{\rho\in \Sigma(1)}$. When drawing the stratification, we label each hyperplane with small ``hairs'' indicating the direction of $\beta(u_\rho)$. The hyperplane arrangement induces a stratification $\sS_{\Sigma, \ul\beta}$ of $M_\RR$, and each stratum $\strata\in \sS_{\Sigma, \ul \beta}$ is labeled by a line bundle on $\X_{\Sigma, \ul\beta}$ via $\bF$. We call the associated line bundle $\mathcal O_{\X_{\Sigma, \ul\beta}}(\strata)$.
Note that if $\tau$ is codimension 1 boundary component of $\sigma$, and there are no ``hairs'' pointing from $\tau$ into $\sigma$, that $\mathcal O_{\X_{\Sigma, \ul\beta}}(\strata)=\mathcal O_{\X_{\Sigma, \ul\beta}}(\tau)$. 

The stratification is $M$-periodic, so we obtain an induced stratification on $M_\RR/M$. This stratification was originally studied in \cite{bondal2006derived}.

\begin{rem}
This torus arrangement describes a portion of the FLTZ stop from \cite{fang2011categorification} corresponding to data from the 1-dimensional cones of the fan. This stratification plays a central role in the coherent-constructible correspondence and in homological mirror symmetry for toric varieties.
\end{rem}

\begin{example}[Beilinson Collection]
	On $\PP^n$, the Thomsen collection coincides with the Beilinson collection. In \cref{fig:thomsenP2}, we draw the stratification of $M_\RR$ given by $\bF$ for $n = 2$.
	\begin{figure}
		\centering
		\scalebox{.6}{\begin{tikzpicture}
\usetikzlibrary{calc, decorations.pathreplacing,shapes.misc}
\usetikzlibrary{decorations.pathmorphing}

\tikzstyle{fuzz}=[red, 
    postaction={draw, decorate, decoration={border, amplitude=0.15cm,angle=90 ,segment length=.15cm}},
]

\draw[dotted]  (-1.5,3) rectangle (3.5,-2);

\begin{scope}[]

\draw[fuzz](1,3)  -- (1,-2);
\draw[fuzz](-1.5,0.5) -- (3.5,0.5);
\draw[fuzz] (3.5,-2)  -- (-1.5,3);
\end{scope}

\begin{scope}[shift={(5,0)}]

\draw[fuzz](1,3)  -- (1,-2);
\draw[fuzz](-1.5,0.5) -- (3.5,0.5) ;
\draw[fuzz] (3.5,-2)  -- (-1.5,3);
\end{scope}

\begin{scope}[shift={(0,5)}]

\draw[fuzz](1,3)  -- (1,-2);
\draw[fuzz](-1.5,0.5) -- (3.5,0.5) ;
\draw[fuzz] (3.5,-2)  -- (-1.5,3);
\end{scope}

\begin{scope}[shift={(5,5)}]

\draw[fuzz](1,3)  -- (1,-2);
\draw[fuzz](-1.5,0.5) -- (3.5,0.5);
\draw[fuzz] (3.5,-2)  -- (-1.5,3);
\end{scope}
\node[fill=white, fill opacity =.5] at (1,0.5) {$\mathcal O$};
\node[fill=white, fill opacity =.5] at (2.5,2) {$\mathcal O(-D_x-D_y)$};
\node[fill=white, fill opacity =.5] at (4.5,4) {$\mathcal O(-D_x-D_y+D_z)$};
\node[fill=white, fill opacity =.5] at (2.5,7) {$\mathcal O(-D_x-2D_y+D_z)$};
\node[fill=white, fill opacity =.5] at (7.5,2) {$\mathcal O(-2D_x-D_y+D_z)$};
\node[fill=white, fill opacity =.5] at (6,5.5) {$\mathcal O(-D_x-D_y+2D_z)$};
\node[fill=white, fill opacity =.5] at (6,0.5) {$\mathcal O(-D_x+D_z)$};
\node[fill=white, fill opacity =.5] at (1,5.5) {$\mathcal O(-D_y+D_z)$};
\node[fill=white, fill opacity =.5] at (-0.5,-1) {$\mathcal O(-D_z)$};
\node[fill=white, fill opacity =.5] at (-0.5,4) {$\mathcal O(-D_y)$};
\node[fill=white, fill opacity =.5] at (4.5,-1) {$\mathcal O(-D_x)$};
\end{tikzpicture} }
		\caption{The stratification $S_\Sigma$ and Thomsen collection for $\PP^2$. The fundamental domain for the torus $M_\RR/M_\ZZ$ is outlined by the dotted line}
		\label{fig:thomsenP2}
	\end{figure}
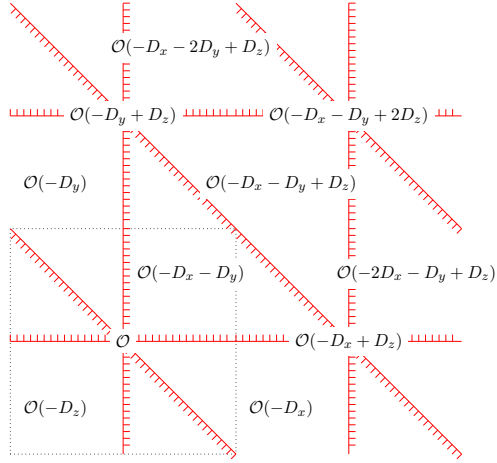
	A fundamental domain for the torus $M_\RR/M_\ZZ$ is marked out by the dashed lines. Observe that every line bundle in the figure is isomorphic to one inside the fundamental domain as claimed. In particular, the line bundles which appear in the Thomsen collection are $\mathcal O, \mathcal O(-1)$ and $\mathcal O(-2)$.
\end{example}
We can also produce a natural set of morphisms of line bundles in the Thomsen collection using the stratification $\sS_{\Sigma, \ul\beta}$. A version of the following statement appears in \cite[Proposition 5.5]{favero2022homotopy}.
\begin{prop}
	\label{prop:cocoreMorphisms}
	For any strata $\strata, \stratb\in \sS_{\Sigma, \ul \beta}$ with $\stratb<\strata$, we have $0 \in \Delta_\beta(\bF(\stratb)-\bF(\strata))$.
\end{prop}
\begin{proof}
	The claim is equivalent to showing that $\bF(m_1)(u_\rho) \geq \bF(m_2)(u_\rho)$ for any $m_1 \in \strata$, $m_2 \in \stratb$ and all $\rho \in \Sigma(1)$ (since $0$ is clearly in the image of $\ul \beta^* \colon M\to K $). Observe that for all $\rho \in \Sigma(1)$ and $m_1 \in \strata$, we have
    \[ \bF(m_1)(u_\rho) - 1 < m_1(u_\rho) \leq \bF(m_1)(u_\rho). \]
    Thus, if $\stratb < \strata$ and $m_2 \in \stratb$, we have
    \[\bF(m_1)(u_\rho) - 1 \leq m_2(u_\rho) \leq \bF(m_1)(u_\rho),  \]
    and, therefore, 
    \[ \bF(m_2) (u_\rho) = \lceil m_2(u_\rho) \rceil \leq \bF(m_1)(u_\rho) \]
    as required.
\end{proof}

\begin{df}
    For all $\strata, \stratb \in \sS_{\Sigma, \ul \beta}$ with $\stratb < \strata$, the \emph{boundary morphism}
        $$\bm_{\strata,\stratb}\in \hom(\mathcal O_\X(\bF(\strata), \mathcal O_\X(\bF(\stratb))$$
    is the morphism corresponding to $0 \in \Delta_\beta(\bF(\stratb)-\bF(\strata))$ guaranteed to exist by \cref{prop:cocoreMorphisms}.
\end{df}

The morphisms appearing in \cref{ex:resolutionOfLine,eq:resolutionPtInOrbiA1} are all boundary morphisms. Additional examples appear in  \cref{ex:pointInA1,fig:fltzCP}. Our resolutions will be built from boundary morphisms. We first observe that boundary morphisms compose to a boundary morphism.

\begin{prop}
	Let $\strata>\stratb>\stratc$. Then $\bm_{\stratb,\stratc}\circ \bm_{\strata,\stratb}=\bm_{\strata,\stratc}$.
\end{prop}
\begin{proof}
	The tensor product of equivariant sections is given by the addition of their lattice point representatives in $K$.
\end{proof}
The Thomsen collection behaves well under open inclusions.
\begin{lem}
	\label{lem:thomsenPullback}
	Let $(\ul i, \ul I):(\Sigma, \ul\beta)\to (\Sigma', \ul\beta')$ be an open inclusion. Let $m'\in M'_\RR$ be any point. Then 
	\[\mathcal O_{\X_{\Sigma, \beta}}(\bF(\ul i^* m'))= i^*\mathcal O_{\X_{\Sigma', \beta'}}(\bF'(m'))\]
	and 
	\[ i^*\bm_{\strata,\stratb}= \bm_{\ul i ^*\strata,\ul i ^*\stratb}.\]
\end{lem}
\begin{proof}
	By \cref{prop:pullbackOfSupportFunction}, we have $i^*\mathcal O_{\X_{\Sigma', \beta'}}(\bF'(m'))=\mathcal O_{\X_{\Sigma, \beta}}(\ul I^*\bF'(m'))$. We check that this support function agrees with $\bF(\ul i^*m')$ by evaluation against primitive vectors $u_\rho$ for the 1-dimensional cones $\rho\in \Sigma(1)$. 
	\begin{align*}
		\bF(\ul i^* m')(u_\rho)=&\lceil\ul\beta^*(\ul i^* m')(u_\rho)\rceil
		=\lceil \ul I^*(\ul\beta')^*m'(u_\rho)\rceil
  \intertext{Because $(\ul i, \ul I)$ is an open inclusion, $I(u_\rho)\in \Sigma'(1)$, allowing us to substitute the definition of $\bF'$,}
		=&\bF'(m')(\ul I u_\rho)=\ul I^*\bF'(m')(u_\rho)
	\end{align*}
	A similar computation proves that pullback sends boundary morphisms to boundary morphisms.
\end{proof}
\subsection{Pushforward along finite quotients}
\label{subsec:finiteQuotient}
As illustrated by \cref{ex:pointInOrbifoldLine}, we will need to compute the pushforward of a line bundle along a finite quotient. The goal of this section is to make such a computation.

\begin{df}
        Let $(\ul\pi, \ul\Pi) \colon (\Sigma, \ul\beta)\to (\Sigma', \ul\beta')$ be a morphism of stacky fans such that $\ul\Pi$ is the identity and $\ul\pi$ has trivial kernel and finite cokernel. We call $(\ul\pi, \ul\Pi)$ a change of group with finite cokernel.
\end{df} 
The map is called a change of group with finite cokernel for the following reason. From the commutativity of the diagram
\[\begin{tikzcd}
	G_\beta\arrow{r}{i_\TT} \arrow[dashed]{d}{f_\pi}& \TT_L\arrow{r}{\TT_\beta} \arrow{d}{\TT_\Pi} &\TT_N \arrow{d}{\TT_\pi}\\
	G_{\beta'}\arrow{r}&\TT_{L'}\arrow{r}{\TT_{\beta'}} &\TT_{N'}
\end{tikzcd}
\]
we obtain $\TT_{\beta'}\circ\TT_{L'}\circ i_\TT=0$ and conclude that there exists a group homomorphism $f_\pi \colon G_{\beta}\to G_{\beta'}$ such that the morphism of toric stacks $[X_\Sigma/G_\beta]\to [X_\Sigma/G_{\beta'}]$ arises from an $f_\pi$ map.
We claim that the cokernel of $f_\pi$ is finite.
\begin{prop}
	\label{prop:soManyThingsAreTheSame}
	Suppose that $(\ul\pi, \ul\Pi): (\Sigma, \ul\beta)\to (\Sigma', \ul\beta')$ is a change of group with finite cokernel. Let $\tilde{\pi}^* \colon M'_\RR/M'\to M_\RR/M$ and $f_\pi: G_{\beta}\to G_{\beta'}$ be the induced maps. We have the following relations between groups 
	\[
		\coker(f_\pi) =\ker(\TT_{\pi})\text{ $\ot$ Cartier Dual $\to$ }
		\coker(\ul\pi^* \colon M'\to M)=\ker(\tilde \pi^*).\]
\end{prop}
\begin{proof}
	First, we prove $\coker(f_\pi \colon G_{\beta}\to G_{\beta'})=\ker(\TT_\pi) $. First, we observe that the map $\TT_\beta \colon  \TT_L\to \TT_N $ is surjective (see the discussion following \cite[Definition 2.4]{geraschenko2015toric}) so $\coker(\TT_\beta)=0$. Then, we apply the zig-zag lemma to obtain the  diagram:
	\[	
	\begin{tikzpicture}
			\node (v1) at (-2,1) {$0$};
			\node (v2) at (0.5,1) {$G_{\beta}$};
			\node (v3) at (2.5,1) {$G_{\beta'}$};
			\node (v4) at (-2,0) {$0$};
			\node (v5) at (0.5,0) {$\TT_L$};
			\node (v6) at (2.5,0) {$\TT_{L'}$};
			\node (v7) at (4.5,0) {$0$};
			\node (v8) at (-4,-2) {$0$};
			\node (v9) at (-2,-2) {$\ker(\TT_{\pi})$};
			\node (v10) at (0.5,-2) {$\TT_N$};
			\node (v11) at (2.5,-2) {$\TT_{N'}$};
			\node (v12) at (-2,-3) {$\ker(\TT_\pi)$};
			\node (v13) at (0.5,-3) {$\text{coker}(\TT_{\beta})=0$};
			
			\draw[->]  (v1) edge (v2);
			\draw[->]  (v2) edge node[midway, fill=white]{$f_\pi$} (v3);
			\draw[->]  (v4) edge (v5);
			\draw[->]  (v5) edge node[midway, fill=white]{$\TT_\Pi$}(v6);
			\draw[->]  (v6) edge (v7);
			\draw[->]  (v8) edge (v9);
			\draw[->]  (v9) edge (v10);
			\draw[->]  (v10) edge node[midway, fill=white]{$\TT_\pi$}(v11);
			\draw[->]  (v12) edge (v13);
			\draw[->]  (v1) edge (v4);
			\draw[->]  (v4) edge (v9);
			\draw[->]  (v9) edge (v12);
			\draw[->]  (v2) edge (v5);
			\draw[->]  (v5) edge node[midway, fill=white]{$\TT_{\beta}$} (v10);
			\draw[->]  (v10) edge (v13);
			\draw[->]  (v3) edge (v6);
			\draw[->]  (v6) edge node[midway, fill=white]{$\TT_{\beta'}$} (v11);
			\draw[->] (v3) .. controls (3.5,1) and (3.5,1) .. (3.5,1) .. controls (4,1) and (4,-1.5) .. (3.5,-1.5) .. controls (2.5,-1.5) and (-2,-1.5) .. (-3,-1.5) .. controls (-3.5,-1.5) and (-3.5,-3) .. (-3,-3) .. controls (-3,-3) and (-3,-3) .. (v12);	
			\node[above] at (1.5,-1.5) {$\delta$};
			\end{tikzpicture}
	\]
	verifying the first claim.
 
    Now, we prove the second equality. Given $[m]\in \coker(\ul\pi^*)$, consider the element $m\tensor_\RR 1\in M_\RR$. Since $\pi$ has a trivial kernel and finite cokernel, the map $\ul\pi_\RR^* \colon  M'_\RR\to M_\RR $ is an isomorphism. Consider the lift $(\ul\pi^*_\RR)^{-1}(m\tensor_\RR 1)\in M'_\RR$, and let $[(\pi^*_\RR)^{-1}(m\tensor_\RR 1)]\in M'_\RR/M'$. Since $\ul\pi^*([(\ul\pi^*_\RR)^{-1}(m\tensor_\RR 1)])=[m\tensor_\RR 1]$, we learn that $[(\ul\pi^*_\RR)^{-1}(m\tensor 1)]\in \ker(\tilde{\pi}^*)$. This yields a map $\coker(\ul \pi^*)\to \ker(\tilde{\pi}^*)$.

	Conversely, any element $[m'\tensor_{\RR}r]\in \ker(\tilde \pi^*)$, 
    we have $\ul\pi^*_\RR(m'\tensor r)\in M$ and thus represents a class in $\coker(\ul\pi^* \colon  M'\to M) $. This gives the inverse $\coker(\ul\pi^*)\ot \ker(\tilde{\pi}^*)$.
\end{proof}

\begin{df} \label{def:finitequotient}
    Let $(\ul \pi, \ul \Pi)\colon (\Sigma, \ul \beta)\to (\Sigma', \ul \beta')$ be a change of group with finite cokernel. If the surjective map $f_\pi$ splits, we say that $(\ul \pi, \ul \Pi)$ is a finite quotient. 
\end{df}

\begin{example}
    Let $(\ul \pi, \ul \Pi)\colon (\Sigma, \ul \id)\to (\Sigma', \ul \beta')$ be a change of group with finite cokernel so that the first toric stack is presented as a toric variety. As $G_{\id}$ is the trivial group, $f_\pi$ splits, and this map is a finite quotient.
\end{example}

A morphism of stacky fans satisfying \cref{def:finitequotient} is called a finite quotient because the stack $\X_{\Sigma', \ul\beta'}$ is the stacky quotient of $\X_{\Sigma, \ul \beta}$ by a finite group.

\begin{lem}[Pushforward along finite quotients]
	\label{prop:quotientOfLineBundle}
        Assume that $|\coker(\ul \pi_*)|\in \kk^\times$. Let $\sF$ be a support function on $\Sigma$ and let $(\ul\pi, \ul\Pi) \colon (\Sigma, \ul\beta) \to(\Sigma', \ul\beta')$ be a finite quotient.  Define the support function $\ul\Pi_*F$ to be the pushforward of $\sF$ along $\ul\Pi$. Then,
	\begin{equation}
		\pi_*\mathcal O(\sF)=\bigoplus_{[m]\in \coker(\ul\pi^*:M'\to M)}\mathcal O_{\X'}(\ul\Pi_*(\sF-\ul\beta^*m)).
		\label{eq:quotientOfLineBundle}
	\end{equation}
\end{lem}
\begin{proof}
	The pushforward is given by the coinduced representation, which can be understood in this setting via the finite Fourier transform (\cref{subsubsec:characterPushforward})
	\begin{align*}
		\pi_*\mathcal O(\sF)=\phi_c\mathcal O(\sF)
		=& \bigoplus_{\chi\in \hom(\coker(f_\pi), \GG_m)} \mathcal O(\ul \Pi_* \sF)_\chi\\
		\intertext{We identify $\coker(f_\pi)$ with $\ker(\TT_{\pi})$ using \cref{prop:soManyThingsAreTheSame}, and $\hom_{gp}(\ker(\TT_{\pi}), \GG_m)$ with $\coker(\ul\pi^* \colon  M'\to M) $ using \cref{eq:cokerToKer}.  }
		=&\bigoplus_{[m]\in \coker(\ul\pi^*:M'\to M)} \mathcal O(\ul \Pi_* \sF)_{\chi^{\ul \beta^*m}}\\
		\intertext{Finally, by \cref{prop:identifyingCharacter}}
		=&\bigoplus_{[m]\in \coker(\ul\pi^*:M'\to M)} \mathcal O(\ul \Pi_* \sF-\ul\beta^*m)
	\end{align*}
 as claimed.
\end{proof}
\begin{example}[Resolving point on the orbifold line]
	We pick up from \cref{ex:bundleInOrbifoldLine}. Observe that under the quotient map $\pi \colon \AA^1\to [\AA^1/(\ZZ/2\ZZ)]$, we have
	\begin{align*}
		\pi_* \mathcal O_{\AA^1}=\mathcal O_{\X_{\Sigma, \ul\beta}} \oplus O_{\X_{\Sigma, \ul\beta}} (\sF_1) &  & \pi^*\mathcal O_{\X_{\Sigma, \ul\beta}}\cong\pi^*\mathcal O_{\X_{\Sigma, \ul\beta}}(\sF_1)\cong \mathcal O_{\AA^1}
	\end{align*}

	Let $\phi \colon  e\to [\AA^1/(\ZZ/2\ZZ)] $ denote the inclusion of the identity point. We now describe a resolution of the sheaf $\phi_*\mathcal O_e$. Consider the point $e'= 1 \in \AA^1$. Using the quotient map, we can push forward a resolution of $\mathcal O_{e'}$ to obtain a resolution for $\mathcal O_{e}$. We take the standard resolution for $\mathcal O_{e'}$ as in \cref{eq:resolutionPtInA1}. The pushforward is
	\begin{equation}
		\begin{tikzpicture}
			\draw (-3,-0.5) -- (4,-0.5);
			\node (v1) at (-1.5,-2.5) {$\mathcal O_{\X_{\Sigma, \ul\beta}}(-D_1)$};
			\node (v2) at (-3,-1) {$\mathcal O_{\X_{\Sigma, \ul\beta}}$};
			\node (v3) at (0.5,-1) {$\mathcal O_{\X_{\Sigma, \ul\beta}}(-D_1)$};
			\node (v5) at (4,-1) {};
			\draw[->]  (v1) edge node[red,midway, fill=white]{$x_0$} (v2);
			\draw[->]  (v1) edge node[midway, fill=white]{$-1$}(v3);
			\draw[thick, red] (0.5,-0.5) -- (0.75,-0.5);
			\node[circle, fill=black, scale=.5] at (0.5,-0.5) {};
			\draw[thick, red] (-3,-0.5) -- (-2.75,-0.5);
			\node[circle, fill=black, scale=.5] at (-3,-0.5) {};
			\node (v4) at (2.5,-2.5) {$\mathcal O_{\X_{\Sigma, \ul\beta}}$};
			\draw[->]  (v4) edge  node[red,midway, fill=white]{$x_0$}(v3);
			\draw[->]  (v4) edge  node[midway, fill=white]{$-1$}(v5);
		\end{tikzpicture}
		\label{eq:resolutionPtInOrbiA1}
	\end{equation}
	The pushforward turns out to be exact so this indeed gives us a resolution of the point $e$ in $[\AA^1/(\ZZ/2\ZZ)]$.
    \label{ex:pointInOrbifoldLine}
\end{example}
We now discuss the behavior of the Thomsen collection under finite quotients.
\begin{lem}
	\label{lemma:thomsenQuotient}
	Let $(\ul\pi, \ul\Pi): (\Sigma, \ul\beta)\to (\Sigma', \ul\beta')$ be a finite quotient, where $|\coker(f_\beta)|\in \kk^\times$. Let $p\in M_\RR/M_\ZZ$ be any point. Let $\tilde \pi \colon  M'_\RR/M'_\ZZ\to M_\RR/M_\ZZ $ be the associated map on tori. Then 
	\[\pi_* \mathcal O_{\X_{\Sigma, \ul\beta}}(\bF(p))= \bigoplus_{p'\in \tilde \pi^{-1}(p)} \mathcal O_{\X_{\Sigma, \ul\beta'}}(\bF(p')).\]
	Furthermore, given $\strata>\stratb$, we have 
	\[\pi_*\bm_{\strata,\stratb}(x_1, \ldots x_k)=(\bm_{\strata_1,\stratb_1}x_1, \ldots,\bm_{\strata_k,\stratb_k}x_k )\]
	where $\strata_i>\stratb_i$ are the lifts of $\strata, \stratb$ which preserve adjacency.
\end{lem}
\begin{proof}
	This follows from \cref{eq:quotientOfLineBundle}, and the identification $\coker(\ul\pi^* \colon M'\to M)=\ker(\tilde \pi^*) $ from \cref{prop:soManyThingsAreTheSame}.
\end{proof}

\subsection{Toric stacky charts} \label{subsec:charts}
In this section, we introduce and discuss the notion of a chart on a toric stack. We will eventually use these charts to cover our toric stack and localize computations of our resolutions. 

\begin{df} \label{def:inclusion}
	A morphism of stacky fans $(\ul\phi, \ul\Phi) \colon (\Sigma', \ul\beta')\to (\Sigma, \ul\beta)$ is an \emph{inclusion} if $\ul\phi$ and $\ul\Phi$ have trivial kernel and torsion-free cokernel, and for every cone, $\sigma$ of $\Sigma'$ there exists a cone $\tau$ of $\Sigma$ so that $\sigma= \ul\Phi^{-1}(\tau)$. 
	An inclusion of stacky fans is an \emph{immersion} if $\Sigma'= \{\ul\Phi^{-1}(\tau) \st \tau\in \Sigma\}$.
	An inclusion of stacky fans is an \emph{open inclusion} if $\ul\phi$ and $\ul\Phi$ are isomorphisms. 
\end{df}

\begin{df}
	\label{def:toricStackyChart}
	  A \emph{stacky coordinate fan} is a stacky fan $(\Sigma, \ul\beta\colon L\to N)$ such that $\ul\beta$ is injective and $\Sigma$ is the closure of a unique top dimensional cone. A stacky coordinate fan is called a \emph{simplicial stacky coordinate fan} if the fan $\Sigma$ is simplicial.
	A \emph{smooth stacky coordinate fan} is a simplicial stacky coordinate fan where the one-dimensional cones $\Sigma(1)$ form a basis for $L$, i.e., $\X_{\Sigma, \beta}=[\AA^n/G]$ for a finite group $G$, and $|G|\in \kk^\times$.
\end{df}

Every inclusion/immersion of stacky fans induces an inclusion/immersion of a toric stack. Unfortunately, since every toric stack may be represented by several different stacky fans, to represent a morphism of toric stacks we may have to modify our (possibly preferred) choice of stacky fans.
\begin{df}
	Given a stacky fan $(\Sigma, \ul\beta\colon L\to N)$, the stacky fan $(\Sigma\oplus 0, \ul\beta\oplus \ul\beta'\colon L\oplus \ZZ^n\to N)$ is a \emph{stabilization} of $(\Sigma, \ul \beta)$. A \emph{stabilized} (smooth, simplicial) stacky coordinate fan is a stabilization of a (smooth, simplicial) stacky coordinate fan.

	A \emph{(simplicial, smooth) toric chart} for $(\Sigma, \ul\beta)$ is an open inclusion of a substack $(\ul i, \ul I)\colon (\Sigma', \ul \beta')\to (\Sigma, \ul \beta)$ and $(\Sigma',\ul \beta')$ is a stabilization of a (simplicial, smooth) stacky toric coordinate fan.

	We say that a toric stack $\X_{\Sigma, \ul\beta}$ is \emph{covered by (simplicial, smooth) charts} if there exists a collection of (simplicial, smooth) toric charts $\{(\ul i_\rho, \ul I_\rho) \colon  (\Sigma'_\rho, \ul \beta'_\alpha)\to (\Sigma, \ul \beta)\}_{\alpha\in A} $ so that the maps $I_\alpha: X_{\Sigma'}\to X_{\Sigma}$ form a covering of $X_\Sigma$.
\end{df}
\begin{rem}
	In particular, if $\X_{\Sigma, \ul\beta}$ is covered by simplicial stacky charts, it contains at least one cone $\sigma$ of $\dim(N)$, and the one-dimensional cones of $\ul\beta(\sigma)$ form an $\RR$-basis for $N_\RR$.
\end{rem}
\begin{prop} 
The morphism of stacky fans $(\ul \id, \ul {\text{Id}}\oplus 0): (\Sigma, \ul\beta)\to (\Sigma\oplus 0, \ul\beta\oplus \ul\beta')$ induces an equivalence of toric stacks.\label{prop:stabilization}
\end{prop}
\begin{proof}
    Follows immediately from \cite[Theorem B.3]{geraschenko2015toric}.
\end{proof}

\begin{example}[Non-separated line]
	\label{ex:pointInNonSeparatedLine}
	We return to the non-separated line introduced in \cref{ex:LineBundlesOnNonSeparatedLine}
We draw its stacky fan on the right-hand side in \cref{fig:coveringStackyLine}. Observe that $\X_{\Sigma, \ul\beta}$ is covered by stabilized smooth stacky charts. 
	\begin{figure}
	    \centering
	    \begin{tikzpicture}[decoration=ticks, segment length=1cm]

	\begin{scope}[]
		\begin{scope}[]
				\draw [help lines, step=1cm] (-2,-2) grid (2,2);
				\node[circle, fill=black, scale=.5] at (0,0) {};
				\draw[blue] (0,0) edge[->] (1,0) edge (2, 0);
				\draw[red] (0,0) edge[->] (0, 1) edge (0, 2);
			\end{scope}

			\begin{scope}[shift={(0,-4.5)}]
				\draw[decorate, help lines](-2,0)-- (2.01,0);
				\draw[gray!20](-2,0) -- (2,0);
				\draw[blue] (0,0) edge[->] (1,0) edge (2,0);
				\draw[red] (0,-0.1) edge[->] (1,-0.1) edge (2,-0.1);
				\node[fill=black, circle, scale=.5] at (0,0) {};
			\end{scope}

			\node (v1) at (0,-2.5) {};
			\node (v2) at (0,-4) {};
			\draw  (v1) edge[->] node[fill=white]{$\beta$} (v2);
\end{scope}

\begin{scope}[shift={(-8,0)}]
		\begin{scope}[]
				\draw [help lines, step=1cm] (-2,-2) grid (2,2);
				\node[circle, fill=black, scale=.5] at (0,0) {};
				\draw[blue] (0,0) edge[->] (1,0) edge (2,0);
			\end{scope}

			\begin{scope}[shift={(0,-4.5)}]
				\draw[decorate, help lines](-2,0)-- (2.01,0);
				\draw[gray!20](-2,0) -- (2,0);
				\draw[blue] (0,0) edge[->] (1,0) edge (2,0);
				\node[fill=black, circle, scale=.5] at (0,0) {};
			\end{scope}

			\node (v1) at (0,-2.5) {};
			\node (v2) at (0,-4) {};
			\draw  (v1) edge[->] node [fill=white] {$\beta$} (v2);

\end{scope}
		\draw[->] (-5,0) -- (-4,0) node[fill=white]{$I$} -- (-3,0);
\draw[->] (-5,-4.5) -- (-4,-4.5) node[fill=white]{$i$} -- (-3,-4.5);
\end{tikzpicture} 	    \caption{A smooth chart covering a portion of the non-separated line.}
	    \label{fig:coveringStackyLine}
	\end{figure}
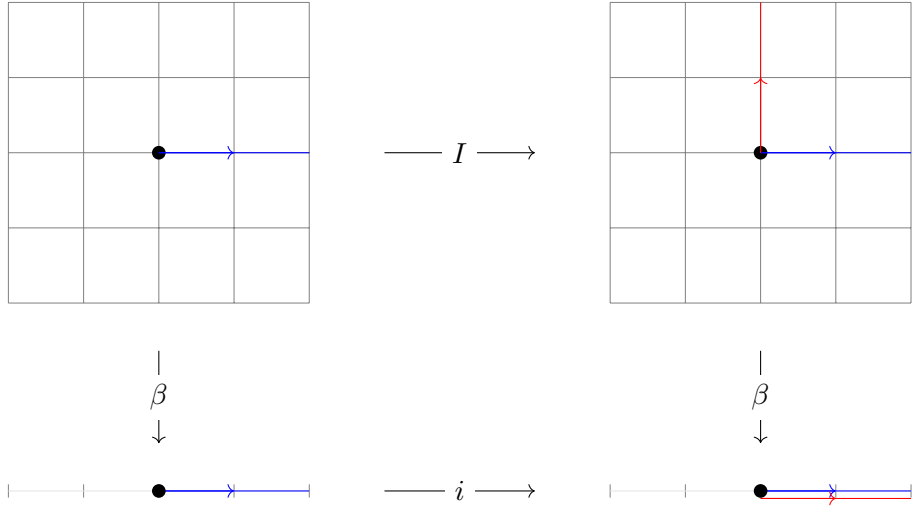
    We draw the map of stacky fans corresponding to one of the charts in the figure. Observe that the $\AA^1$ chart does not use the usual (non-stacky) fan, but rather the stacky fan which is obtained by stabilizing the standard fan.

	The doubled origin corresponds to two toric divisors $D_0, D_1$ (coming from the $G_\beta$-invariant divisors $z_1=0, z_2=0$), but by the discussion in \cref{ex:LineBundlesOnNonSeparatedLine}, we only need to consider divisors of the form $a(D_0+D_1)$.
    The $G_\beta$-equivariant sections of $\mathcal O_{X_\Sigma}$ are given by the polynomial ring $\KK[z_1z_2]$.
	Let $e=[1]$ in $\X_{\Sigma, \ul\beta}$.
	The complex
	\begin{equation}
		\begin{tikzpicture}
			\draw (-2.5,-0.5) -- (4,-0.5);
			\node (v1) at (1,-1) {$\mathcal O_{X_\beta}$};
			\node (v2) at (-2.5,-1) {$\mathcal O_{X_\beta}$};
			\node (v3) at (4,-1) {};
			\draw[->]  (v1) edge node[ fill=white]{${\color{red}x_0}{\color{blue}x_1}$} (v2);
			\draw[->]  (v1) edge node[midway, fill=white]{$-1$}(v3);
			\draw[thick, red] (-2.5,-0.55) -- (-2.25,-0.55);
			\draw[thick, blue] (-2.5,-0.45) -- (-2.25,-0.45);
			\node[circle, fill=black, scale=.5] at (-2.5,-0.5) {};
		\end{tikzpicture}
		\label{eq:resolutionPtInNonSepLine}
	\end{equation}
	is a resolution of  $\mathcal O_e$, which is seen by restricting the resolution to the  $\AA^1$ charts. When restricted, this resolution exactly matches \cref{eq:resolutionPtInA1}. 
\end{example}

\begin{example}[Weighted projective space]
	\label{ex:pointInWeightedP1}
	We now consider the stacky fan associated with a weighted projective space. Consider the fan $\Sigma$ on $\ZZ^2$ so that $X_\Sigma=\AA^2\setminus \{0\}$. We take $\ul\beta \colon  \ZZ^2\to \ZZ$ which is given by the matrix $\begin{pmatrix}1 &  -2\end{pmatrix}$. This corresponds to the action of $G_\beta \curvearrowright X_\Sigma$ by  $(z_1, z_2)\mapsto (tz_1, t^2 z_2)$. The action has finite stabilizers, and the quotient $X_{\Sigma}/G$ is the $(1,2)$ weighted projective space, denoted $\PP(1, 2)$. Note that sheaves on $[X_{\Sigma}/G]$ do not agree with sheaves on the quotient.

	The weighted projective space can also be described via a covering by smooth stacky charts. The first chart, which we call $U_0$, comes from the toric morphism which descends to the quotient as
	\begin{align*}
		\AA^1 \to& \PP(1,2) \\
		z_1\mapsto& [z_1:1]_{(1, 2)}.
	\end{align*}
	The second chart, which  we call $U_1$, is parameterized by a morphism of toric stacks $[\AA^1/(\ZZ/2\ZZ)]\to [X_\Sigma/G_\beta]$ where $[\AA^1/(\ZZ/2\ZZ)]$ is realized as a toric stack with $L=\ZZ, N=\ZZ, \beta'=(2)$ and $\Sigma=\{\langle 0 \rangle,\langle 1\rangle\cdot \NN \}$. The morphism is given by
	\[\begin{tikzcd}
			\ZZ\arrow{d}{2 \cdot} \arrow{r}{\begin{pmatrix}0\\1 \end{pmatrix}}& \ZZ^2 \arrow{d}{\beta}\\
			\ZZ\arrow{r}{\id}&  \ZZ
		\end{tikzcd}
	\]

	Support functions for $[X_\Sigma/G_\beta]$ are determined by the integer values
	\begin{align*}
		a_0=\sF(\ul\beta(\langle 1, 0\rangle)) = F(\langle 1\rangle ) &  & a_1=\sF(\ul\beta(\langle 0, 1\rangle))=F(\langle 2 \rangle)
	\end{align*}
	Let $e\in [X_\Sigma/G_\beta]$ be the identity point.
	We claim that the following chain complex is a resolution of $\mathcal O_{e}$ 
	\[
		\begin{tikzpicture}
			\draw (-3,-0.5) -- (4,-0.5);
			\node (v1) at (-1.5,-2.5) {$\mathcal O_{\mathbb P_{(1,2)}}(-D_0)$};
			\node (v2) at (-3,-1) {$\mathcal O_{\mathbb P_{(1,2)}}$};
			\node (v3) at (0.5,-1) {$\mathcal O_{\mathbb P_{(1,2)}}(-D_0)$};
			\node (v5) at (4,-1) {};
			\draw[->]  (v1) edge node[red,midway, fill=white]{$x_0$} (v2);
			\draw[->]  (v1) edge node[blue,midway, fill=white]{$-x_1$}(v3);
			\draw[thick, blue] (0.5,-0.5) -- (0.25,-0.5);
			\node[circle, fill=black, scale=.5] at (0.5,-0.5) {};
			\draw[thick, red] (-3,-0.5) -- (-2.75,-0.5);
			\draw[thick, blue] (-3,-0.5) -- (-3.25,-0.5);
			\node[circle, fill=black, scale=.5] at (-3,-0.5) {};
			\node (v4) at (2.5,-2.5) {$\mathcal O_{\mathbb P_{(1,2)}}(-D_1)$};
			\draw[->]  (v4) edge  node[midway, fill=white]{$1$}(v3);
			\draw[->]  (v4) edge  node[blue,midway, fill=white]{$-x_1$}(v5);
		\end{tikzpicture}
	\]
	To check that this is a resolution of $\mathcal O_e$, we restrict to the $U_0$ and $U_1$ charts. The pullback of this resolution to the $U_1$ chart is \cref{eq:resolutionPtInOrbiA1}, while the pullback of the resolution to the $U_0$ charts is chain homotopic to \cref{eq:resolutionPtInA1}. Therefore, this is a resolution of $\mathcal O_e$ in every toric chart and therefore a resolution of $\mathcal O_e$.
\end{example}

\subsection{Equivariant codimension two complements}
\label{subsec:equivCodim2}
In this section, we will address a few additional technical aspects of toric stacks that we will need to verify that our resolutions behave as expected.

First, we verify that immersions of toric stacks behave well under further quotients by a toric subgroup. Let $(\Sigma, \ul\beta\colon L\to N)$ be a stacky fan. Given a saturated subgroup $N_G\to N$, we obtain a new stacky fan $(\Sigma, \ul\beta'\colon L\to N/N_G)$, along with a morphism of stacky fans $(\ul\phi, \ul\Phi)\colon (\Sigma, \ul\beta)\to (\Sigma, \ul\beta')$. This construction identifies the stack $ [\X_{\Sigma, \ul \beta}/\TT_{G}]$ as $\X_{\Sigma, \ul \beta'}$.

\begin{prop}
	\label{prop:quotientOfSubstack}
	Let $(\ul \phi, \ul \Phi)\colon (\Sigma_Y, \ul \beta_Y\colon L_Y\to N_Y) \to (\Sigma_X, \ul \beta_X\colon L_X\to N_X)$ be an immersion of a closed toric substack inducing a morphism of toric stacks $\phi: \Y \to \X$. Let  $\TT_G\subset \TT_N$ be a toric subgroup corresponding to a saturated sublattice $N_G\subset N$. If $\Y$ is $\TT_G$-invariant, there is an immersion $\phi'$, making the following diagram commute. 
        \[\begin{tikzcd}
	   \Y \arrow{r}{\phi} \arrow{d} & \X \arrow{d} \\
	   \left[ \Y/\TT_{G} \right] \arrow{r}{\phi'} & \left[ \X/\TT_{G}\right]
        \end{tikzcd}\]
\end{prop}
\begin{proof}
For brevity of notation, we set $\X' = [\X/\TT_G]$ and $\Y' = [\Y/\TT_G]$.
	As $\Y$ is a closed toric substack, the maps $\ul\phi, \ul\Phi$ are inclusions. The condition that $\Y$ is a $G$-invariant substack is equivalent to  $N_G\subset \phi(N_Y)$ We therefore obtain a saturated sublattice $\ul\phi^{-1}(N_G)\subset N_Y$. The toric fan $(\Sigma_{\Y'},\ul \beta_{\Y'})$ for  $\Y'$ is defined using the construction preceding the proposition. The map $\ul\phi \colon  N_Y\to N_X $ descends to a map $\ul\phi' \colon N_{Y'}\to N_{X'}$.
   The map $\ul\Phi' \colon  L_{Y'}\to L_{X'} $ is taken so it agrees with $\ul\Phi$. 

   \[\begin{tikzcd}
	L_{\Y'}=L_\Y\arrow{r}{\ul\Phi=\ul\Phi'} \arrow{d}{\ul\beta_Y} \arrow[bend right = 60]{dd}{\ul\beta_{\Y'}} & L_\X=L_{\X'} \arrow{d}{\ul\beta_X} \arrow[bend left= 60]{dd}{\beta_{\X'}}\\
	N_Y \arrow{r}{\ul\phi} \arrow{d}& N_X \arrow{d} \\
	N_{Y'}=N_Y/\ul\phi^{-1}(N_G)\arrow{r}{\ul\phi'} &  N_X/N_G= N_{X'}
\end{tikzcd}\]
It remains to show that $\phi' \colon  \Y'\to \X' $ is an immersion of a closed substack. Since $\ul\Phi'=\ul\Phi$, we only need to check that the map $\ul\phi' \colon  N_{Y'}\to N_{X'} $ is an injection. This follows as $\ul\phi$ is an injection and $\ul\phi(\ul\phi^{-1}(N_G))=N_G$.
\end{proof}
When $X^\circ\into X$ has a complement of codimension at least 2, then the pushforward of a line bundle is a coherent sheaf and is still a line bundle if $X$ is smooth (Hartog's principle). If $G$ acts freely on $X$, and $X^\circ \to X$ is $G$-equivariant, it is clear that the dimension of $(X\setminus X^\circ)/G$ is the relevant quantity to study (as opposed to the codimension of $X^\circ$ in $X$). When $G$ does not act freely, it is not clear to us in general what quantity is appropriate to examine. For our applications, the following criterion will be useful.
\begin{df}
	\label{def:equivariantCodimension2Complement}
	An open inclusion $(\ul i, \ul I)\colon (\Sigma^\circ, \ul\beta^\circ)\to (\Sigma, \ul\beta)$ has \emph{complement of equivariant codimension two} if for all $\rho\in \Sigma(1)$ with $\ul \beta(\rho)\neq 0$, there exists $\rho^\circ\in \Sigma^\circ$ with $\ul I(\rho^\circ)=\rho$. 
\end{df}

\begin{example}
	Let $\Sigma$ be the coordinate fan on $\ZZ^2$ so that $X_\Sigma=\AA^2$, and let $\ul \beta \colon  \ZZ^2\xrightarrow{(1\;0)}\ZZ $. Further, let $\Sigma'$ be the fan for $\AA^1 \times \bb{G}_m$ and $\ul\beta^\circ=\ul\beta$. $\X_{\Sigma^\circ, \ul\beta^\circ}$ is a stabilization of $\AA^1$, and $\X_{\Sigma^\circ, \ul\beta^\circ}\hookrightarrow \X_{\Sigma, \ul \beta}$ is an immersion with complement of equivariant codimension two. Note that in this example $X_\Sigma\setminus X^\circ_\Sigma$ is \emph{not} codimension two, but it is codimension two in the quotient.
\end{example}

An important feature of open inclusions with complements of equivariant codimension two is that we can ``by hand'' define a correspondence from line bundles on $\X^\circ=[X^\circ/G]$ to line bundles on $\X=[X/G]$. This can be thought of as a version of Hartog's principle for equivariant line bundles.
Let $(\ul i, \ul I)\colon (\Sigma^\circ, \ul\beta^\circ)\to (\Sigma, \ul\beta)$ be an open immersion with a complement of equivariant codimension 2 where $\Sigma$ is smooth. Write $I_\Sigma \colon \Sigma^\circ(1)\into \Sigma(1)$ for the map on rays induced by $\ul i$. Given $\sF^\circ \colon  L^\circ\to \RR $ a support function, we obtain a support function $\ul I_\flat\sF \colon  L \to \RR $ whose values on primitive generators for $\Sigma^\circ(1)$ are defined by 
\[\ul I_\flat\sF(u_\rho)= \begin{cases}
	\sF\left(u_{\ul I^{-1}(\rho)}\right) &\text{ if $\rho\in I_\Sigma(\Sigma^\circ(1))$}\\
	0 & \text{otherwise}.
\end{cases}
\]

We now show that this construction comes from a functor between sheaves on $[X^\circ/G]$ and $[X/G]$.
\begin{rem}
	If $X^\circ\subset X$ had codimension two (and not simply equivariant codimension two) then $I_*\colon\Sh(X^\circ)\to \Sh(X)$ takes line bundles to line bundles and $I_*\mathcal O_{X^\circ}(\sF)= \mathcal O_{X}(\ul I_\flat\sF)$. However, there is no reason for this operation to compute the pushforward when $X^\circ$ has codimension 1 complement. 
\end{rem}

\begin{example}
    Consider the toric stack $[X/G]=[\AA^1/\TT^1]$. The inclusion $i: [X^\circ/G]:=[\TT^1/\TT^1]\to[X/G]$ is an example of an inclusion with equivariant codimension two. The pushforward $i_*\mathcal O_{[X^\circ/G]}(0)$ is not a line bundle on $[X/G]$. This can be seen, for instance, by letting $\mathcal F_0$ be the skyscraper sheaf of the point $0 \in \AA^1$ with the trivial group action and computing 
    \[\hom(i_*\mathcal O_{[X^\circ/G]}(0), \mathcal F_0)=0\]
	However, when one takes the $\TT^1$-invariant sections of $\mathcal O_{[X^\circ/G]}(0)$ first, we observe that $\mathcal O^G_{[X^\circ/G]}(0)$ is the constant $\kk$-sheaf. Its pushforward to $[X/G]$ is again the constant $\kk$-sheaf. After tensoring with $\mathcal O_{[X/G]}$, we obtain our desired answer.
\end{example}

Let $\mathcal F$ be a sheaf on $[X^\circ/G]$. For $U\subset X$, define 
\[i_\flat \mathcal F(U):= \mathcal F^{G}(U\cap X^\circ)\tensor_{\mathcal O_{[X^\circ/G]}^G(U\cap X^\circ)} \mathcal O_{[X/G]}(U).\]
For $f\in \hom_{[X^\circ/G]}(\mathcal F, \mathcal G)$, we similarly define the morphism $i_\flat f$ on sections  $x\in i_\flat \mathcal F(U)$ as 
\[i_\flat f(x)= f(y)\tensor r.\]
after writing $x=y\tensor r$ where $y\in \mathcal F^{G}(U\cap X^\circ)$ and $r \in \mathcal{O}_{[X/G]}(U)$.
\begin{lem}\label{lem:exactEquivariantPushforward}
	The functor $i_\flat$ satisfies the following properties.
	\begin{enumerate}
		\item \emph{Intertwines support function:} \label{prop:support}For any line bundle $\mathcal O_{[X^\circ/G]}$ on $[X^\circ/G]$, we have 
		\[i_\flat\mathcal O_{[X^\circ/G]}(\sF)=\mathcal O_{[X/G]}(\ul I_\flat\sF).\]
		\item \emph{Comparison of stalks to orbit closure:} \label{prop:orbits}Let $\pi^*:\Sh([X/G])\to \Sh(X)$ be the forgetful functor.  Let $x_0\in X$ be in the closure of the orbit $G\cdot(x_1)$. Then for $\alpha\in \{0, 1\}$ the stalk functors 
		\[\pi^*\circ i_\flat(-)_{x_\alpha}: \Sh([X^\circ/G])\to \kk-\text{mod}\]
		are isomorphic.
		\item \emph{Exactness:} \label{prop:exactness}The composition $\pi^*\circ i_\flat$ is an exact functor.
		\item \emph{Coherence:} \label{prop:coherence}The image of $\pi^*\circ i_\flat|_{\Coh([X^\circ/G])} \colon \Coh([X^\circ/G]) \to \Sh(X)$ is contained in $\Coh(X)$.
	\end{enumerate}
\end{lem}

\begin{proof}
	We first prove that this is a functor of topological sheaves; this follows as all operations (taking invariant sections, pullback along open inclusion, and tensor product) are functors of topological sheaves. By a similar argument, $i_\flat \mathcal F(U)$ is an $\mathcal O_{[X/G]}$ module, and $i_\flat f$ is a morphism of $\mathcal O_{[X/G]}$ modules.

	We now check \cref{prop:support} by exhibiting an isomorphism of sheaves $i_\flat\mathcal O_{[X^\circ/G]}(\sF)=\mathcal O_{[X/G]}(\ul I_\flat\sF)$.
	 Let $[X_\sigma/G]$ be a smooth stacky chart for $X$, with the  1-dimensional cones of $\Sigma_\sigma$ generated by $\{u_{\rho_i}\}_{i=1}^k$.
	  Then a section $s\in \mathcal O_{[X^\circ G]}^{G}(\sF)(U\cap X^\circ)$ can be written as a sum of monomials $s=\sum_{m'\in M'} a_{m'} \chi^{\ul\beta^*m'}$, where the $k'=\ul\beta^*m'$ satisfy the following.
	\begin{itemize}
		\item Because the section is $G$-invariant, $\langle k', l'\rangle =0$ for all $l'\in \ker(\ul \beta^\circ)$. Further, $(\Sigma_\sigma(1)\setminus \Sigma^\circ(1))\subset \ker(\ul\beta^\circ)$ implies that $\langle u_{\rho_i}, k'\rangle \geq 0$ for all $\rho_i\in \Sigma_\sigma(1)\setminus \Sigma^\circ$. 
		\item For each $\rho_i\in \Sigma_\sigma(1)\cap \Sigma^\circ(1)$, we have the relation $\langle u_{\rho_i}, k'\rangle \geq \sF(u_{\rho_i})$.
	\end{itemize}
	It follows that for all $u_{\rho_i}\in \Sigma_\sigma(1)$ we have $\langle u_{\rho_i}, \ul\beta^*m  \rangle \geq \sF(u_{\rho_i})$. Therefore, $s$ corresponds to a section of $\mathcal O_{[X/G]}^G(\ul I_\flat\sF)(U)$. The result follows from 
	\begin{align*}
		\mathcal O_{[X^\circ/G]}^G(\ul I_\flat\sF)(U\cap X^\circ)\tensor_{\mathcal O_{[X^\circ/G]}^G(U\cap X^\circ)}\mathcal O_{[X/G]}(U)
		=&\mathcal O_{[X/G]}^G(\ul I_\flat\sF)(U)\tensor_{\mathcal O_{[X/G]}^G(U)}\mathcal O_{[X/G]}(U)\\
		=&\mathcal O_{[X^\circ/G]}(\ul I_\flat\sF)(U)
	\end{align*}

	We now prove \cref{prop:orbits} which allows us to compare the stalks of $(\pi^*\circ i_\flat)\mathcal F$ over $X\setminus X^\circ$ to stalks in $X^\circ$. Pick $x_0\in X$ and $x_1\in X^\circ$ so that $x_0$ is in the closure of $G\cdot x_1$. Because $x_0$ is in the closure of $G\cdot x_1$, we can take a sequence of subgroups $G_0\supset G_1\supset\cdots$ of $G$ and sequence of open sets $V_i\ni x_1$ satisfying:  
 \begin{itemize}
  \item for all $i$, there exists $U_i\ni x_0$ an open subset of $X$ so that $U_i\cap X^\circ = \bigcup_{g\in G_i}V_i$ and;
  \item for every $U\ni x_0$ and open subset of $X$ there exists $j$ sufficiently large so that for all $i>j$  
  \[\bigcup_{g\in G_i} g\cdot V_i\subset U\cap X^0.\]
  \end{itemize}
  We then compute the stalk 
 \begin{align*}
	(\pi^*\circ  i_\flat F)_{x_0}=& \lim_{U\ni x_0} (\pi^*\circ i_\flat  \mathcal F)(U)\\
	=& \lim_{i\to\infty} \mathcal F^{G}(U_i\cap X^0)\tensor_{\mathcal O^G_{[X^\circ/G](U_i\cap X^\circ)}}\mathcal O_X(U_i)\\
	\cong& \lim_{i\to\infty} \mathcal F^{G}	\left(\bigcup_{g\in G_i} g\cdot V_i\right)\tensor_{\mathcal O^G_{[X/G]}(\bigcup_{g\in G_i} g\cdot V_i)}\mathcal O_X(U_i)\\
	\intertext{Since $\mathcal F^{G}$ and $\mathcal O^G_{[X^\circ/G]}$ are sheaves of $G$ invariant sections,}
	\cong & \lim_{i\to\infty} \mathcal F^{G}	\left(V_i\right)\tensor_{\mathcal O^G_{[X^\circ/G]}(V_i)}\mathcal O_X(U_i)\\
	=&\mathcal F^G_{x_1}\tensor_{(\mathcal O^G_{[X^\circ /G]})_{x_1}}\tensor \mathcal (O_X)_{x_0}\\
	\cong& \mathcal F^G_{x_1}\tensor_{(\mathcal O^G_{[X^\circ /G]})_{x_1}}\tensor \mathcal (O_X)_{x_1}\cong (\pi^*\circ  i_\flat F)_{x_1}
 \end{align*}

 From \cref{prop:orbits}, it follows that $(\pi^*\circ i_\flat)$ is exact, as exactness can be computed on stalks, and for points $x_1\in X\setminus X^\circ$ we have an isomorphism of stalk functors 
 \[(\pi^*\circ i_\flat)(-)_{x_1}= (I_*\circ (\pi^\circ)^*)(-)_{x_1}.\]
 The latter functor is exact, as $(\pi^\circ)^*$ is an equivalence of categories, $I_*$ is an open inclusion, and $x_1$ is in the image of $I_*$.
 
Finally, it remains to show \cref{prop:coherence}. Since $\pi^*\circ i_\flat$ is exact, and sends $\mathcal O_{[X^\circ/\TT_Y]}$ to $\mathcal O_X$, it sends coherent sheaves to coherent sheaves.
\end{proof}

\begin{rem}
    In a discussion with Mahrud Sayrafi, it was observed that the properties of $\pi^*\circ i_\flat$ we exploit are reminiscent of properties of functors between sheaves on Mori dream spaces  \cite{hu2000mori}.
\end{rem}

We will make use of \cref{lem:exactEquivariantPushforward} in conjunction with the following local construction of an open inclusion with equivariant codimension two complement from a toric subvariety. 

\begin{prop}
	\label{prop:quotientOfAffineSubstacks}
	Let $\phi \colon  Y\to \AA^n $ be a toric subvariety. Assume that $\kk$ has characteristic zero.
	There exists $X^\circ \subset \AA^n$ a toric subvariety such that
	\begin{itemize}
		\item $Y|_{X^\circ}=\TT_Y$ is an algebraic subtorus of $\mathbb{G}_m^n$,
		\item The action of $\TT_Y$ on $X^\circ$ is free,
		\item $[X^\circ/\TT_Y]$ is covered by smooth stacky charts,\footnote{This is the only property which requires the assumption of $\text{char}(\kk)=0$ to prove.} and
		\item $[X^\circ/\TT_Y]\into [\AA^n/\TT_Y]$ has a complement of equivariant codimension two.
	\end{itemize}
\end{prop}
	We represent this data in the form of the following diagram.
	\[
		\begin{tikzcd}
			e \arrow{d}{\phi_e}  & \mathbb{T}_Y  \arrow{l}  \arrow{d}{\phi^\circ}  \arrow{r} & Y \arrow{d}{\phi} \\
			\; [X^\circ/\mathbb{T}_Y] 	\arrow[bend right]{rrr}{i} &  X^\circ \arrow{r}{I} 	\arrow{l}{\pi^\circ}		  & \mathbb{A}^n \arrow{r}{\pi} & \;[\mathbb{A}^n/\mathbb{T}_Y]
		\end{tikzcd}
	\]
\begin{proof}
	Let $\Sigma$ be the coordinate fan so $X_\Sigma=\AA^n$. Consider $\underline \pi \colon  L_{\AA^n}\to L_{\AA^n}/L_Y $. Then, define  $\Sigma^\circ:=\{\sigma\in \Sigma\st \underline \pi|_{\sigma} \text{ is injective}\}$. Suppose that $\sigma\in \Sigma^\circ$. Since $\underline \pi|_{\sigma}$ is injective, $\underline \pi|_{\tau}$ with $\tau<\sigma$ is injective. Therefore, $\Sigma^\circ$ is a fan. 
	
	Observe that every cone of $Y$ will be in the kernel of $\underline \pi$, so $Y|_{X^\circ}=\TT_Y$. Additionally, since every cone of $\Sigma^\circ$ is disjoint from $L_Y$, the action of $\TT_Y$ on the corresponding toric orbit has trivial stabilizers. It follows that $\TT_Y$ acts freely on $X^\circ$.

    The toric stack $[X^\circ/\mathbb{T}_Y]$ is defined via the data of $(\Sigma^\circ, \underline \pi)$. The inclusion $(\Sigma^\circ,\ul \pi)\into (\Sigma,\ul \pi)$ is an open immersion. By construction, it has a complement of equivariant codimension 2.

    We now show that it is covered by smooth stacky charts indexed by the cones $\Sigma^\circ(n-k)$, where $k=\dim(Y)$.
	Given any $\tau\in \Sigma^\circ$, let $\rho_1, \ldots, \rho_m$ be the 1-dimensional cones belonging to $\tau$. Then $m\leq (n-k)$. Suppose that $m<(n-k)$. Let $\rho_{m+1}, \ldots \rho_n$ be the remaining cones of $\Sigma(1)$ (recall that this is the standard coordinate fan). Look at the subspace $\RR\cdot \ul\pi(\tau)$ of $L/L_Y$ which is generated by the image of $\tau$. This is an $m$-dimensional subspace of $L/L_Y$, which has dimension $(n-k)$.  Since the map $\ul\pi: L\to L/L_Y$ is surjective, at least one of the $\ul\pi(\rho_{i})$ with $i\geq m+1$ must disjoint from $\RR\cdot \ul\pi(\tau)$ (otherwise the image of $\ul\pi$ would be contained in $\RR\cdot\ul\pi(\tau)$). Let $\tau'$ be the cone spanned by $\tau$ and this particular $\rho_i$. It follows that $\dim(\RR\ul\pi(\tau'))=m+1$, i.e., $\ul\pi|_{\tau'}$ is injective. From this we conclude that every cone $\tau\in \Sigma^\circ$ is contained within some $\sigma\in \Sigma^\circ(n-k)$. 
 
    It remains to construct for each $\sigma\in \Sigma^\circ(n-k)$ a smooth stacky chart that covers this cone.  Let $\Sigma_{\sigma}$ be the fan of cones of $\Sigma$ subordinate to $\sigma$. Then $(\Sigma_\sigma, \ul\pi)\into (\Sigma^\circ, \ul\pi)$ covers $\sigma$.  We claim that $(\Sigma_\sigma, \ul\pi)$ is the stabilization of a smooth stacky chart.
    Consider the smooth stacky fan 
       $(\Sigma_{\sigma}, \ul \beta_\sigma)$ where $L_\sigma$ is the sublattice spanned by $\sigma$, $\Sigma_{\sigma}$ is the induced coordinate fan on $L_\sigma$, and 
	   $\ul \beta_\sigma\colon L_{\sigma} \to N$ is given by restriction of $ \ul \pi|_{L_\sigma}$. Since $\sigma$ is a cone of the coordinate fan for $L$, we have a splitting $L=L_\sigma\oplus L_{\sigma}^\bot$. Therefore, $(\Sigma_\sigma, \ul\pi)$ is a stabilization of the stacky chart $(\Sigma^\circ_\sigma, \beta)$.
\end{proof}
We do not have a good characterization for when the third property of \cref{prop:quotientOfAffineSubstacks} holds over fields of positive characteristic.
\begin{df}
    We say that $\phi\colon Y\to \AA^n$ is $\kk$-admissible if 
	\cref{prop:quotientOfAffineSubstacks} holds.\label{def:admissibleOverK} We say that a toric substack $\phi\colon \Y \to \X$ is $\kk$-admissible if $\X$ is covered by stacky charts, and if in every smooth stacky chart $[\AA^n/G]$ on $\X$ the lift (see \cref{lem:pushforwardquotient}) of $\Y \cap [\AA^n/G]$ in $\AA^n$ is $\kk$-admissible.
\end{df}
\begin{example}
    If $\X$ is a toric stack covered by smooth stacky charts (i.e. at each chart $[\AA^n/G]$ the order of $G$ does not divide the characteristic of $\kk$), then the diagonal in $\X\times \X$ is $\kk$-admissible. 
\end{example}
\subsection{Returning to resolutions of toric subvarieties in \texorpdfstring{$\AA^2$}{the plane}}
 Now that we have established the background on toric stacks and the tools needed to apply our general strategy, we conclude this section with two more examples of putting those tools to work.

\begin{example}[Resolving $z_1^2=z_2$ equivariantly]
	\label{ex:equivariantResolutionOfParabola}
	Take the parabola $\phi \colon \AA^1 \to \AA^2 $ parameterized by $z\mapsto (z, z^2)$. Let $i\colon \AA^2\setminus \{0\}\to \AA^2$ be the inclusion. Consider the action of $\TT_\beta\curvearrowright\AA^2$ given by $(z_1, z_2)\mapsto (tz_1, t^2z_2)$. We can take the quotient $\pi_\beta \colon \AA^2\setminus \{0\}\to [\AA^2\setminus\{(0,0)\}/\TT_\beta]$. There is a resolution $C_\bullet(S^{\phi_e}, \mathcal O^{\phi_e})$ for a point on the quotient stack so that $\pi^*i_\flat C_\bullet(S^{\phi_e}, \mathcal O^{\phi_e})$ is
	
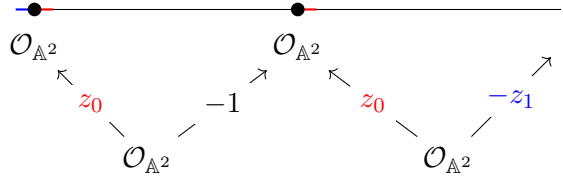
\[
		\begin{tikzpicture}
			\draw (-3,-0.5) -- (4,-0.5);
			\node (v1) at (-1.5,-2.5) {$\mathcal O_{\AA^2}$};
			\node (v2) at (-3,-1) {$\mathcal O_{\AA^2}$};
			\node (v3) at (0.5,-1) {$\mathcal O_{\AA^2}$};
			\node (v5) at (4,-1) {};
			\draw[->]  (v1) edge node[red,midway, fill=white]{$z_0$} (v2);
			\draw[->]  (v1) edge node[midway, fill=white]{$-1$}(v3);
			\draw[thick, red] (0.5,-0.5) -- (0.75,-0.5);
			\node[circle, fill=black, scale=.5] at (0.5,-0.5) {};
			\draw[thick, red] (-3,-0.5) -- (-2.75,-0.5);
			\draw[thick, blue] (-3,-0.5) -- (-3.25,-0.5);
			\node[circle, fill=black, scale=.5] at (-3,-0.5) {};
			\node (v4) at (2.5,-2.5) {$\mathcal O_{\AA^2}$};
			\draw[->]  (v4) edge  node[red, midway, fill=white]{$z_0$}(v3);
			\draw[->]  (v4) edge  node[blue,midway, fill=white]{$-z_1$}(v5);
		\end{tikzpicture}
	\]
	which is homotopic to our original resolution from \cref{ex:nonequivParaResolution}. Note that the line bundles used in all phases of the construction are in a Thomsen collection.
\end{example}

\begin{example}[Resolving $z_0z_1=1$]
	\label{ex:resolutionOfHyperbola}
	Take the toric subvariety $\phi_{(-1, 1)} \colon  \mathbb{G}_m \to \AA^2 $ parameterized by $z\mapsto (z, z^{-1})$. We now look at the map $\pi_{(-1, 1)} \colon  \AA^2\setminus\{(0,0)\}\to X_\beta:=[\AA^2\setminus\{(0,0)\}/\TT_\beta] $; the latter is the toric stack corresponding to the non-separated line. Then applying $\pi^*i_\flat$ to our resolution from \cref{eq:resolutionPtInNonSepLine} of $\mathcal O_e$ is the standard resolution of $z_0z_1=1$
	
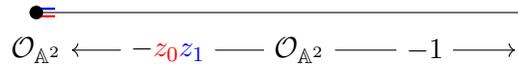
\[
		\begin{tikzpicture}
			\draw (-2.5,-0.5) -- (4,-0.5);
			\node (v1) at (1,-1) {$\mathcal O_{\AA^2}$};
			\node (v2) at (-2.5,-1) {$\mathcal O_{\AA^2}$};
			\node (v3) at (4,-1) {};
			\draw[->]  (v1) edge node[ fill=white]{$-{\color{red}z_0}{\color{blue}z_1}$} (v2);
			\draw[->]  (v1) edge node[midway, fill=white]{$-1$}(v3);
			\draw[thick, red] (-2.5,-0.55) -- (-2.25,-0.55);
			\draw[thick, blue] (-2.5,-0.45) -- (-2.25,-0.45);
			\node[circle, fill=black, scale=.5] at (-2.5,-0.5) {};
		\end{tikzpicture}
	\]
\end{example}

\section{Resolution of toric substacks} \label{sec:Resolutionideas}
This section serves to give the broad strokes of the proof of \cref{mainthm:res} reducing it to a variety of lemmas that will be addressed in \cref{sec:Resolutionproofs}. In fact, we break the proof into steps of increasing difficulty working from points in toric varieties to substacks of stacks covered by smooth stacky charts. Finally, it is important to note that the resolutions appearing in \cref{mainthm:res} are explicitly defined in \cref{subsec:defOfCbullet}.

\subsection{Resolving points in  toric varieties} \label{subsec:strategy}
We first lay out the steps for resolving the identity point in a smooth toric variety.

\begin{thm}
	\label{thm:resolutionOfPoints}
	Let $X$ be a toric variety covered by smooth charts.
	Let $\phi \colon  e\to X $ be the inclusion of the identity point.
	There exists a resolution $C_\bullet(S^\phi, \mathcal O^\phi)\xrightarrow \epsilon \phi_*\mathcal O_e$ on $X$, where $C_{k}(S^\phi, \mathcal O^\phi)$ is a direct sum of line bundles from the Thomsen collection. Furthermore, there is an epimorphism  $\alpha^\phi \colon \mathcal O_{\X}\to C_0(S^\phi, \mathcal O^\phi) $, and $C_k(S^\phi, \mathcal O^\phi)=0$ whenever $k<0$ or $k> \dim(X)$.
\end{thm}
 In \cref{subsec:pathsandsheaves}, we describe the (a priori not exact) complex of sheaves $C_{\bullet}(S^\phi, \mathcal O^\phi)$, which is defined combinatorially from the morphism of stacky fans $(\ul \phi, \ul \Phi)\colon (\Sigma_Y, \ul \beta_Y)\to (\Sigma_X, \ul \beta_X)$ whenever $\X$ is covered by smooth stacky charts; the version needed for \cref{thm:resolutionOfPoints} comes from specializing $\Y=e$ and $\X=X$.
 From the construction of $C_{\bullet}(S^\phi, \mathcal O^\phi)$, it is immediate that there is a morphism  $\alpha^\phi \colon \mathcal O_{\X}\to C_0(S^\phi, \mathcal O^\phi) $ and  $C_k(S^\phi, \mathcal O^\phi)=0$ whenever $k<0$ or $k> \dim(\X)-\dim(\Y)$.
 Before we get to the definition of $C_{\bullet}(S^\phi, \mathcal O^\phi)$ in \cref{subsec:pathsandsheaves}, we list some of the properties that this complex will satisfy. We state the lemmas for the general setting of toric substacks, as we will use them later. Proofs are delayed until \cref{sec:Resolutionproofs}.

\begin{lem}[$C_{\bullet}(S^\phi, \mathcal O^\phi)$ respects products]
    \label{lem:kunneth}
    Let $\phi_1 \colon  \Y_1\to \X_1 $ and $\phi_2 \colon  \Y_2\to \X_2 $ be two immersed closed substacks. Then $C_\bullet(S^{\phi_1\times \phi_2}, \mathcal O^{\phi_1\times \phi_2})=C_\bullet(S^{\phi_1}, \mathcal O^{\phi_1})\tensor C_\bullet(S^{\phi_2}, \mathcal O^{\phi_2})$
\end{lem}

 Recall that given a ray $\rho \in \Sigma(1)$, the star of $\rho$ is $\str(\rho) = \{ \sigma \in \Sigma : \rho < \sigma \}$. Removing $\str(\rho)$ from $\Sigma$ corresponds to taking the complement of the divisor corresponding to $\rho$. We then prove that $C_{\bullet}(S^\phi, \mathcal O^\phi)$ transforms under deletion of a toric boundary divisor in the following way.
 
\begin{lem}[Functoriality along restrictions up to homotopy] 
	Let $\phi: \Y_{\Sigma', \ul \beta'} \into \X_{\Sigma, \ul \beta}$ be an immersion of a closed toric substack.
	For any $\rho \in \Sigma(1)$, let $\Sigma_\rho=\Sigma \setminus \str(\rho)$ giving us the toric stack $\X_\rho:=\X_{\Sigma_{\rho},\ul \beta}$ and let $\Sigma'_{\rho}=\Sigma'\setminus \{\alpha \st \phi(\alpha)\cap \str(\rho)\neq \emptyset\},$ giving us the toric stack $\Y_\rho:= \Y_{\Sigma_{\rho}', \ul \beta'}$ fitting into the diagram
	      \[\begin{tikzcd}
			     \Y_\rho \arrow{r}{\phi_\rho} \arrow{d}{i_{\rho}'}& \X_\rho \arrow{d}{i_\rho}\\
			      \Y\arrow{r}{\phi} & \X
		      \end{tikzcd}
	      \]
    where $\phi_\rho$ is induced by the same morphism of stacky fans as $\phi$ and $i_\rho, i_\rho'$ are the open inclusions. Suppose that $\X_{\Sigma_{\rho},\ul\beta}$ is covered by smooth stacky charts.
	      There is a homotopy equivalence $\HE\colon i^*_\rho C_\bullet(S^\phi, \mathcal O^\phi)\to C_\bullet(S^{\phi_\rho}, \mathcal O^{\phi_\rho})$. Furthermore, the diagram
            \begin{equation} \label{eq:extrarefdiagram} \begin{tikzcd}
			     \mathcal O_{\X} \arrow{r}{\alpha^\phi} \arrow{dd} & C_0(S^\phi, \mathcal O^\phi) \arrow{d}\\
                & i_\rho^* C_0 (S^\phi, \mathcal O^\phi) \arrow{d}{\Psi} \\
			      \mathcal O_{\X_\rho} \arrow{r}{\alpha^{\phi \rho}} & C_0(S^{\phi_\rho}, \mathcal O^{\phi_\rho})
		      \end{tikzcd}
	        \end{equation}
       commutes where the left vertical arrow and upper right vertical arrow are the natural restriction morphisms. 
		  \label{lem:functorialityRestriction}
\end{lem}
Given these two lemmas, we can prove \cref{thm:resolutionOfPoints}.
\begin{proof}[Proof of \cref{thm:resolutionOfPoints}] The theorem is proved in two steps:
\begin{description}
        \item[Step 1: Resolution of points in charts.]
 By \cref{ex:pointInA1}, the theorem holds for $\phi \colon e\to \AA^1 $ and furthermore $C_{\bullet}(S^\phi, \mathcal O^\phi)$ is the Koszul resolution. The case $\phi \colon  e\to \AA^n $ follows from applying \cref{lem:kunneth}.
	\item[Step 2: Restricting to smooth toric charts.] Let $\Y=\{e\}$ and let $X$ be a toric variety which can be covered with smooth toric charts $\{X_\sigma\}_{\sigma\in \Sigma(n)}$.  Let $j_\sigma\colon X_{\sigma}\to X$ be the corresponding inclusion of a coordinate chart. For each $\sigma$, let  $\rho_{\sigma,1}, \ldots, \rho_{\sigma,k}\in \Sigma(1)$ be the 1-dimensional cones which are disjoint from $\sigma$.
	Consider the iterated inclusion 
	\[\begin{tikzcd}
		e \arrow{d}{\phi}& e \arrow{d}\arrow{l}{i_{\rho_{\sigma, 1}'}} & \arrow{l}\cdots &\arrow{l}{i_{\rho_{\sigma, k}'}} e \arrow{d}{\phi_{\rho_{\sigma, 1},\ldots,\rho_{\sigma, k}}} \arrow[bend right]{lll}{j'_\sigma}\\
		X  & X_\rho  \arrow{l}{i_{\rho_{\sigma, 1}}}  & \arrow{l} \cdots & X_{\rho_{\sigma, 1},\rho_{\sigma, 2},\ldots,  \rho_{\sigma, k}} = \arrow{l}{i_{\rho_{\sigma, k}}} X_{\sigma} \arrow[bend left]{lll}{j_\sigma}\\
	 \end{tikzcd}
	 \]
	 where $X_{\rho_{\sigma,1}, \ldots, \rho_{\sigma, l}}$ is the toric stack whose fan consists of cones which do not contain rays $\rho_{\sigma,1}, \ldots, \rho_{\sigma, l}$. 
	 By repeated application of \cref{lem:functorialityRestriction}, we obtain that $j_\sigma^*C_\bullet(S^\phi, \mathcal O^\phi)$ is chain homotopic to $C_\bullet(S^{\phi_{\rho_1,\ldots,\rho_k}},\mathcal O^{\rho_1,\ldots,\rho_k})$. The latter is the Koszul resolution of the skyscraper sheaf of $\{e\}$ by the previous step.  
	 This implies that $C_\bullet(S^\phi, \mathcal O^\phi)$ resolves $\phi_*\mathcal O_{e|_{X_{\sigma}}}$ on every smooth toric chart. Since these charts cover $X$, it is a resolution of $\phi_*\mathcal O_e$.
\end{description}
\end{proof}
\subsection{Resolving points in toric stacks}\label{subsec:pointsinstacks}

We will now generalize \cref{thm:resolutionOfPoints} to toric stacks. Note that we use this generalization even to resolve some smooth toric subvarieties of a smooth toric variety (see, for instance, \cref{ex:equivariantResolutionOfParabola,ex:resolutionOfHyperbola}). 

\begin{specialThm}
	\hypertarget{thm:resolutionOfPointsInStacks}{}
	 Let $\X$ be a toric stack which is covered by smooth stacky charts.
Let $\phi \colon  e\to X $ be the inclusion of the identity point.
	There exists a resolution $C_\bullet(S^\phi, \mathcal O^\phi)\xrightarrow \epsilon \phi_*\mathcal O_e$ on $\X$, where $C_{k}(S^\phi, \mathcal O^\phi)$ is a direct sum of line bundles from the Thomsen collection. Furthermore, there is a map   $\alpha^\phi \colon \mathcal O_{\X}\to C_0(S^\phi, \mathcal O^\phi) $ which is an epimorphism in homology, and $C_k(S^\phi, \mathcal O^\phi)=0$ whenever $k<0$ or $k> \dim(\X)$.
\end{specialThm}

 We will need an additional lemma in this setting. 
 
\begin{lem}[Pushforward functoriality along finite group quotients]
	Let $\pi_X \colon \X\to \X'=[\X/G]$ be a finite group quotient and let $\phi'\colon\Y'\to \X'$ be an inclusion of a toric substack. Suppose that $\X$ is covered with smooth stacky charts. Then there exists $\Y\to \X$ an inclusion of a toric substack and $\pi_Y\colon\Y\to \Y'$ a finite group quotient so that the diagram  
	\[
		\begin{tikzcd}
			\mathcal{Y} \arrow{d}{\pi_Y} \arrow{r}{\phi}&  {\X} \arrow{d}{\pi_X}\\
			\mathcal{Y}'\arrow{r}{\phi'} &  {\X'} 
		\end{tikzcd}.
	\]
	commutes.\footnote{This is the pullback in the category of toric stacks with toric morphisms.} Additionally assume $|G|\in \kk^\times$ and  that $\X$ is covered with smooth stacky charts.
	Observe that $(\phi'\circ \pi_Y)_*\mathcal O_\Y=\bigoplus_{q\in \ker(\tilde \pi_Y)}(\phi'_*)\mathcal O_{\Y'}(\bF'(q))$. The pushforward $(\pi_X)_*C_\bullet(S^\phi, \mathcal O^\phi)$ is isomorphic to $\bigoplus_{q\in \ker(\tilde \pi_Y)}C_{\bullet}(S^{\phi'},\mathcal O^{\phi'})\tensor \mathcal O_{\X'}(\bF'(q))$ with the augmentation map splitting across the decomposition.
	\label{lem:pushforwardquotient}
\end{lem}
The proof of \href{thm:resolutionOfPointsInStacks}{Theorem $3.1'$} proceeds in a similar way to the proof of \cref{thm:resolutionOfPoints}. 
\begin{description}
    \item[Step 1\'{}: Resolution of points in smooth stacky coordinate charts]
   We show that  when $\X$ is a smooth stacky coordinate chart (\cref{def:toricStackyChart}) and $\phi \colon  e\to \X $ is the inclusion of the identity that $C_\bullet(S^\phi, \mathcal O^\phi)$ is a resolution of $\mathcal O_e$.

	As $\X$ is a smooth stacky coordinate chart, $\X$ is a finite quotient $\X=[\AA^n/G]$. Let $\pi_G \colon  \AA^n\to \X $, and $\phi' \colon  e\to \AA^n$.
	The statement then follows from our earlier calculation on $\AA^n$ and \cref{lem:pushforwardquotient}. By Nakayama's isomorphism, $(\pi_X)_* C_\bullet(S^\phi, \mathcal O^\phi) =\bigoplus_{q\in \ker(\tilde \pi_Y)}C_{\bullet}(S^{\phi'},\mathcal O^{\phi'})\tensor \mathcal O_{\Y'}(\bF'(q))$ is an exact resolution of $(\phi' \circ \pi_Y)_* \mathcal O_\Y = \bigoplus_{q\in \ker(\tilde \pi_Y)}(\phi'_*)\mathcal O_{\Y'}(\bF(q))$. Since the augmentation respects the splitting, we conclude $C_\bullet(S^{\phi'}, \mathcal O^{\phi'})$ resolves $\phi_*'\mathcal O_Y$.

    \item[Step 2\'{}: restricting to smooth stacky charts] We can apply exactly the same argument as in Step 2 above simply by replacing $X$ with $\X$, and smooth chart with smooth stacky chart.
\end{description}

\subsection{Resolution of toric substacks}\label{subsec:stacksinstacks}

We now move to the most general case that we handle here -- resolving the structure sheaf of a closed toric substack in a toric stack covered by smooth stacky charts.

\begin{thm}
	\label{thm:resolutionOfSubstacks}
	Let $\X$ be a toric stack that can be covered by smooth stacky charts.
	Let $\phi \colon  \Y\to \X $ be an immersion of a closed toric substack.
	There exists a resolution $C_\bullet(S^\phi, \mathcal O^\phi)\xrightarrow \epsilon \phi_*\mathcal O_\Y$ on $\X$, where $C_{k}(S^\phi, \mathcal O^\phi)$ is a direct sum of line bundles from the Thomsen collection. Furthermore, there is a map  $\alpha^\phi \colon \mathcal O_{\X}\to C_0(S^\phi, \mathcal O^\phi)$ which is an epimorphism on homology, and $C_k(S^\phi, \mathcal O^\phi)=0$ whenever $k<0$ or $k> \dim(\X)-\dim(\Y)$.
\end{thm}
To prove this more general version of the theorem, we add a few more lemmas to our repertoire.
\begin{lem}[Pullback functoriality along toric quotients]
	\label{lem:pullbackquotient}
	 Suppose that we have a subtorus  $\TT\subset \TT_N$ acting on $\X$, and that $\Y$ is $\TT$-equivariant. Additionally, suppose that the action of $\TT$ has trivial stabilizers. From \cref{prop:quotientOfSubstack}, we obtain the diagram:
	      \[
		      \begin{tikzcd}
			     \; \mathcal Y \arrow{d}{\pi_\TT} \arrow{r}{\phi}&  \mathcal X \arrow{d}{\pi_\TT}\\
			      \; [\mathcal Y/\TT]  \arrow{r}{\phi/\TT} &\;  [\mathcal X/\TT]
		      \end{tikzcd}
	      \]
	Then, $\pi^*_\TT C_\bullet(S^{\phi/\TT}, \mathcal O^{\phi/\TT})= C_\bullet(S^\phi, \mathcal O^\phi)$, and $\pi^*_\TT\circ \alpha^{\phi/\TT}= \alpha^\phi$
\end{lem}

 \begin{cor}[Invariance under stabilization] Consider maps $(\ul \phi, \ul \Phi)\colon(\Sigma_Y, \ul \beta_Y)\to (\Sigma_X, \ul \beta_X)$ and $(\ul \phi', \ul \Phi')\colon(\Sigma_{\Y'}, \ul \beta_{\Y'})\to (\Sigma_{\X'}, \ul \beta_{\X'})$  which are morphisms of stacky fans which are related by stabilization. Then the complexes of sheaves $C_{\bullet}(S^\phi, \mathcal O^\phi)$ and $C_{\bullet}(S^{\phi'}, \mathcal O^{\phi'})$ are canonically isomorphic.
 \end{cor}
 
\begin{lem}[Functoriality along embeddings with equivariant codimension 2 complement]
	\label{lemma:functorialityCodim2}
	Suppose we have a diagram
	\[
			\begin{tikzcd}
				\; [Y\cap X^\circ/\mathbb T] \arrow{d}{\phi_{[Y/\mathbb T]}}  & Y\cap X^\circ \arrow{l}  \arrow{d}{\phi^\circ}  \arrow{r} & Y \arrow{d}{\phi} \\
				\; [X^\circ/\mathbb{T}] 	\arrow[bend right]{rrr}{i} &  X^\circ \arrow{r}{I} 	\arrow{l}{\pi^\circ}		  &  X \arrow{r}{\pi} & \;[X/\mathbb{T}]
			\end{tikzcd}
	\]
	where the image of $Y$ is invariant under $\TT$ and $i$ is an inclusion of  equivariant codimension two (in the sense of \cref{def:equivariantCodimension2Complement}). Then $\pi^*i_\flat C_\bullet(S^{\phi_{[Y/\mathbb T]}}, \mathcal O^{\phi_{[Y/\mathbb T]}})= C_\bullet(S^\phi, \mathcal O^\phi)$.
\end{lem}		
\begin{proof}[Proof of \cref{thm:resolutionOfSubstacks}] As before, we want to be able to resolve toric substacks in local models, and restrict to that setting.
\begin{description}
	\item[Step 3: Resolution of $\kk$-admissible toric subvarieties of smooth charts]  Let $X=\AA^n$ be a smooth toric chart, and let $Y$ be a closed $\kk$-admissible toric subvariety (\cref{def:admissibleOverK}).	
	By \cref{prop:quotientOfAffineSubstacks}, we can construct $I\colon X^\circ \into X$ so that $[X^\circ/\TT_Y]$ is covered by smooth stacky charts, and $Y\cap X^\circ=\TT_Y$. We reprint the diagram of stacks here for readability 
\[
		\begin{tikzcd}
			e \arrow{d}{\phi_e}  & \mathbb{T}_Y  \arrow{l}  \arrow{d}{\phi^\circ}  \arrow{r} & Y \arrow{d}{\phi} \\
			\; [X^\circ/\mathbb{T}_Y] 	\arrow[bend right]{rrr}{i} &  X^\circ \arrow{r}{I} 	\arrow{l}{\pi^\circ}		  & \mathbb{A}^n \arrow{r}{\pi} & \;[\mathbb{A}^n/\mathbb{T}_Y]
		\end{tikzcd}
\]
Let $\phi_e \colon  \{e\} \to [X^\circ/\TT_Y] $ be the inclusion of the identity point. By \href{thm:resolutionOfPointsInStacks}{Theorem $3.1'$}, we have  $C_\bullet(S^{\phi_e}, \mathcal O^{\phi_e})$ which resolves $(\phi_e)_*\mathcal O_{e}$. 
By \cref{lemma:functorialityCodim2}, we obtain \[\pi_*i_\flat C_\bullet(S^{\phi^\circ}, \mathcal O^{\phi^\circ})=C_\bullet(S^\phi, \mathcal O^\phi).\]  
For any point $x_0\in \AA^n$, there exists a point $x_1\in X^\circ$ so that $x_0$ is in the closure of $\TT_Y \cdot x_1$.  Then by \cref{lem:exactEquivariantPushforward}, we can compute the stalks of $H_k(C_\bullet(S^\phi, \mathcal O^\phi))$ on $\AA^n$ by comparing them to stalks of points in $X^\circ$. We obtain:
 \[H_k(C_\bullet(S^\phi, \mathcal O^\phi))_{x_0}= H_k(C_\bullet(S^\phi, \mathcal O^\phi))_{x_1}= H_k( C_\bullet(S^{\phi^\circ}, \mathcal O^{\phi^\circ}))_{i^{-1}\pi(x_1)}.\]
Since $C_\bullet(S^{\phi^\circ}, \mathcal O^{\phi^\circ})$ is a resolution of $\mathcal O_{\TT_Y}$, we learn that 
\begin{itemize}
	\item $C_\bullet(S^\phi, \mathcal O^\phi)$ is a resolution, i.e., it is exact except at $k=0$;
	\item  As the closure of the $\TT_Y$ orbit in $X$ is $Y$, we obtain that $H_0(C_\bullet(S^\phi, \mathcal O^\phi))$ has the same stalks as $\phi_*\mathcal O_Y$. We have a morphism $\alpha\colon \mathcal O_{\AA^n}\to H_0(C_\bullet(S^\phi, \mathcal O^\phi))$. By another application of \cref{lem:exactEquivariantPushforward}, this morphism is an epimorphism on stalks for all $x\in Y$ and is zero otherwise.  We conclude that $H_0(C_\bullet(S^\phi, \mathcal O^\phi))\cong \phi_*\mathcal O_Y$. 
\end{itemize}
	\item[Step 4: Resolution of closed toric substacks of smooth stacky charts] As $\X$ is a smooth stacky coordinate chart, $\X$ is a finite quotient $\X=[\AA^n/G]$. Just as in Step $1'$, the statement follows from the previous step and  \cref{lem:pushforwardquotient}.
	\item[Step 5: Resolution of toric substacks]
	By repeated application of \cref{lem:functorialityRestriction} (as in Step 2 to reduce to Step 4), $C_\bullet(S^\phi, \mathcal O^\phi)$ restricts to a resolution of $\phi_*\mathcal O_\Y$ on every smooth stacky toric chart. This is not enough to determine that $C_\bullet(S^\phi, \mathcal O^\phi)$ is a resolution of $\mathcal O_\Y$. Let $C_\bullet(S^\phi, \mathcal O^\phi)\to \mathcal F$ be exact. The morphism $\alpha^\phi$ determines a map $\mathcal O_\X \to \mathcal F$. Since $\mathcal F$ agrees with $\phi_* \mathcal O_\Y$ on every chart for a cover of $\X$, it is a twist of $\phi_*\mathcal O_\Y$ by a line bundle. The existence of an epimorphism $\mathcal O_\X\to \mathcal F$ determines that $\mathcal F \cong \phi_*\mathcal O_\Y$.
\end{description}
\end{proof}
\subsection{Exit paths and sheaves} \label{subsec:pathsandsheaves}

Let $(\Sigma_X, \ul \beta_X)$ and $(\Sigma_Y, \ul \beta_Y)$ be stacky fans associated to toric stacks $\X$ and $\Y$, respectively. Suppose further we have a closed immersion $\phi \colon \Y\to \X$ induced by an immersion of stacky fans $(\ul \Phi \colon L_Y\to L_X, \ul \phi\colon  N_Y\to N_X )$. As $(\ul \phi, \ul \Phi)$ is an immersion (\cref{def:inclusion}), we have that the cones of $\Sigma_Y$ are $\ul\Phi^{-1}(\Sigma_X)$.
We have a dual map $\ul\phi^* \colon  M_X\to M_Y$ inducing a map $\tilde\phi: M_X\tensor \RR/ M_X\to M_Y\tensor \RR/ M_Y$ on tori.
Consider now the \emph{real} torus $T^{\phi}:=\ker( \tilde{\phi})$, whose dimension is the codimension of $\Y$. \footnote{This is a torus (not a disjoint union of tori) by a variation of \cref{prop:soManyThingsAreTheSame} and the assumption that $\coker(\phi)$ is free.}
For every ray $\rho\in \Sigma_X(1)$, we obtain a map
\begin{align*}
     T^\phi\to& \, \RR/\ZZ\\
	[m]\mapsto& \, \ul \beta_X^*m(u_\rho)
\end{align*}
where, as before, $u_\rho$ is the primitive generator of $\rho$. We denote the kernel of this map by $T^{\phi}(\rho)$ which is either all of $T^\phi$ (when $\ul\beta_X(\rho)\in \Im(\ul\phi)$), or a disjoint union of toric hyperplanes\footnote{For simplicity of notation, we will usually call the $T^{\phi}(\rho)$ a toric hyperplane, even if it consists of several disjoint translates of one hyperplane.}. 

The set of subtori $\{T^{\phi}(\rho)\st \ul\beta_X(\rho) \not\in \Im(\ul\phi)\}$ is a toric hyperplane arrangement. Let $S^\phi$ be the corresponding stratification of $T^\phi$. 

By applying \cref{eq:lineBundleFromTorusPoints} and working under the assumption that $\Sigma_X$ is smooth, we associate to each stratum $\strata\subset T^\phi\subset M_{X, \RR}/M$  of $S^\phi$ a line bundle $\mathcal O^\phi(\strata)\in \Pic(\X)$.
For each point $p\in T^{\phi}$, let $\dim_{\sS}(p)$ be the dimension of the stratum containing $p$. If $\strata, \stratb$ are strata of $S^\phi$ such that $\strata> \stratb$, we let $\EP(\strata, \stratb)$ denote the homotopy classes of paths $\gamma \colon I\to T^{\phi}$ satisfying
\begin{align*}
	\gamma(0)\in \sigma, &  & \gamma(1)\in \tau, \text{ and} &  & t_1\leq t_2 \Rightarrow \dim_{\sS}(\gamma(t_1))\leq \dim_{\sS}(\gamma(t_2)).
\end{align*}
The exit path category $\EP(\sS^\phi)$ has objects equal to strata of $\sS^\phi$, morphisms given by $\hom(\strata, \stratb) = \EP(\strata, \stratb)$, and composition defined by concatenation of paths.
Given $\gamma\in \EP(\strata, \stratb)$, there exists lifts  $\tilde \strata, \tilde \stratb \in \sS_{\Sigma_X, \beta_X}$ satisfying the properties 
\begin{align*}
	(\tilde \strata \cap \ker(\ul\phi^*)_\RR)/\ker(\ul\phi)=\strata && (\tilde \stratb\cap \ker(\ul\phi^*)_\RR)/\ker(\ul\phi)=\stratb && \tilde \strata>\tilde \stratb
\end{align*}
By \cref{prop:cocoreMorphisms}, this determines a boundary morphism
\[ \bm^\phi_\gamma: \mathcal O^\phi(\strata)\to\mathcal \mathcal O^\phi( \stratb)\]
which is independent of the lift chosen. In summary:
\begin{prop}
There is a functor $\mathcal O^\phi \colon  \EP(\sS^\phi)\to \Coh(X) $ which sends
	\begin{align*}
	\strata\mapsto & \mathcal O^\phi(\strata) \\
	\gamma \mapsto & \bm^\phi_\gamma.
\end{align*}
\end{prop}

\subsection{Definition of \texorpdfstring{$C_{\bullet}(S^\phi, \mathcal O^\phi)$}{the resolution}}
\label{subsec:defOfCbullet}
Whenever the 1-dimensional cones of $\Sigma$ project to an $\RR$-basis for $N_X$ under $\ul\beta$, which occurs when we assume that $\X$ is covered by smooth stacky charts, the strata of $\sS^\phi$ are simply connected. If $\sS^\phi$ is a regular CW complex, then $\EP(\sS^\phi)$ is a posetal category, i.e.,  $|\EP(\strata,\stratb)| \in \{ 0,1 \}$. In our setting, it is possible that $S^\phi$ is not a regular CW complex --- for example, when $\phi \colon  \{\bullet\}\into \AA^1 $, as in \cref{ex:pointInA1}. From $\sS^\phi$, we build quiver $Q^\phi$ whose vertices are the strata of $\sS^\phi$, and whose edge sets are 
\[E(\strata, \stratb) =\left\{\begin{array}{cc}
	 \EP(\strata, \stratb) & \text{if $\dim(\stratb)=\dim(\strata)-1$}\\ \emptyset & \text{otherwise}\end{array}
\right.\] 
 We will let $Q^\phi$ denote this quiver, and $\gamma$ denote the edges in the quiver. We retain a map $|-|:Q^\phi \to \NN$ which assigns to each stratum its dimension.
Observe that when $\sS^\phi$ is a regular CW complex, $Q^\phi$ is the Hasse diagram of the poset associated with $S^\phi$.

Now arbitrarily assign orientations $o_\strata$ to every stratum $\strata\in \sS^\phi$. To each edge $\gamma\in Q^\phi$, we say that $\sgn(\gamma)=+1$ if the boundary orientation of $\strata$ (with outwards direction assigned by $\dot \gamma$) agrees with the boundary orientation of $\stratb$, and $-1$ otherwise. 

\begin{prop}\label{prop:QisaMorseQuiver}
		$(Q^\phi, |-|, o_\strata)$ is an oriented Morse quiver (see \cref{def:MorseQuiver} and following discussion). The functor $\mathcal O^\phi \colon  Q^\phi\to \Coh(X) $ is a sheaf on $Q^\phi$ (as defined in \cref{app:morseForQuivers}).
\end{prop}
\begin{proof}
	Let $|\strata|-|\stratb|=2$. Take lifts $\strata, \stratb$ to the universal cover so that the $\tilde \strata>\tilde \stratb$. Then there are two unique strata, $\tilde\stratc^+, \tilde\stratc^-$ with the property that $\tilde \strata>\tilde \stratc^\pm> \tilde \stratb$, and (because all strata are simply connected and we're on the universal cover) there exist unique morphisms $\gamma_1^\pm\in \EP(\tilde \strata, \tilde \stratc^\pm)$ and $\gamma_2^\pm\in \EP( \tilde \stratc^\pm, \tilde \stratb)$. This set has an involution (sending $\gamma_1^+\cdot \gamma_2^+$ to $\gamma_1^-\cdot \gamma_2^-)$, and $\sgn$ is constructed in such a way that $\sgn(\gamma_1^+\cdot \gamma_2^+)=-\sgn(\gamma_1^-\cdot \gamma_2^-)$. Observe that after fixing a lift $\tilde \strata_0$ of $\strata$, we can compute $\EP(\strata, \stratb)=\bigcup_{\text{Lifts $\tilde \stratb<\tilde \strata_0$}} \EP(\tilde \strata, \tilde \stratb)$. Therefore we have an involution on $Q^\phi$ making it an oriented Morse quiver. Since we have an agreement of the compositions $\gamma_1^+\cdot \gamma_2^+=\gamma_1^-\cdot \gamma_2^-$, $\mathcal O^\phi$ is a sheaf on $Q^\phi$ (as $\mathcal O^\phi$ is a functor on the exit path category).
\end{proof}
 By \cref{eq:morseComplex}, we obtain a chain complex $C_\bullet(Q^\phi, \mathcal O^\phi)$. For the purpose of streamlining notation, we will write \[C_\bullet(S^\phi, \mathcal O^\phi):= C_\bullet(Q^\phi, \mathcal O^\phi)\] instead.
The stratification $\sS^\phi$ has a special stratum $\sigma_0$ corresponding to the identity of the torus. The stratum $\sigma_0$ is always zero dimensional under our assumptions and $\bF(\sigma_0)=\mathcal O_{\X}$. We define the morphism $\alpha \colon  \mathcal O_{\X}\to C_{\bullet}(S^\phi, \mathcal O^\phi) $ to map via the identity to that direct summand.

\subsection{Examples of resolutions}
\begin{example}[Resolution of point in $\AA^1$]
\label{ex:pointInA1}
    For $\phi \colon e\to \AA^1 $ the inclusion of the identity point, the real torus is $T^\phi=S^1$. Observe that the dimension of the torus is the codimension of the point. The stratification comes from a single point corresponding to the one toric divisor of $\AA^1$. 
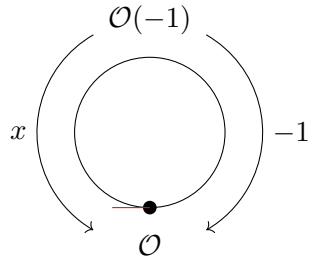
\begin{equation}
\begin{tikzpicture}

\draw  (0,0) ellipse (1 and 1);
\node[fill, circle, scale=.5] at (0,-1) {};
\draw[red] (0,-1) -- (-0.5,-1);
\node at (0,1.5) {$\mathcal O(-1)$};
\node at (0,-1.5) {$\mathcal O$};
\draw[->] (120:1.5) arc  (120:240:1.5);
\draw[->] (60:1.5) arc (60:-60:1.5);
\node[left] at (-1.5,0) {$x$};
\node[right] at (1.5,0) {$-1$};
\end{tikzpicture} \label{eq:resolutionPtInA1}
\end{equation}
    The complex is the standard resolution for the point in $\AA^1$.
\end{example}
\begin{example}[Resolution of point in $\PP^2$]
\label{ex:pointInP2}
	We now consider the identity point $e = (1:1:1)$ in $\PP^2$ included by $\phi \colon e \to \PP^2$.
	The stratification  $S^\phi$ is drawn in \cref{fig:fltzp2}. We label each stratum with a sheaf and each exit path with a morphism of sheaves in \cref{fig:fltzCP}. 
	\begin{figure}
		\centering
		\begin{subfigure}{.45\linewidth}
			\scalebox{.6}{\begin{tikzpicture}[scale=2]
\newcommand{\OO}{\mathcal O}
\usetikzlibrary{calc, decorations.pathreplacing,shapes.misc}
\usetikzlibrary{decorations.pathmorphing}

\tikzstyle{fuzz}=[red,
    postaction={draw, decorate, decoration={border, amplitude=0.15cm,angle=90 ,segment length=.15cm}},
]

\draw  (-1.5,3) rectangle (3.5,-2);
\draw[fuzz](1,3) node (v7) {} -- (1,-2);
\draw[fuzz](-1.5,0.5) -- (3.5,0.5) node (v6) {};
\draw[fuzz] (3.5,-2) node (v10) {} -- (-1.5,3);
\draw[thick,->] (0,0.5) -- (-0.5,0.5);
\draw[thick,->] (0,1.5) -- (0.5,1);
\draw[thick,->] (1,1) -- (1,1.5);
\end{tikzpicture} }
			\caption{FLTZ Stratification associated to $\PP^2$. Orientation is read off from the small arrows.}
			\label{fig:fltzp2}
		\end{subfigure}
		\begin{subfigure}{.45\linewidth}
			\scalebox{.6}{\begin{tikzpicture}[scale=2]
\newcommand{\OO}{\mathcal O}
\usetikzlibrary{calc, decorations.pathreplacing,shapes.misc}
\usetikzlibrary{decorations.pathmorphing}

\tikzstyle{fuzz}=[red, 
    postaction={draw, decorate, decoration={border, amplitude=0.15cm,angle=90 ,segment length=.15cm}},
]

\draw  (-1.5,3) rectangle (3.5,-2);
\draw[fuzz](1,3) node (v7) {} -- (1,-2);
\draw[fuzz](-1.5,0.5) -- (3.5,0.5) node (v6) {};
\draw[fuzz] (3.5,-2) node (v10) {} -- (-1.5,3);
\node[fill=white, fill opacity=50] (v2) at (1,0.5) {$\OO(0)$};
\node[fill=white, fill opacity=50] (v11) at (-1.5,3) {$\OO(-1)$};
\node[fill=white, fill opacity=50] (v1) at (-1.5,0.5) {$\OO(-1)$};
\node[fill=white, fill opacity=50] (v4) at (1,-2) {$\OO(-1)$};
\node[fill=white, fill opacity=50] (v3) at (-0.5,-1) {$\OO(-1)$};
\node[fill=white, fill opacity=50] (v8) at (2.5,2) {$\OO(-2)$};
\draw  (v1) edge (v2);
\node (v9) at (3.5,3) {};
\node (v5) at (-1.5,-2) {};

\draw[thick,  ->]  (v3) edge  node[midway, fill=white, fill opacity=50]{1} (v1);
\draw[thick,  ->]  (v3) edge  node[midway, fill=white, fill opacity=50]{1} (v4);
\draw[thick,  ->]  (v3) edge  node[midway, fill=white, fill opacity=50]{1} (v5);
\draw[thick,  ->]  (v1) edge  node[midway, fill=white, fill opacity=50]{$-x_2$}(v2);
\draw[thick,  ->]  (v6) edge  node[midway, fill=white, fill opacity=50]{$x_0$}(v2);
\draw[thick,  ->]  (v4) edge  node[midway, fill=white, fill opacity=50]{$x_2$}(v2);
\draw[thick,  ->]  (v7) edge  node[midway, fill=white, fill opacity=50]{$-x_1$}(v2);
\draw[thick,  ->]  (v8) edge  node[midway, fill=white, fill opacity=50]{$-x_0$}(v7);
\draw[thick,  ->]  (v8) edge  node[midway, fill=white, fill opacity=50]{$-x_1$}(v6);
\draw[thick,  ->]  (v8) edge  node[midway, fill=white, fill opacity=50]{$-x_2$}(v9);
\draw[thick,  ->]  (v10) edge node[midway, fill=white, fill opacity=50]{$-x_0$}(v2);
\draw[thick,  ->]  (v11) edge node[midway, fill=white, fill opacity=50]{$x_1$}(v2);
\end{tikzpicture} }
			\caption{Overlaying a diagram of sheaves, with signs determined by the orientation}
			\label{fig:fltzCP}
		\end{subfigure}
    \caption{Stratification and associated diagram of sheaves for resolving the identity point in $\PP^2$.}
	\end{figure}
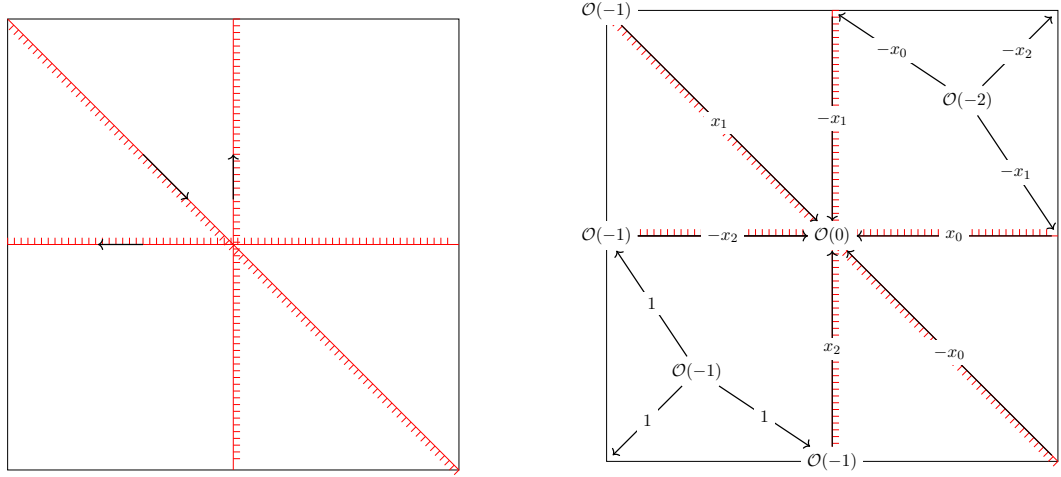
	For this example, the orientation on the 2-strata is given by the standard orientation of the plane, while the orientation on the 1-strata is given by the arrows indicated in \cref{fig:fltzp2}. The resulting complex $C_\bullet(\sS^\phi, \mathcal O^\phi)$ is given by the top row of the following commutative diagram.
	\[
		\scalebox{.75}{\begin{tikzpicture}

            \node (v1) at (-6.5,1.5) {$\mathcal O(-2)\oplus \mathcal O(-1)$};
            \node (v2) at (1,1.5) {$\mathcal O(-1)\oplus \mathcal O(-1)\oplus \mathcal O(-1)$};
            \node (v3) at (10,1.5) {$\mathcal O$};
            \draw[->]  (v1) edge node[fill=white]{$\begin{pmatrix} -x_2& 1\\ -x_0 & 1\\ -x_1& 1\end{pmatrix}$} (v2);
            \draw[->]  (v2) edge node[fill=white]{$\begin{pmatrix} x_0-x_1& x_1-x_2& x_2-x_0 \end{pmatrix}$}(v3);
            \node (v6) at (10,-1.5) {$\mathcal O$};
            \node (v5) at (1,-1.5) {$\mathcal O(-1)\oplus \mathcal O(-1)$};
            \node (v4) at (-6.5,-1.5) {$\mathcal O(-2)$};
            \draw[->]  (v1) edge[bend left] node[fill=white, right]{$\begin{pmatrix} -1& 0 \end{pmatrix}$}(v4);
            \draw[->]  (v2) edge[bend left] node[fill=white, right]{$\begin{pmatrix} -1 & 0& 1\\ 0& -1&1  \end{pmatrix}$} (v5);
            \draw[->]  (v3) edge[bend left] node[fill=white, right]{$\begin{pmatrix} 1 \end{pmatrix}$}(v6);

            \draw[<-]  (v1) edge[bend right] node[fill=white, left]{$\begin{pmatrix} -1\\ x_0-x_1+x_2\end{pmatrix}$}(v4);
            \draw[<-]  (v2) edge[bend right] node[fill=white, left]{$\begin{pmatrix} 0& -1\\ -1 & 0\\ -1&-1  \end{pmatrix}$} (v5);
            \draw[<-]  (v3) edge[bend right] node[fill=white, left]{$\begin{pmatrix} 1  \end{pmatrix}$}(v6);

            \draw  (v5) edge node[fill=white]{$\begin{pmatrix} x_0-x_1& x_1-x_2  \end{pmatrix}$} (v6);
            \draw  (v4) edge node[fill=white]{$\begin{pmatrix} x_1-x_2\\ x_1-x_0  \end{pmatrix}$}(v5);

        \end{tikzpicture} }
	\]
	We first check ``by hand'' that this is a resolution of  $\mathcal O_e$. The bottom row provides a resolution of $\mathcal O_e$ (as the intersection of the lines $x_0-x_1$ and $x_1-x_2$). The morphisms between the top and bottom rows are chain homotopy equivalences.

	We now instead sketch how to use \cref{lem:functorialityRestriction} to prove that $C_\bullet(\sS^\phi, \mathcal O^\phi)$ resolves $\mathcal O_e$ following our general proof strategy. 
	Let $i_\rho\colon \AA^2\to \PP^2$ be the toric inclusion of the $x_0x_1$-plane. Then we can then look at  $i^*_{\rho}C_\bullet(\sS^\phi, \mathcal O^\phi)$ which is a diagram of sheaves on $\AA^2$ (see \cref{fig:fltzCP2}). 
	\begin{figure}
		\centering
		\begin{subfigure}{.45\linewidth}
			\scalebox{.6}{\begin{tikzpicture}[scale=2]
\newcommand{\OO}{\mathcal O}
\usetikzlibrary{calc, decorations.pathreplacing,shapes.misc}
\usetikzlibrary{decorations.pathmorphing}

\tikzstyle{fuzz}=[red,
    postaction={draw, decorate, decoration={border, amplitude=0.15cm,angle=90 ,segment length=.15cm}},
]

\draw  (-1.5,3) rectangle (3.5,-2);
\draw[fuzz](1,3) node (v7) {} -- (1,-2);
\draw[fuzz](-1.5,0.5) -- (3.5,0.5) node (v6) {};
\draw (3.5,-2) node (v10) {} (-1.5,3);
\node[fill=white, fill opacity=50] (v2) at (1,0.5) {$\OO$};
\node[fill=white, fill opacity=50] (v11) at (-1.5,3) {$\OO$};
\node[fill=white, fill opacity=50] (v1) at (-1.5,0.5) {$\OO$};
\node[fill=white, fill opacity=50] (v4) at (1,-2) {$\OO$};
\node[fill=white, fill opacity=50] (v3) at (-0.5,-1) {$\OO$};
\node[fill=white, fill opacity=50] (v8) at (2.5,2) {$\OO$};
\draw  (v1) edge (v2);
\node (v9) at (3.5,3) {};
\node (v5) at (-1.5,-2) {};

\draw[thick,  ->]  (v3) edge  node[midway, fill=white, fill opacity=50]{1} (v1);
\draw[thick,  ->]  (v3) edge  node[midway, fill=white, fill opacity=50]{1} (v4);
\draw[thick, red, ->]  (v3) edge  node[midway, fill=white, fill opacity=50]{1} (v5);
\draw[thick,  ->]  (v1) edge  node[midway, fill=white, fill opacity=50]{$-x_2$}(v2);
\draw[thick,  ->]  (v6) edge  node[midway, fill=white, fill opacity=50]{$x_0$}(v2);
\draw[thick,  ->]  (v4) edge  node[midway, fill=white, fill opacity=50]{$x_2$}(v2);
\draw[thick,  ->]  (v7) edge  node[midway, fill=white, fill opacity=50]{$-x_1$}(v2);
\draw[thick,  ->]  (v8) edge  node[midway, fill=white, fill opacity=50]{$-x_0$}(v7);
\draw[thick,  ->]  (v8) edge  node[midway, fill=white, fill opacity=50]{$-x_1$}(v6);
\draw[thick,  ->]  (v8) edge  node[midway, fill=white, fill opacity=50]{$-x_2$}(v9);
\draw[thick,  ->]  (v10) edge node[midway, fill=white, fill opacity=50]{$-x_0$}(v2);
\draw[thick,  ->]  (v11) edge node[midway, fill=white, fill opacity=50]{$x_1$}(v2);

\end{tikzpicture} }
			\caption{Taking $i^*_\rho C_\bullet(\sS^\phi, \mathcal O^\phi)$}
			\label{fig:fltzCP2}
		\end{subfigure}
		\begin{subfigure}{.45\linewidth}
			\scalebox{.6}{\begin{tikzpicture}[scale=2]
\newcommand{\OO}{\mathcal O}
\usetikzlibrary{calc, decorations.pathreplacing,shapes.misc}
\usetikzlibrary{decorations.pathmorphing}

\tikzstyle{fuzz}=[red,
    postaction={draw, decorate, decoration={border, amplitude=0.15cm,angle=90 ,segment length=.15cm}},
]

\draw  (-1.5,3) rectangle (3.5,-2);
\draw[fuzz](1,3) node (v7) {} -- (1,-2);
\draw[fuzz](-1.5,0.5) -- (3.5,0.5) node (v6) {};
\draw (3.5,-2) node (v10) {} (-1.5,3);
\node[fill=white, fill opacity = 50] (v2) at (1,0.5) {$\OO$};
\node[fill=white, fill opacity = 50] (v1) at (-1.5,0.5) {$\OO$};
\node[fill=white, fill opacity = 50] (v4) at (1,-2) {$\OO$};
\node[fill=white, fill opacity = 50] (v3) at (-1.5,-2) {$\OO$};
\node (v8) at (3.5,3) {};
\draw  (v1) edge (v2);
\node (v9) at (3.5,3) {};
\node (v5) at (-1.5,-2) {};

\draw[thick, ->]  (v3) edge node[midway, fill=white, fill opacity=50]{1} (v1);
\draw[thick, ->]  (v3) edge node[midway, fill=white, fill opacity=50]{1} (v4);
\draw[thick, ->]  (v1) edge node[midway, fill=white, fill opacity=50]{-1}(v2);
\draw[thick, ->]  (v6) edge node[midway, fill=white, fill opacity=50]{$x_0$}(v2);
\draw[thick, ->]  (v4) edge node[midway, fill=white, fill opacity=50]{1}(v2);
\draw[thick, ->]  (v7) edge node[midway, fill=white, fill opacity=50]{$-x_1$}(v2);
\draw[thick, ->]  (v8) edge node[midway, fill=white, fill opacity=50]{$x_0$}(v7);
\draw[thick, ->]  (v8) edge node[midway, fill=white, fill opacity=50]{$x_1$}(v6);
\end{tikzpicture} }
			\caption{Applying a homotopy to $i^*_\rho C_\bullet(\sS^\phi, \mathcal O^\phi)$ to obtain $C_\bullet(S^{\phi'}, \mathcal O^{\phi'})$ where $\phi':e\to \AA^2$.}
			\label{fig:fltzCP3}
		\end{subfigure}
        \caption{Restriction to a chart in $\PP^2$ is homotopic to the diagram for $\AA^2$.}
	\end{figure}
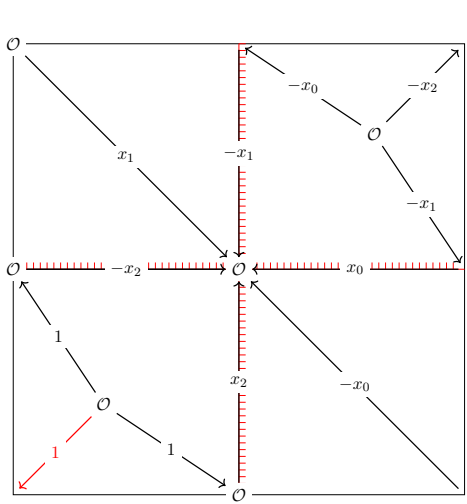
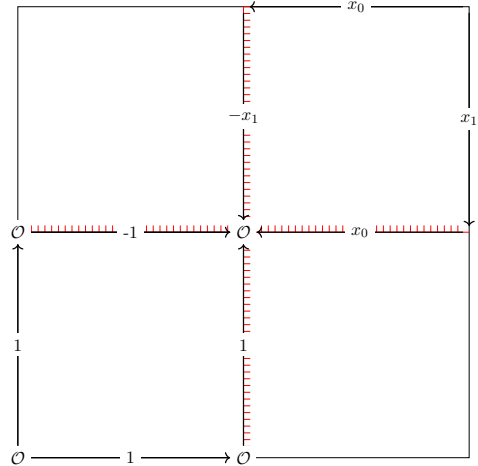
	The highlighted red arrow is now an invertible morphism, and we can construct a chain homotopy that simplifies the diagram by contracting along this arrow. This homotopy is the content of \cref{lem:functorialityRestriction}. The resulting diagram, drawn in \cref{fig:fltzCP3}, is the Koszul resolution of the skyscraper sheaf on $\AA^2$.
	Since a similar story applies to the $x_0x_2$ and $x_1x_2$ planes, we obtain that $C_\bullet(\sS^\phi, \mathcal O^\phi)$ is a resolution of the skyscraper sheaf by line bundles.
\end{example}

\begin{example}[Resolution of the diagonal on $\PP^1\times \PP^1$]
	We now consider the diagonal inclusion $\phi \colon \PP^1\to \PP^1\times \PP^1 $. 
	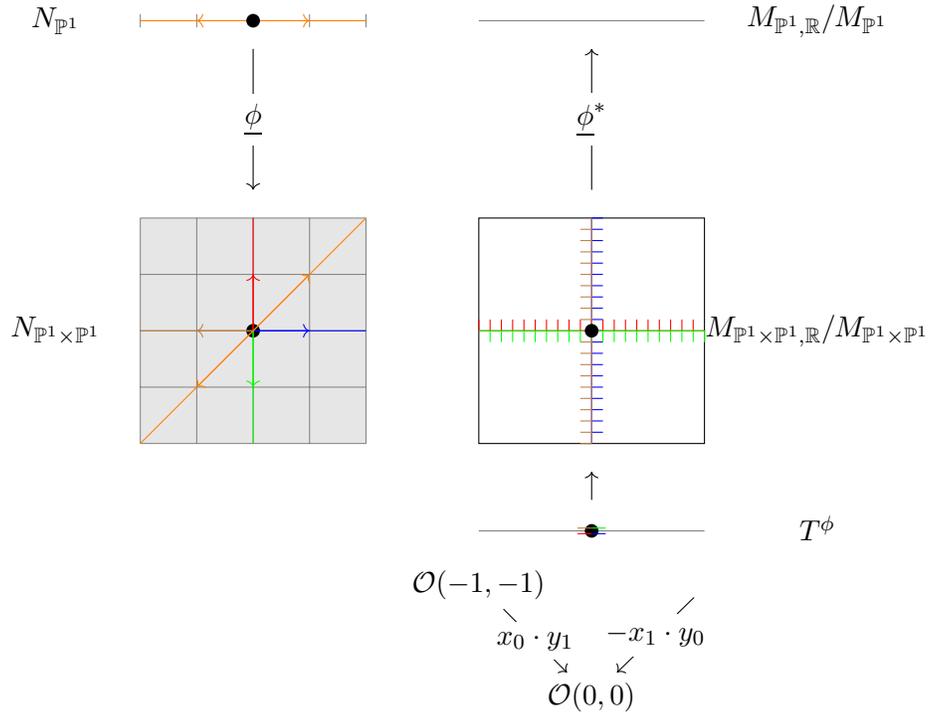
\begin{figure}
	    \centering
	    \begin{tikzpicture}[decoration=ticks, segment length=.75cm, scale=.75]

    \begin{scope}[]
        \begin{scope}[]
            \fill[gray!20]  (-2, -2) rectangle (2,2);
            \draw [help lines, step=1cm] (-2,-2) grid (2,2);
            \node[circle, fill=black, scale=.5] at (0,0) {};
            \draw[blue] (0,0) edge[->] (1,0) edge (2, 0);
            \draw[red] (0,0) edge[->] (0, 1) edge (0, 2);
            \draw[brown] (0,0) edge[->] (-1,0) edge (-2, 0);
            \draw[green] (0,0) edge[->] (0, -1) edge (0, -2);
            \draw[orange](-2,-2) edge (2,2);
            \draw[orange] (0,0) edge[->] (1,1) edge[->] (-1,-1);
        \end{scope}
        \begin{scope}[shift={(0,5.5)}]
            \draw[decorate, help lines](-2,0)-- (2.01,0);
            \draw[orange](-2,0) -- (2,0);
            \draw[orange] (0,0) edge[->] (1,0) edge[->] (-1,0);
        \end{scope}
\draw[->] (0,5) -- node[midway, fill=white]{$\ul \phi$} (0,2.5);
    \end{scope}

\tikzstyle{fuzz}=[red,
postaction={draw, decorate, decoration={border, amplitude=0.15cm,angle=90 ,segment length=.15cm}},
]

\begin{scope}[shift={(4.5,0)}]

\draw  (-0.5,2) rectangle (3.5,-2);
\draw[fuzz, blue](1.5,2) --(1.5,-2);
\draw[fuzz, red](-0.5,0) -- (3.5,0);
\draw[fuzz, green](3.5,0) -- (-0.5,0);
\draw[fuzz, brown] (1.5,-2)  -- (1.5,2);
\draw[gray] (-0.5,-3.55) -- (3.5,-3.55);
\node[circle, fill=black, scale=.5] at (1.5,-3.55) {};
\node[circle, fill=black, scale=.5]  at (1.5,0) {};
\draw[red] (1.5,-3.6) -- (1.25,-3.6);
\draw[brown] (1.5,-3.5) -- (1.25,-3.5);
\draw[blue] (1.5,-3.6) -- (1.75,-3.6);
\draw[green] (1.5,-3.5) -- (1.75,-3.5);
\draw[<-] (1.5,5) -- node[midway, fill=white]{$\ul \phi^*$} (1.5,2.5);
\end{scope}
\node (v1) at (4,-4.5) {$\mathcal O(-1, -1)$};
\node (v2) at (6,-6.5) {$\mathcal O(0,0)$};
\draw  (v1) edge[->] node[midway, fill=white]{$x_0\cdot y_1$} (v2);
\node (v3) at (8,-4.55) {$\;\;\;\;\;$};
\draw  (v3) edge[->] node[midway, fill=white]{$-x_1\cdot y_0$} (v2);
\node[circle, fill=black, scale=.5]  at (0,5.5) {};
\node at (-3.5,5.5) {$N_{\mathbb P^1}$};
\node at (-3.5,0) {$N_{\mathbb P^1\times \mathbb P^1}$};
\node at (10,-3.5) {$T^\phi$};
\node at (10,0) {$M_{\mathbb P^1\times \mathbb P^1,\mathbb R}/M_{\mathbb P^1\times \mathbb P^1}$};
\node at (10,5.5) {$M_{\mathbb P^1,\mathbb R}/M_{\mathbb P^1}$};
\draw[gray] (4,5.5) -- (8,5.5);
\draw[->] (6,-3) -- (6,-2.5);
\end{tikzpicture}
 	    \caption{Lattices and real tori which play a role in the construction of the resolution of the diagonal in $\PP^1\times \PP^1$.}
	    \label{fig:p1p1diagonal}
	\end{figure}
	On the left-hand side of \cref{fig:p1p1diagonal}, we have the map of fans for the diagonal inclusion $\phi \colon \PP^1\to \PP^1\times \PP^1 $. On the bottom of the right-hand side, we have the torus $M_{\PP^1, \RR}\times M_{\PP^1, \RR}$ with the stratification $\mathcal{S}_{\PP^1 \times \PP^1}$. The kernel of $\ul \phi^*$, on the bottom right, inherits a stratification labeled by line bundles on $\PP^1\times \PP^1$. The corresponding diagram gives a resolution for $\phi_*\mathcal O_{\PP^1}$. This can be seen, for instance, by restricting to the four toric charts (setting some of the $x_i, y_i$ to $1$) as in the previous example.
\end{example}

\begin{example}[Resolution of the diagonal on $\PP^2\times \PP^2$]
	While it is beyond our ability to draw the map of fans for the diagonal embedding $\phi \colon  \PP^2\to \PP^2\times \PP^2 $, the torus $T^\phi$ is two-dimensional. The stratification with the associated resolution overlayed is drawn in \cref{fig:p2DiagonalResolution}.
	\begin{figure}
    \centering
		\scalebox{.5}{\begin{tikzpicture}[scale=2]
\newcommand{\OO}{\mathcal O}
\usetikzlibrary{calc, decorations.pathreplacing,shapes.misc}
\usetikzlibrary{decorations.pathmorphing}

\tikzstyle{fuzz}=[
    postaction={draw, red, decorate, decoration={border, amplitude=0.15cm,angle=90 ,segment length=.15cm}},
]
\tikzstyle{antifuzz}=[
    postaction={draw, blue, decorate, decoration={border, amplitude=0.15cm,angle=-90 ,segment length=.15cm}},
]

\draw  (-1.5,3) rectangle (3.5,-2);
\draw[fuzz](1,3) -- (1,-2);
\draw[fuzz](-1.5,0.5) -- (3.5,0.5) ;
\draw[fuzz] (3.5,-2) -- (-1.5,3);
\draw[antifuzz](1,3) node (v7) {} -- (1,-2);
\draw[antifuzz](-1.5,0.5) -- (3.5,0.5) node (v6) {};
\draw[antifuzz] (3.5,-2) node (v10) {} -- (-1.5,3);
\node[fill=white, fill opacity=50] (v2) at (1,0.5) {$\OO(0)$};
\node[fill=white, fill opacity=50] (v11) at (-1.5,3) {$\OO(-1,-1)$};
\node[fill=white, fill opacity=50] (v1) at (-1.5,0.5) {$\OO(-1,-1)$};
\node[fill=white, fill opacity=50] (v4) at (1,-2) {$\OO(-1,-1)$};
\node[fill=white, fill opacity=50] (v3) at (-0.5,-1) {$\OO(-1,-2)$};
\node[fill=white, fill opacity=50] (v8) at (2.5,2) {$\OO(-2,-1)$};
\draw  (v1) edge (v2);
\node (v9) at (3.5,3) {};
\node (v5) at (-1.5,-2) {};

\draw[thick, ->]  (v3) edge node[midway, fill=white, fill opacity=50]{$1    \cdot  y_0$} (v4);
\draw[thick, ->]  (v3) edge node[midway, fill=white, fill opacity=50]{$1    \cdot  y_2$} (v5);
\draw[thick, ->]  (v3) edge node[midway, fill=white, fill opacity=50]{$1    \cdot  y_1$} (v1);
\draw[thick, ->]  (v1) edge node[midway, fill=white, fill opacity=50]{$ x_2 \cdot  y_0$}(v2);
\draw[thick, ->]  (v6) edge node[midway, fill=white, fill opacity=50]{$ x_0 \cdot  y_2$}(v2);
\draw[thick, ->]  (v4) edge node[midway, fill=white, fill opacity=50]{$ x_2 \cdot  y_1$}(v2);
\draw[thick, ->]  (v7) edge node[midway, fill=white, fill opacity=50]{$ x_1 \cdot  y_2$}(v2);
\draw[thick, ->]  (v8) edge node[midway, fill=white, fill opacity=50]{$ x_0 \cdot  1$}(v7);
\draw[thick, ->]  (v8) edge node[midway, fill=white, fill opacity=50]{$ x_1 \cdot  1$}(v6);
\draw[thick, ->]  (v8) edge node[midway, fill=white, fill opacity=50]{$ x_2 \cdot  1$}(v9);
\draw[thick, ->] (v10) edge node[midway, fill=white, fill opacity=50]{$ x_0 \cdot  y_1$}(v2);
\draw[thick, ->] (v11) edge node[midway, fill=white, fill opacity=50]{$ x_1 \cdot  y_0$}(v2);
\end{tikzpicture} }
  \caption{The resolution of the diagonal of $\PP^2\times \PP^2$ (signs omitted).}
  \label{fig:p2DiagonalResolution}
	\end{figure}
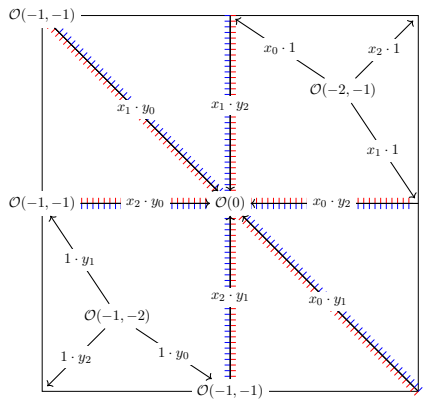
	The red (respectively blue) lines represent the 1-dimensional cones from the first (respectively second) fan in $\Sigma_{\PP^2}\times \Sigma_{\PP^2}$.
\end{example}
\section{Proofs} \label{sec:Resolutionproofs}
In this section, we prove the lemmas whose proofs were omitted in \cref{sec:Resolutionideas}.

\subsection{\texorpdfstring{\cref{lem:kunneth}}{Property}: K\"unneth Formula}
\label{pf:kunneth}

We will prove that if $\phi_1 \colon \Y_1\to \X_1 $ and $\phi_2 \colon  \Y_2\to \X_2 $ are toric immersions, then \[C_\bullet(S^{\phi_1\times \phi_2}, \mathcal O^{\phi_1\times \phi_2})=C_\bullet(S^{\phi_1}, \mathcal O^{\phi_1})\tensor C_\bullet(S^{\phi_2}, \mathcal O^{\phi_2}).\]
Let $\pi_i: \X_1\times \X_2\to \X_i$ for $i \in \{1,2\}$ be projection onto a factor and let $(\ul\pi_i, \ul\Pi_i)$ be the associated map of stacky fans.
The stratification $S^{\phi_1\times\phi_2}$ of $T^{\phi_1\times \phi_2}=T^{\phi_1}\times T^{\phi_2}$ is by hyperplanes of the form $T^{\rho_1}\times T^{\phi_2}$ and $T^{\phi_1}\times T^{\rho_2}$ where $\rho_i\in \Sigma_i(1)$. It follows that $S^{\phi_1\times\phi_2}$ is the product stratification $S^{\phi_1}\times S^{\phi_2}$. On $(\strata_1,\strata_2)\in S^{\phi_1\times \phi_2}$ it is clear that $\bF(\strata_1,\strata_2)=\ul \Pi_1^*\bF_1(\strata_1)+\ul \Pi_2^*\bF_2(\strata_2)$ and thus so $\mathcal O^{\phi_1\times \phi_2}(\strata_1, \strata_2)=\pi^*_1\mathcal O^{\phi_1} (\sigma_1) \tensor \pi^*_2\mathcal O^{\phi_2}(\strata_2)$.

Write $\strata\gtrdot\stratb$ is $\strata>\stratb$ and $|\strata|-|\stratb|=1$.
Additionally, if $\gamma_{12} \colon (\strata_1, \strata_2)\gtrdot(\stratb_1, \stratb_2)$ it follows that either $\strata_1\gtrdot\stratb_1$, $\strata_1=\stratb_2$ or $\strata_1=\strata_2, \stratb_1\gtrdot\stratb_2$. 
It follows that $C_\bullet(S^{\phi_1\times \phi_2}, \mathcal O^{\phi_1\times \phi_2})=
\pi_1^*C_\bullet(S^{\phi_1}, \mathcal O^{\phi_1})\tensor \pi_2^*C_\bullet(S^{\phi_2}, \mathcal O^{\phi_2})$

\subsection{\texorpdfstring{\cref{lem:functorialityRestriction}}{Property}: Restriction Homotopy Functoriality}
\label{pf:restrictionHomotopy}
	We now prove the ``meat'' of our approach: \cref{lem:functorialityRestriction}.
	We divide the proof into two cases. The first case, where there exists a $\rho'\in \Sigma(1)$ parallel to $\rho$, is straightforward. In that case, $S^\phi$ and $S^{\phi_\rho}$ agree as stratifications, and we have a commutative diagram of functors
	\[\begin{tikzcd}
		\EP(S^\phi) \arrow[equal]{r} \arrow{d}{\mathcal O^\phi} & \EP(S^{\phi_\rho}) \arrow{d}{\mathcal O^{\phi_\rho}}\\
		\Coh(X)\arrow{r}{i^*_\rho}& \Coh(X_\rho)
	\end{tikzcd}
	\]
    from which we obtain an isomorphism of complexes $C_\bullet(S^\phi,  \mathcal O^\phi) =i^*_\rho C_\bullet(S^{\phi_\rho}, \mathcal O^{\phi_\rho})$.

	We now examine the second case, where $\sS^\phi$ and $\sS^{\phi_\rho}$ do not agree. The stratification $\sS^\phi$ is a refinement of $S^{\phi_\rho}$, and we can draw inspiration from discrete Morse theory.
	\begin{thm}[\cite{forman1998morse}]
		Consider a CW complex $S$, and $\ZZ$ be a constant sheaf on $M$, the total space of $S$. Given a discrete Morse function $f\colon S\to \RR$, we can produce a chain complex  $C_\bullet(S, f;\ZZ)$ (the discrete Morse complex of $f$), and another CW complex $\tilde S$ (whose cells are the critical cells of $f$) so that 
		\begin{itemize}
			\item We have a homotopy equivalence $\HE\colon C_\bullet(\tilde S, \ZZ)\sim C_\bullet(S, f;\ZZ)$, and the spaces $S$ and $\tilde S$ are homotopic. 
			\item The discrete Morse complex of $f$ agrees with the cellular complex of $\tilde S$, i.e., $C_\bullet(S, f;\ZZ)=C_\bullet(\tilde S;\ZZ)$.
		\end{itemize}
	\end{thm}
	In particular, when $S$ is a refinement of $\tilde S$, there exists a discrete Morse function $f\colon S\to \RR$ whose critical cells are in bijection with the cells of $\tilde S$, so that $\HE\colon C_\bullet(S;\ZZ)\sim C_\bullet(S, f;\ZZ)$. \Cref{app:morseForQuivers} reviews the construction of this complex for certain quivers generalizing CW complexes and homology with coefficients in an abelian category $\mathcal C$ in place of $\ZZ$.

 We now sketch the proof of \cref{lem:functorialityRestriction} in the case where $S^\phi\neq S^{\phi_\rho}$. We first define an acyclic partial matching on the strata of $S^{\phi}$. The discrete Morse homology machinery ``simplifies'' our complex, giving us a new Morse quiver $\widetilde{S^\phi}$ with a new sheaf $\widetilde{ i^*_\rho\mathcal O^\phi}$. Then invariance of discrete Morse homology (the version we need is a variation of Forman's theorem, and given in \cref{thm:MorseInvariance}) provides us the homotopy equivalence which is the middle arrow in the diagram below. 
\[
	\begin{tikzcd}[column sep = 2cm]
		S^\phi\arrow{d}{\mathcal O^\phi}\arrow[equal]{r} & S^\phi \arrow{d}{i^*_\rho \mathcal O^\phi} & \widetilde {S^{\phi}}\arrow{d}{\widetilde{i^*_\rho \mathcal O^\phi}}\arrow[equal]{r} & S^{\phi_\rho}\arrow{d}{\mathcal O^{\phi_\rho}}\\
		\Coh(X) \arrow{r}{i^*_\rho} &\Coh(X_\rho)&  \Coh(X_\rho) \arrow[equal]{r} &\Coh(X_\rho)\\
		C_\bullet(S^\phi, \mathcal O^\phi) \arrow{r}{i^*_\rho} & C_\bullet(S^\phi, i^*_\rho\mathcal O^\phi)\arrow["\HE " above, " \text{\cref{thm:MorseInvariance}}" below]{r} & C_\bullet(\widetilde {S^\phi},\widetilde{ i^*_\rho\mathcal O^\phi}) \arrow[equal, "\text{ \cref{lem:equalityOfComplexes}}" below]{r} &  C_\bullet(S^{\phi_\rho},\mathcal O^{\phi_\rho}) 
	\end{tikzcd}.
\]
It remains to show that this new Morse complex is isomorphic (not simply homotopic!) to our desired result. 
\begin{prop}
	\label{lem:equalityOfComplexes}
	There exists an acyclic matching on $E^\ott$ on edges of $Q^\phi$ respecting the sheaf $i^*_\rho\mathcal O^\phi$ such that the Morse reduction (see \cref{def:MorseReduction})   $C_\bullet(\widetilde {S^\phi},\widetilde{ i^*_\rho\mathcal O^\phi})$ is isomorphic to $C_\bullet(S^{\phi_\rho},\mathcal O^{\phi_\rho})$.
\end{prop}
\begin{proof} First, we construct the acyclic partial matching. 

We say that $\gamma\in \EP(\strata, \stratb)$ ends positively inside $T^{\phi}(\rho)$ if $|\strata|=|\stratb|+1$, $\gamma(1)\subset T^\phi(\rho)$, and the lift $\tilde \gamma\subset M$ satisfies $\gamma(0)(u_\rho)<\gamma(1)(u_\rho)$\footnote{In fact, since the strata are convex, we may arrange all of our paths to be straight lines, so the value of $u_\rho$ either increases, decreases or is constant along every path}.
Let $E^\ott$ denote the set of all $\gamma\in E(Q^\phi)$ which end positively inside $T^{\phi}(\rho)$. 
This is an acyclic partial matching on the quiver $Q^\phi$. It is a partial matching because for every $\strata$, 
\[\left|\bigcup_{E^\ott\cap\stratb\lessdot \strata}\EP(\strata,\stratb)\right|+ \left|\bigcup_{E^\ott\cap\strata\lessdot \stratc}\EP(\stratc,\strata)\right|\leq 1.\]
that is, every stratum exclusively is either the $\rho$-positive face of some stratum or has at most one $\rho$-positive face.
We now show that the partial matching is acyclic. Let $\eps:=\strata_0\tto\strata_1\ott\strata_2\tto\cdots \tto\strata_{2k}$ be a candidate cycle\footnote{Because of the partial matching condition, the $\ott$ may occur at most every other in a path. Since the $\ott$ decrease the index, while the $\tto$ arrows increase the index, any candidate cycle must be alternating}. 
Because the $\ott$ arrows end positively inside $T^{\phi}(\rho)$, we learn that every $\strata_{2i+1}\subset T^{\phi}(\rho)$. It follows that there exists $\strata\in S^{\phi_\rho}$ whose image contains the images of $\strata_i$ as subsets. Lift $\eps$ to a path $\tilde \eps$ in $M$. Since $\eps$ is contained in the simply connected region $\strata$, $\eps$ is a cycle if and only if its lift to $M$ is a cycle. However, the value of $u_\rho: M\to \RR$ increases along the path $\eps$, so $\tilde \eps$ is not a cycle.

Given a quiver $Q^\phi$ with an acyclic matching, there exists a Morse reduction $\tilde Q^\phi$, whose vertices are critical vertices of $Q^\phi$ and whose edges are Morse trajectories  (\cref{def:MorseReduction}). We now prove that the quivers $\tilde Q^\phi$ and $Q^{\phi_\rho}$ are isomorphic. For every stratum $\sigma \in Q^{\phi_\rho}$, look at the stratum $\stratb_i\in Q^\phi$ with $\stratb\subset \strata$ and $|\stratb_i|=|\strata|$. We can find lifts $\tilde {\stratb_i}, \tilde \strata$ so that $\tilde{\stratb_i}\subset \tilde \strata$. Then there is a unique $\tilde \stratb_0$ on which $u_\rho: M\to \RR$ achieves a minimal value, and $\stratb_0\in \Crit(Q^\phi, E^\ott)$. This sets up a bijection between vertices of $\tilde Q^\phi$ and $Q^{\phi_\rho}$. The bijection between gradient flow lines $\mathcal M(\strata, \stratb)$ which form the edges of $\tilde Q^\phi$ and the exit paths $\EP(\strata, \stratb)$ which form the edges of $Q^{\phi_\rho}$ is constructed in a similar manner. Let $H: \tilde Q^\phi\to Q^{\phi_\rho}$ denote this isomorphism of quivers. 

We now show that this acyclic matching respects $i^*_\rho\mathcal O^\phi: Q^\phi\to  \Coh(X)$. For, $\stratb\ott\strata \colon \gamma $ an exit path in $E^\ott$, the boundary morphism $\mathcal O^\phi(\gamma) \colon  \mathcal O^\phi(\strata)\to \mathcal O^\phi(\stratb) $ is computed by taking lifts $\tilde \stratb\lessdot \tilde \strata \subset M$ and taking the section indexed by the origin in $\Delta_\beta(\bF(\stratb)-\bF(\strata))$. 
Since $\bF(\stratb)(u_{\rho'})=\bF(\strata)(u_{\rho'})$ for all $\rho'\in (\Sigma(1)\setminus \{\rho\})= \Sigma_{\rho}(1)$, we obtain that $i^*_\rho\mathcal O^\phi(\stratb)=i^*_{\rho}\mathcal O^\phi(\strata)$, and $i^*_\rho \mathcal O^\phi(\gamma)=\id$.

To show that $C_\bullet(\widetilde {S^\phi},\widetilde{ i^*_\rho\mathcal O^\phi})$ is isomorphic to $ C_\bullet(S^{\phi_\rho},\mathcal O^{\phi_\rho})$, we must show that we have the following commutative diagrams:
\[
	\begin{tikzcd}
		\widetilde{Q^\phi} \arrow{rr}{\sim}\arrow{dr}{\widetilde{ i^*_\rho\mathcal O^\phi}} & & Q^{\phi_\rho} \arrow{dl}{\mathcal O^{\phi_\rho}} \\& \Coh(X_\rho)
	\end{tikzcd}
	\begin{tikzcd}
		\widetilde{Q^\phi} \arrow{rr}{\sim} \arrow{dr}{\tilde{\sgn}} & & Q^{\phi_\rho} \arrow{dl}{\sgn_\rho} \\ &\{\pm 1\}
	\end{tikzcd}
\]

For the first diagram, let $\eps \colon  \strata\to \stratb $ be an edge of $\widetilde{Q^\phi}$. We can expand it into its constituent edges from $Q^\phi$,  $\eps=\strata_1\xrightarrow{\gamma_1}\stratb_1\xleftarrow{\delta_1}\strata_2\xrightarrow{\gamma_2}\cdots \xrightarrow{\gamma_k}\stratb_k$. 
Observe that $i^*_\rho \mathcal O^\phi(\strata_i) = i^*_\rho \mathcal O^\phi(\stratb_i)= i^*_\rho\mathcal O^\phi(\strata_{i+1})$ whenever $i+1<k$. From the earlier discussion showing that $E^\ott$ respect $i^*_\rho \mathcal O^\phi$, we obtain that the morphisms $i^*_\rho\mathcal O^\phi(\delta_j)=\id$. A similar discussion shows that $i^*_\rho\mathcal O^\phi(\gamma_i)=\id$ when $i\neq k$. It follows that 
\begin{align*}
	\widetilde {i^*_\rho\mathcal O^\phi} (\epsilon)=&i^*_\rho\mathcal O^\phi(\gamma_k)\circ (i^*_\rho\mathcal O^\phi(\delta_{k-1}))^{-1}\circ\cdots\circ(i^*_\rho\mathcal O^\phi(\delta_1))^{-1}\circ  i^*_\rho\mathcal O^\phi(\gamma_1)\\
	=&i^*_\rho O^\phi(\gamma_k).
\end{align*}
This shows that the left diagram commutes. 

To show that the right diagram commutes requires orienting the stratification $\sS^{\phi_\rho}$. For any stratum $\strata\in \sS^{\phi_\rho}$, let $\stratb$ be the corresponding critical stratum of $\sS^\phi$. Since the image of $\stratb$ is a codimension 0 subset of $\strata$, we can define the orientation of $\strata$ to be the one that agrees with $\stratb$. For this choice of orientation, the signs assigned to Morse flow lines in \cref{def:MorseReduction} agrees with the signs assigned to exit paths for $\sS^{\phi_\rho}$.
\end{proof}

Finally, to complete the proof of \cref{lem:functorialityRestriction}, we observe that the identity stratum $\sigma_0$ is always critical of the lowest degree so $\Psi$ acts by the identity on $\sigma_0$. It follows that  \cref{eq:extrarefdiagram} commutes. 

\subsection{\texorpdfstring{\cref{lem:pushforwardquotient}}{Property}: Pushforward along finite quotients}

We have a stacky fan $(\Sigma, \ul\beta\colon L\to N)$  and a finite group quotient $(\ul\pi, \ul \Pi)\colon (\Sigma_X, \ul\beta_X)\to (\Sigma_{X'}, \ul\beta_{X'})$ (\cref{def:finitequotient}). Let $(\ul\phi', \ul\Phi')\colon (\Sigma_{Y'}, \ul\beta_{Y'})\to (\Sigma_{X'}, \ul\beta_{X'})$ be a closed substack. Construct from this the fan $(\Sigma_Y, \ul \beta_Y\colon L_Y\to N_Y)$ where $L_Y=L_{Y'}$, and $N_Y$ fits into the pullback square

\[
	\begin{tikzcd}
		N_Y \arrow{r}{\ul \phi} \arrow{d}{\ul \pi_Y}& N_X \arrow[hook]{d}{\ul \pi_X}\\
		N_{Y'}\arrow[hook]{r}{\ul \phi'}& N_{X'}
	\end{tikzcd}
\]
Because $\ul \phi', \ul\pi_X$ are injective, the maps $\ul\phi$ and $\ul\pi$ are injective as well. We obtain an injection $\coker(\ul\phi)\into \coker(\ul \phi')$. Since the latter is torsion free, the former is as well. Similarly, the injection $\coker(\ul\pi_Y)\into \coker(\ul\pi_X)$ tells us that the former is a finite group. 
We conclude that we have a pullback square
\[\begin{tikzcd}
	(\Sigma_Y, \ul{\beta}_Y) \arrow{r}{(\ul\phi, \ul \Phi)} \arrow{d}{(\ul\pi_Y, \ul \Pi_Y)} & (\Sigma_X, \ul{\beta}_X) \arrow{d}{(\ul\pi_X, \ul \Pi_X)}\\
	(\Sigma_{Y'}, \ul{\beta}_{Y'}) \arrow{r}{(\ul\phi', \ul \Phi')}& (\Sigma_{X'}, \ul{\beta}_{X'})
\end{tikzcd}
\]
making $\Y\to\X$ a closed immersion, and $\Y\to \Y'$ a finite group quotient.

By \cref{prop:quotientOfLineBundle}, we have $\phi'_*(\pi_Y)_*\mathcal O_\Y=\bigoplus_{q\in \ker(\tilde \pi_Y)}\phi'_*\mathcal O_{\Y'}(\bF(q))$; i.e. the pushforward of the structure sheaf is the sum over all possible ways to twist the structure sheaf by a character of $G_{\beta_Y}$.
We want to show that  $\pi_*C_\bullet(S^\phi, \mathcal O^\phi)$ admits a similar decomposition into sums of $C_{\bullet}(S^{\phi'}, \mathcal O^{\phi'})$ which have been twisted by a character of $G_{\beta_Y}$.
We have the following exact sequence of tori.
\[
	\begin{tikzcd}
		\ker(\tilde \pi_\phi)\arrow{d}\arrow{r}& \ker(\tilde \pi_X)\arrow{d}\arrow{r}& \ker(\tilde \pi_Y)\arrow{d}\\
		T^{\phi'} \arrow{r} \arrow{d}{\tilde \pi_\phi}& M_{X'}\tensor \RR/M_{X'}\tensor \RR \arrow{d}{\tilde \pi_X} \arrow{r}{\tilde\phi'}& M_{Y'}\tensor \RR/M_{Y'}\arrow{d}{\tilde \pi_Y }\\
		T^{\phi}\arrow{r} & M_X\tensor \RR/ M_X  \arrow{r}{\tilde \phi} & M_Y\tensor \RR/M_Y
	\end{tikzcd}
\]
The bottom two rows of this diagram split. Furthermore, the splitting can be chosen to commute with the vertical arrows. It follows that the top row of the diagram splits so we may write $\ker(\tilde \pi_X)=\ker(\tilde \pi_\phi)\times \ker(\tilde \pi_Y)$.
Write $i_1:\ker(\tilde \pi_\phi)\to \ker(\tilde \pi_X)$ and $i_2:\ker(\tilde \pi_Y)\to \ker(\tilde \pi_X)$.
Geometrically, this decomposition gives us a splitting 
\[\tilde \pi_X^{-1}(T^\phi)=\bigsqcup_{q\in \ker(\tilde \pi_Y)}T^{\phi'}+i_2(q)\]
Since $T^{\phi'}(\rho)= (\tilde \pi_\phi)^{-1}(T^{\phi})(\rho)$, $\tilde \pi_\phi$ sends strata to strata. Therefore, we can partition the stratification $S^{\phi'}$ into $\bigcup_{\strata\in S^\phi}\pi_\phi^{-1}(\strata)$. This also can be understood in terms of our splitting, as after picking preferred lifts $\tilde \strata\in\pi_\phi^{-1}(\strata)$, we have $S^{\phi}=\bigcup_{\strata\in S^\phi}\bigcup_{p\in \ker(\tilde \pi_\phi)} \tilde \strata + p$.

Combining this with \cref{prop:quotientOfLineBundle},
\begin{align*}
    (\pi_X)_*C_\bullet(S^\phi, \mathcal O^\phi)=&\bigoplus_{\strata\in S^\phi} (\pi_X)_*\mathcal O_{\X}(\bF(\strata))\\
    \intertext{By \cref{lemma:thomsenQuotient},}
    =&\bigoplus_{\strata\in S^\phi}\bigoplus_{(p, q)\in \ker(\tilde \pi_X)}\mathcal O_{\X'}(\bF'(\tilde\strata+(p, q)))\\
    =&\bigoplus_{\strata\in S^\phi}\bigoplus_{q \in \ker(\tilde \pi_Y)}\bigoplus_{p\in \ker(\tilde \pi_\phi)}\mathcal O_{\X'}(\bF'(\tilde\strata+(p, q)))\\
    =&\bigoplus_{q \in \ker(\tilde \pi_Y)}\bigoplus_{\strata'\in S^{\phi'}}\mathcal O_{\X'}(\bF'(\strata'+(0, q)))\\
	=&\bigoplus_{q\in \ker(\tilde \pi_Y)}C_{\bullet}(S^{\phi'}, \mathcal O^{\phi'})\tensor \mathcal O_{\X'}(\bF'(q)).
\end{align*}
Given any exit path $\gamma\in \EP(\strata,\stratb)$, it also follows from \cref{lemma:thomsenQuotient} that $\pi_* f_\gamma^\phi = \bigoplus_{\gamma'\in \tilde\pi_X^{-1}(\gamma)} \gamma'^\phi$. Since every exit path in $S^{\phi'}$ is isotopic to exactly one lift of an exit path in $S^\phi$, we obtain that the differentials on $\pi_*C_\bullet(S^\phi, \mathcal O^\phi)$ and $\bigoplus_{q\in \ker(\tilde \pi_Y)}C_{\bullet}(S^{\phi'}, \mathcal O^{\phi'})\tensor \mathcal O_{\X'}(\bF'(q))$  agree as well. 
Similarly, the pushforward of the augmentation map $\eps\colon C_\bullet(S^\phi, \mathcal O^\phi)\to \phi_*\mathcal O_Y$ is 
\[
\bigoplus_{q\in \ker(\tilde \pi_Y)} 
\bigoplus_{p\in \ker(\tilde \phi)} \eps\tensor \id_{\mathcal O_{\X'}(\bF(p, q))}. 
\]
This splits as the sum of maps $\eps_q: C_{\bullet}(S^{\phi'}, \mathcal O^{\phi'})\tensor \mathcal O_{\X'}(\bF'(q)) \to \phi'_*\mathcal O_{\Y'}\tensor \mathcal O_{\X'}(\bF'(q)$. Specializing to $q=0$ yields that $C_{\bullet}(S^{\phi'}, \mathcal O^{\phi'})$ resolves $\phi'_*\mathcal O_{\Y'}$.

\subsection{\texorpdfstring{\cref{lem:pullbackquotient}}{Property}: Pullback along toric reductions}
First, observe that we have the following commutative diagram
\[\begin{tikzcd}
	 \; & M_\TT \arrow[equal]{r} & M_\TT\\
	\ker(\ul\phi^*)		 \arrow[hook]{r} & M_\X \arrow{r}{\ul\phi^*} \arrow{u} 						& M_\Y \arrow{u}\\
	\ker((\ul\phi/\TT)^*)\arrow[hook]{r} & M_{[\X/\TT]}\arrow{r}{(\ul\phi/\TT)^*} \arrow[hook]{u}{\ul \pi_\X^*} & M_{[\Y/\TT]} \arrow[hook]{u}{\ul \pi_\Y^*} 
\end{tikzcd}\]
From this diagram, we see that $\Im(\underline \pi_\X^*|_{\ker((\ul\phi/\TT)^*)})= \ker(\ul\phi^*)$. Since $\ul\pi_\X^*$ is injective, we obtain an isomorphism between $\ker((\ul\phi/\TT)^*)$ and $\ker(\ul \phi^*)$. Similarly, we obtain an isomorphism of tori, which by abuse of notation we denote
\[\tilde \pi \colon T^{\phi/\TT}\to T^{\phi}.\]
Note that $\X$ and $\X/\TT$ are quotients of the same toric variety $X_\Sigma$. For every 1-dimensional cone $\rho\in \Sigma(1)$, we have subtori of $T^{\phi}$ and $T^{\phi/\TT}$ determined by the maps 
\[\begin{tikzcd}
	T^{\phi(\rho)}     \arrow{r} &T^\phi       \arrow{rr}{\text{Eval at }\ul\beta_X(\rho)} & &\RR/\ZZ\\
	T^{\phi/\TT(\rho)} \arrow{r} &T^{\phi/\TT} \arrow{rr}{\text{Eval at }\ul\beta_X\ul\pi(\rho)} \arrow{u}{\tilde \pi} &&\RR/\ZZ \arrow[equal]{u}
\end{tikzcd}\]
so the sub-tori $T^{\phi(\rho)}$ and $T^{\phi/\TT(\rho)}$ are identified under the map $\tilde \pi$. It follows that the stratifications $S^\phi$ and $S^{\phi/\TT}$ agree. Thus, \cref{lem:pullbackquotient} reduces to checking the behavior of the Thomsen collection under $\tilde \pi$ which is done in \cref{lem:thomsenPullback}.
The commutativity $\pi^*_\TT\circ \alpha^\phi = \alpha^{\phi/\TT}\circ \pi^*_\TT$ is a consequence of $\tilde \pi$ preserving the identity of the torus.

\subsection{\texorpdfstring{\cref{lemma:functorialityCodim2}}{Property}: Open inclusion with equivariant codimension 2 complement}
The two sets
\begin{align*}
\{\rho\in \Sigma(1)\st \ul\beta_X \rho\not\in \Im(\phi)\} && \{\rho\in \Sigma^\circ(1)\st \ul\beta_{X^\circ}\rho\not\in \Im(\phi_{[Y/T]})\}
\end{align*}
agree.
Since the stratifications $S^{\phi_{[Y/T]}}$ and $S^{\phi}$ only depend on the data of these cones, we obtain by the same argument as the proof of \cref{lem:pullbackquotient} that $S^\phi=S^{\phi_{[Y/T]}}$. So the quivers $Q^\phi, Q^{\phi_{[Y/T]}}$ which define the complexes are isomorphic. By \cref{lem:exactEquivariantPushforward}, at any stratum $\strata\in Q^{\phi_{[Y/T]}}$, we have $\pi^*i_\flat\mathcal O^{\phi_{[Y/T]}}(\strata)=\mathcal O^{\phi}(\strata)$, and for any exit path $\gamma$ we have $\pi^*i_\flat f_\gamma^{\phi_{[Y/T]}}=f_\gamma^\phi$. It follows that $\pi^*i_\flat C_\bullet(S^{\phi_{[Y/T]}}, \mathcal O^{\phi_{[Y/T]}})=C_\bullet(S^\phi, \mathcal O^\phi)$.

\section{Toric Frobenius and generation} 
\label{sec:frobgen}

In this section, we relate the Thomsen collection to the toric Frobenius morphism and conjecture a generalization of \cref{maincor:generate}. To illustrate this behavior concisely, we will largely restrict attention to smooth toric varieties.

On any toric stack, there is a family of toric morphisms 
$$ \fF_\ell \colon \X \to \X $$
parameterized by positive integers $\ell \in \NN$ and given by the morphism of stacky fans where $\Phi$ and $\phi$ are both multiplication by $\ell$. These morphisms are often referred to as the \emph{toric Frobenius} morphisms. Restricting to the algebraic torus $\TT_N \simeq \mathbb{G}_m^n$, $\fF_\ell$ is simply given by
$$ (z_1, \hdots, z_n ) \mapsto (z_1^\ell, \hdots, z_n^\ell).$$

When $X$ is a smooth toric variety, it was originally shown by Thomsen \cite{thomsen2000frobenius} that the pushforward of any line bundle under $\fF_\ell$ splits as a direct sum. Alternative proofs of Thomsen's theorem have been given in \cite{bogvad1998splitting, achinger2015characterization}. We give another short proof using \cref{prop:quotientOfLineBundle}. With more care, it seems likely that the proof presented here applies to a large class of toric stacks.

\begin{thm}[\cite{thomsen2000frobenius}] \label{thm:thomsen}
    Suppose $X$ is a smooth toric variety, $\fF_\ell$ is the toric Frobenius morphism for $\ell \in \NN$, and $D$ is a divisor on $X$. Then,
        \begin{equation} \label{eq:thomsen}
        (\fF_\ell)_* \mathcal{O}_{X}(D) \simeq \bigoplus_{E\in \Pic(X)} \mathcal{O}_{X} (E)^{\oplus \mu(E,D)}
        \end{equation}
    where $\mu(E,D)$ is the number of $\TT_N$-invariant divisors in linearly equivalent to $D-\ell E$ with non-negative coefficients less than $\ell$. 
\end{thm}
\begin{proof}[Proof (Assuming $\ell$ is coprime to the characteristic $p$)] 
When considered as a map of stacky fans, $\fF_\ell$ can be factored as
\[\begin{tikzcd}
	N & N & N \\
	N & N & N
	\arrow["{\ul \beta_\ell = \ell \cdot }", from=1-2, to=2-2]
	\arrow["\id"', from=1-1, to=2-1]
	\arrow["{\ell \cdot }", from=2-1, to=2-2]
	\arrow["{\id}", from=1-1, to=1-2]
	\arrow["{\id}", from=2-2, to=2-3]
	\arrow["{\ell \cdot }", from=1-2, to=1-3]
	\arrow["\id", from=1-3, to=2-3]
\end{tikzcd}\]
giving us maps of toric stacks $ X\xrightarrow{g_\ell} \X_{\Sigma, \ell} = [X/G_\ell] \xrightarrow{h_\ell} X$.
The first map is a finite quotient by $G_\ell \simeq (\ZZ/\ell \ZZ)^{\dim(X)}$.
Let $F$ be a support function for $D$. We apply \cref{prop:quotientOfLineBundle} to obtain that 
\begin{align*}(g_\ell)_* \mathcal O_{X}(\sF)= &\bigoplus_{[m]\in \coker\left(\ul g_\ell^*: M\to M\right)} \mathcal O_{\X_{\Sigma, \ell}}((\ul G_\ell)_*\sF- m) \\
=&\bigoplus_{[m]\in M/\ell M} \mathcal O_{\X_{\Sigma, \ell}}(\sF- m)\end{align*}
We now describe the pushforward $(h_\ell)_*\mathcal O_{\X_{\Sigma, \ell}}(\sF- m)$.
For any $U\subset X_\Sigma$, we have 
$$
    (h_\ell)_*  \mathcal O_{\X_{\Sigma, \ell}}(\sF- m) (U) = \mathcal O_{\X_{\Sigma, \ell}}^{G_\ell}(\sF- m)((H_\ell)^{-1}(U)
    =\mathcal O_{\X_{\Sigma, l}}^{G_\ell}(\sF- m)(U)
$$
since $(H_\ell)^{-1}(U)$ is the union over the $G_\ell$ translates of $U$. 
The $G_\ell$ invariant sections of $\mathcal O_{\X_{\Sigma, \ell}}^{G_\ell}(\sF- m)$ correspond with those of a line bundle on $X$, with a basis of sections indexed by  $\Delta_{\beta_\ell}(\sF-m)\cap \ell M$. We compute
\begin{align*}
    \Delta_{\beta_\ell}(\sF+\ul\beta_\ell^*m)\cap \ell M =& \{\ell m'\in M \st \ell \langle m', u_\rho \rangle \geq F(u_\rho) - \langle m, u_\rho \rangle \text{ for all } \rho \in \Sigma(1) \}\\
    =&\left\{ m'\in M \st  \langle m', u_\rho \rangle \geq \frac{\sF(u_\rho) - \langle m, u_\rho \rangle }{\ell} \text{ for all } \rho \in \Sigma(1) \right\} \\
    =& \Delta (\sF_{m, \ell})
\end{align*}
where $\sF_{m,\ell}$ is the support function defined by $F_{m, \ell}(u_\rho) = \left\lceil \frac{F(u_\rho) - \langle m, u_\rho \rangle }{\ell} \right\rceil$ for all $\rho \in \Sigma(1)$ where $\lceil \cdot \rceil$ is the ceiling function. That is, we have shown
\begin{equation} \label{eq:Frobfloor}
    (\fF_\ell)_*\mathcal O_{X}(D)=\bigoplus_{[m]\in M/\ell M} \mathcal O_{X}\left(\sum_{\rho \in \Sigma(1)} \left\lfloor \frac{a_\rho - \langle m, u_\rho \rangle}{\ell} \right\rfloor D_\rho \right)
\end{equation} 
where $\lfloor \cdot \rfloor$ is the floor function, $D = \sum a_\rho D_\rho$, and we have changed the sign of $m$ for convenience. It remains to show that this agrees with \eqref{eq:thomsen}. That is, for a given divisor $E = \sum e_\rho D_\rho$ on $X$, we need to show that $\mu(E,D)$ is equal to the number of divisors that appear in \eqref{eq:Frobfloor} linearly equivalent to $E$. 

We choose a fundamental domain $V \subset M$ in bijection with $M/\ell M$. For $m \in V$, we have that $\sum \left\lfloor \frac{a_\rho - \langle m, u_\rho \rangle}{\ell} \right\rfloor D_\rho$ is linearly equivalent to $E$ if and only if there is an $m' \in M$ such that
$$ \left\lfloor \frac{a_\rho - \langle m, u_\rho \rangle}{\ell} \right\rfloor = e_\rho +  \langle m', u_\rho \rangle, $$
that is,
$$ \frac{a_\rho - \langle m, u_\rho \rangle}{\ell}=   e_\rho + \langle m', u_\rho \rangle + \frac{k_\rho}{\ell} $$
for some $k_\rho \in \{0, \hdots, \ell-1 \}$ and all $\rho \in \Sigma(1)$. Rearranging, we have
$$ a_\rho - \ell e_\rho = k_\rho + \langle m + \ell m', u_\rho \rangle $$
for all $\rho \in \Sigma(1)$. Note that $D' = \sum k_\rho D_\rho$ is a divisor linearly equivalent to $D - \ell E$ so we have produced the desired bijection since every $m'' \in M$ can be written uniquely as $m + \ell m'$ for $m \in V$ and $m' \in M$.
\end{proof}

\begin{rem} In characteristic $p$, the pushforward under $\fF_\ell$ computed in \cref{thm:thomsen} agrees with that of absolute Frobenius when $\ell = p^k$ \cite[Theorem 1]{thomsen2000frobenius}. Unfortunately, the proof presented here does not apply to that setting. However, note that \eqref{eq:Frobfloor} applies without any assumptions as the combinatorial argument relating this expression to \eqref{eq:thomsen} does not use the assumption on characteristic.
\end{rem}

\begin{rem} Since $\fF_\ell$ is an affine morphism, all higher direct images vanish, and \cref{thm:thomsen} also describes the derived pushforward. 
\end{rem}

\begin{rem}
    \cref{thm:thomsen} has been extended to smooth toric DM stacks in \cite{ohkawa2013frobenius}. Using that, the following discussion applies to that setting with minor modifications.
\end{rem}

We will be primarily interested in which line bundles are summands of $(\fF_\ell)_*\mathcal{O}_X(D)$ and not their multiplicity. Thus, for a divisor $D$ on a smooth toric variety $X$ we set $\Frob_\ell(D) \subset \Pic(X)$ to be the set of line bundles that are summands of $(\fF_\ell)_* \mathcal{O}_\X(D)$ for a given $\ell \in \NN$, and we set
$$ \Frob(D) = \bigcup_{\ell \in \NN} \Frob_\ell(D). $$
It follows from \cref{thm:thomsen} that $\Frob(D)$ is always finite \cite[Proposition 1]{thomsen2000frobenius}.

In the case $D=0$, we have $\Frob_\ell(0) = \Frob(0)$ for large highly divisible $\ell$. In fact, for large highly divisible $\ell$, the image of \eqref{eq:quotientOfLineBundle} coincides with the image of $M_{\ZZ[1/\ell]}/M$ as the stratification $\mathcal S_{\Sigma}$ is given by finitely many rational hyperplanes. As \eqref{eq:Frobfloor} identifies the image of $M_{\ZZ[1/\ell]}/M$ with $\Frob_\ell(0)$, we obtain the following.

\begin{cor}
    On a smooth toric variety, the Thomsen collection coincides with $Frob(0)$.
\end{cor}

In fact, viewing $\Frob(0)$ in this way was originally suggested in \cite{bondal2006derived} where it is claimed that $\Frob(0)$ generates $D^b\Coh(X)$ when $X$ is a smooth and proper toric variety (see \cref{rem:bondalgeneration}). As suggested by \cite[Conjecture 3.6]{uehara2014exceptional}, it is interesting to study the extent to which this generalizes to non-trivial divisors.

\begin{qs} \label{qs:frobgenerate}
    For which divisors $D$ on a smooth toric variety $X$ is $(\fF_\ell)_*\mathcal{O}_X(D)$ a classical generator of $D^b\Coh(X)$ for some $\ell$? 
\end{qs}

\begin{rem}
    The resolutions presented in this paper can be modified in order to replace the Thomsen collection with $\Frob(K)$ where $K = -\sum D_\rho$ is the canonical divisor. Namely, one replaces $\bF$ with
        $$ m \mapsto - \lceil \langle m, u_\rho \rangle \rceil -1 $$
    and all the arrows in $Q^\phi$ will be reversed. However, it seems unlikely that this strategy can be applied to other divisors in any generality. In particular, if $D$ is not $0$ or $K$, it is not clear to us what to expect the generation time of $(\fF_\ell)_*\mathcal{O}_X(D)$ to be when it is a generator.
\end{rem}

There is a simple obstruction to generation by $(\fF_\ell)_*\mathcal{O}_X(D)$ for any $\ell$. First, observe that since $\fF_\ell$ is affine $(\fF_\ell)_*$ does not change cohomology. Thus, if $\mathcal{O}_X(D)$ has vanishing cohomology, none of its Frobenius pushforwards can generate $D^b\Coh(X)$. This obstruction can be generalized by considering morphisms with certain other sheaves. 

\begin{df} A toric morphism $\phi \colon  Y \to X $ is a \emph{linear inclusion} if $\ul \phi$ is an immersion in the sense of \cref{def:inclusion} and $\ul \phi(\sigma) \in \Sigma_X$ for every $\sigma \in \Sigma_Y$ where $\Sigma_Y$ and $\Sigma_X$ are the fans of $Y$ and $X$, respectively. 
\end{df}

In other words, a linear inclusion identifies the fan of $Y$ with a subfan of $X$ lying in a linear subspace of $(N_X)_\RR$. In particular, a linear inclusion identifies $\Sigma_Y(1)$ with a subset of $\Sigma_X(1)$ and the inclusion $N_Y \to N_X$ splits. Thus, $\phi$ induces a well-defined pullback map $\phi^*\colon \text{Div}_{\TT_{N_X}}(X) \to \text{Div}_{\TT_{N_Y}}(Y)$ on torus invariant divisors that respects linear equivalence and is given by forgetting components that do not correspond to rays of $\Sigma_Y$. Note that $Y$ is a closed subvariety of $X$ so $\mathcal{O}_Y$ is a coherent sheaf on $X$. As before, we will abuse notation at times and write $\mathcal{O}_Y$ in place of $\phi_* \mathcal{O}_Y$. The pullback $\phi^*$ demonstrates that linear inclusions are crepant and morphisms from $\mathcal{O}_Y$ reduce to computing cohomology on $Y$. 

\begin{lem} \label{prop:linearcoh}
    Suppose that $\phi\colon Y \to X$ is a linear inclusion. Then, 
    $$\Hom_{D^b\Coh(X)}^\bullet (\mathcal{O}_Y, \mathcal{O}_X(D)) \simeq H^\bullet (\mathcal{O}_Y(\phi^*D)) $$
    for any divisor $D$ on $X$. 
\end{lem}
\begin{proof} Let $K_X$ and $K_Y$ be the canonical bundles on $X$ and $Y$. From the discussion above, we have $\phi^*K_X = K_Y$. Then, we use Serre duality and adjunction to obtain
\begin{align*}
    \Hom_{D^b\Coh(X)}^\bullet (\phi_* \mathcal{O}_Y, \mathcal{O}_X(D)) &\simeq \Hom_{D^b\Coh(X)}^\bullet ( \mathcal{O}_X(D), (\phi_* \mathcal{O}_Y) \otimes K_X)^\vee \\ 
    &\simeq \Hom_{D^b\Coh(X)}^\bullet ( \mathcal{O}_X(D) \otimes K_X^\vee, \phi_* \mathcal{O}_Y )^\vee \\
    &\simeq \Hom_{D^b\Coh(Y)}^\bullet ( \phi^*(\mathcal{O}_X(D) \otimes K_X^\vee), \mathcal{O}_Y )^\vee \\
    &\simeq \Hom_{D^b\Coh(Y)}^\bullet (\mathcal{O}_Y(\phi^*D) \otimes K_Y^\vee, \mathcal{O}_Y )^\vee \\
    &\simeq \Hom_{D^b\Coh(Y)}^\bullet ( \mathcal{O}_Y, \mathcal{O}_Y(\phi^*D)) \simeq H^\bullet(\mathcal{O}_Y(\phi^*D)) 
\end{align*}
as desired.
\end{proof} 

Moreover, linear inclusions interact well with toric Frobenius leading to the following obstruction.

\begin{prop} \label{lem:linearobstruction}
    If $\phi \colon Y \to X$ is a linear inclusion and $D$ any divisor on $X$, then
        $$ \Hom^\bullet_{D^b\Coh(X)} (\mathcal{O}_Y, (\fF_\ell)_*\mathcal{O}_{X}(D)) \simeq \Hom^\bullet_{D^b\Coh(X)} (\mathcal{O}_Y, \mathcal{O}_{X}(D))^{\oplus \ell^k} $$
    where $k = \dim(X) - \dim(Y)$.
\end{prop}
\begin{proof} We choose a splitting $M_X \simeq M' \oplus (N_Y)^\perp$. Using \eqref{eq:Frobfloor}, we have 
        $$(\fF_\ell)_* \mathcal{O}_X (D) = \bigoplus_{[m] \in (N_Y)^\perp/ \ell (N_Y)^\perp} \bigoplus_{[m']\in M'/\ell M'}  \mathcal O_{X}\left(\sum_{\rho \in \Sigma_X(1)} \left\lfloor \frac{a_\rho - \langle m + m', u_\rho \rangle}{\ell} \right\rfloor D_\rho \right). $$
Then, by \cref{prop:linearcoh} and using \eqref{eq:Frobfloor} on $Y$, we have 
    \begin{align*}
    \Hom^\bullet_{D^b\Coh(X)} (\mathcal{O}_Y, &(\fF_\ell)_*\mathcal{O}_{X}(D)) \\ 
    &\simeq  \bigoplus_{[m] \in (N_Y)^\perp/ \ell (N_Y)^\perp} \bigoplus_{[m']\in M'/\ell M'}  H^\bullet \left( \mathcal O_{Y}\left(\sum_{\rho \in \Sigma_Y(1)} \left\lfloor \frac{a_\rho - \langle m', u_\rho \rangle}{\ell} \right\rfloor D_\rho \right) \right) \\
    &\simeq \bigoplus_{[m']\in M'/\ell M'}  H^\bullet \left( \mathcal O_{Y}\left(\sum_{\rho \in \Sigma_Y(1)} \left\lfloor \frac{a_\rho - \langle m', u_\rho \rangle}{\ell} \right\rfloor D_\rho \right) \right)^{\oplus \ell^k} \\
    &\simeq H^\bullet ( (\fF_\ell)_* \mathcal{O}_Y(\phi^*D) )^{\oplus \ell^k} \\
    &\simeq H^\bullet(\mathcal{O}_Y(\phi^*D))^{\oplus \ell^k} \\
    &\simeq \Hom^\bullet_{D^b\Coh(X)} (\mathcal{O}_Y, \mathcal{O}_{X}(D))^{\oplus \ell^k}
    \end{align*}
as desired.
\end{proof}

In particular, if $(\fF_\ell)_*\mathcal{O}_X(D)$ generates $D^b\Coh(X)$, $\mathcal{O}_\X(D)$ must have nonzero morphisms from $\mathcal O_Y$ for every linear inclusion $\phi\colon Y \to X$. The condition that generation by the Frobenius pushforward of a line bundle requires nonzero cohomology is recovered from the case $Y = X$ and $\phi$ is the identity. For $\Sigma_Y=\{0\}$, $Y$ is the point $e$ corresponding to the identity in $\TT_N$, and, of course, $\mathcal{O}_e$ has nonzero morphisms in the derived category to any line bundle. However, \cref{lem:linearobstruction} in general gives additional obstructions. See \cref{ex:twiceblowupfrobenius}.

Intuition from homological mirror symmetry suggests that linear morphisms give all obstructions to generation by Frobenius pushforward.

\begin{conj} \label{conj:frobgen}
    Let $X$ be a smooth toric variety and $D$ a divisor on $X$. $(\fF_\ell)_* \mathcal{O}_\X(D)$ is a classical generator of $D^b\Coh(X)$ for some $\ell \in \NN$ if and only if $\Hom^\bullet_{D^b\Coh(X)} ( \phi_* \mathcal O_Y, \mathcal{O}_X(D)) \neq 0$ for every linear inclusion $\phi \colon  Y \to X $.   
\end{conj}

\subsection{Which line bundles satisfy \texorpdfstring{\cref{conj:frobgen}}{Conjecture \ref{conj:frobjen}}?}

We will now discuss how to describe and compute the line bundles that satisfy the hypothesis of \cref{conj:frobgen} while illustrating further the behavior of the toric Frobenius morphisms. Note that the set of line bundles on a smooth toric variety with nonvanishing cohomology can be described explicitly \cite[Proposition 4.3]{borisov2009conjecture} (see also \cite{altmann2020immaculate}). Namely, $\mathcal{O}_X(D)$ has nonzero cohomology if and only if $D$ is linearly equivalent to a divisor of the form 
    \begin{equation} \label{eq:BHcone} 
    \sum_{\rho \in A} -D_\rho + \sum_{\rho \not \in A} r_\rho D_\rho - \sum_{\rho \in A} r_\rho D_\rho 
    \end{equation}
with $r_\rho \in \ZZ_{\geq 0}$ for some subset $A \subset \Sigma(1)$ such that $\mathcal{O}_X (\sum_{\rho \in A} -D_\rho)$ has nonzero cohomology. We will denote the set of divisors in $\text{Div}_{\TT_N}(X)$ of the form $\sum_{\rho \in A} -D_\rho$ for some $A \subset \Sigma(1)$ by $\mathfrak{C}$. That is, $\mathfrak{C}$ is the cube $[-1,0]^{|\Sigma(1)|} \subset \ZZ^{|\Sigma(1)|} \simeq \text{Div}_{\TT_N} (X)$ where the last isomorphism comes from choosing the basis $\{D_\rho\}_{\rho \in \Sigma(1)}$. As noted in \cite{achinger2015characterization}, this observation is intimately related to toric Frobenius.

\begin{prop} \label{prop:frobvertices}
    Let $X$ be a smooth toric variety. For any subset $A \subset \Sigma(1)$, $\mathcal{O}_X( \sum_{\rho \in A} -D_\rho) \in \Frob(D)$ if and only if $D$ is linearly equivalent to a divisor of the form \eqref{eq:BHcone}. 
\end{prop}
\begin{proof} For the given $A \subset \Sigma(1)$, let $E =\sum_{\rho \in A} - D_\rho$.  \cref{thm:thomsen} shows that $E \in \Frob_\ell(D)$ if and only if $D$ is linearly equivalent to a divisor in $\ell E + (1- \ell) \mathfrak{C}$. However, observe that the set of divisors of the form \eqref{eq:BHcone} is exactly
$$ E + \bigcup_{\ell \in \NN} (\ell -1)(E - \mathfrak{C}) $$
from which the claim follows.
\end{proof}

In fact, \cref{prop:frobvertices} shows that \cite[Proposition 4.3]{borisov2009conjecture} restricted to smooth toric varieties follows from the fact that $(\fF_\ell)_*$ preserves cohomology (\emph{cf.} \cite[Theorem 3]{achinger2015characterization}). 

\begin{rem} The if part of \cref{prop:frobvertices} is essentially \cite[Lemma 3.5(ii)]{uehara2014exceptional}. There, the observation is made that if $D$ is equivalent to a divisor $E$ of the form \eqref{eq:BHcone}, then $\Frob(D) = \Frob(E)$ and applying \eqref{eq:Frobfloor} with $m = 0$ and $\ell$ sufficiently large shows that $\mathcal{O}_\X(\sum_{\rho \in A} - D_\rho)$ is in $\Frob(D)$.
\end{rem}

The description above extends to linear inclusions. Given \cref{prop:linearcoh}, describing divisors on $X$ with nonzero morphisms from $\mathcal{O}_Y$ for a linear inclusion $\phi \colon  Y \to X $ reduces to the previous task of classifying divisors on $Y$ with nonvanishing cohomology using the pullback map $\phi^*\colon \text{Div}_{\TT_{N_X}}(X) \to \text{Div}_{\TT_{N_Y}}(Y)$. Namely, 
    $$ \Hom_{D^b\Coh(X)} (\mathcal{O}_Y, \mathcal{O}_X(D)) \neq 0$$
if and only if $\phi^* D$ is linearly equivalent to a divisor in $E + \ZZ_{\geq 0} (E - \mathfrak{C}_Y)$ for some $E \in \mathfrak{C}_Y$ with nonzero cohomology. Alternatively stated, $D$ is required to be linearly equivalent to a divisor of the form
    \begin{equation} \label{eq:linearBHcone} 
    -\sum_{\rho \in A}  D_\rho + \sum_{\rho \in \Sigma'(1) \setminus A} r_\rho D_\rho - \sum_{\rho \in A} r_\rho D_\rho + \sum_{\rho \not \in \Sigma'(1)} s_\rho D_\rho
    \end{equation}
with $r_\rho \in \ZZ_{\geq 0}$ and $s_\rho \in \ZZ$ for some subset $A \subset \Sigma'(1)$ such that $H^\bullet(\mathcal{O}_Y(\sum_{\rho \in A} -D_{\rho})) \neq 0$. Moreover, these sets of divisors are again related to toric Frobenius as $\mathcal{O}_Y(\sum_{\rho \in A} - D_\rho) \in \Frob(\phi^* D)$ if and only if $D$ is linearly equivalent to a divisor of the form \eqref{eq:linearBHcone}. 

Although we now have a relatively concrete description of the line bundles satisfying the hypothesis of \cref{conj:frobgen}, one can give a more computable necessary condition for $\mathcal{O}_X(D)$ to have nonvanishing cohomology by passing to the real Picard group as observed in \cite[Section 4]{borisov2009conjecture}. This is also the perspective taken in \cite{achinger2015characterization} to give a geometric description of $\Frob_\ell(D)$, which we will partially recall here. For a smooth toric variety $X$, we denote by $\Pic(X)_\RR = \Pic(X) \otimes_\ZZ \RR$ and let $Z$ be the zonotope in $\Pic(X)_\RR$ defined as the convex hull of the image of $\mathfrak{C}$. Note that as $\Pic(X)$ is free, we have an inclusion $\Pic(X) \subset \Pic(X)_\RR$, and we will at times abuse notation by using the same notation for a line bundle or set of line bundles and its image in $\Pic(X)_\RR$. 

\begin{prop} \label{prop:frobgeom}
    Let $D$ be a divisor on a smooth toric variety $X$ and let $[D] = \mathcal{O}_\X(D) \in \Pic(X)_\RR$. 
    \begin{enumerate} 
        \item For $\ell \gg 0$, $\Frob_\ell(D) \subset Z$.
        \item If $[D] \in Z$, then $\Frob(D) \subset \bigstar_Z([D])$ where $\bigstar_Z([D])$ is the union of the interiors of the faces of $Z$ containing $[D]$. Moreover, if $D \in \mathfrak{C}$, then $\Frob(D)$ coincides with  $ \Pic(X) \cap \bigstar_Z([D])$. 
        \item If $p \in Z \cap \Frob(D)$, then $[D]$ lies in the translated cone
        $$ C_p = p + \cone(p - Z) $$
        in $\Pic(X)_\RR$. 
\end{enumerate}
\end{prop}
\begin{proof} First, observe that \cref{thm:thomsen} implies that 
\begin{equation} \label{eq:achingerobs} \Frob_\ell(D) \subset \left(\frac{\ell -1}{\ell} Z + \frac{1}{\ell} [D]\right) \end{equation}
from which the first claim follows.

For the second claim, we have
$$ \bigstar_Z(p) = \bigcup_{\ell \geq 1} \left( p + \frac{\ell-1}{\ell}(Z -p) \right)$$
for all $p \in Z$. Combined with \eqref{eq:achingerobs}, it is then apparent that $\Frob(D) \subset \bigstar_Z([D])$ given that $[D] \in Z$. Now, suppose that $D \in \mathfrak{C}$ so that $D = \sum_{\rho \in A} -D_\rho$ for some subset $A \subset \Sigma(1)$. By \eqref{eq:Frobfloor}, we have that
    \begin{align*} \Frob_\ell(D) &= \bigcup_{[m] \in M/\ell M} \mathcal{O}_{X} \left(\sum_{\rho \in A} \left\lfloor \frac{-1 - \langle m, u_\rho \rangle}{\ell} \right\rfloor D_\rho + \sum_{\rho \not\in A} \left\lfloor - \frac{\langle m, u_\rho \rangle}{\ell} \right\rfloor D_\rho \right) \\
    &= \bigcup_{[m] \in M/\ell M} \mathcal{O}_{X} \left(\sum_{\rho \in A} \left( \left\lceil - \frac{ \langle m, u_\rho \rangle}{\ell} \right\rceil -1 \right) D_\rho + \sum_{\rho \not\in A} \left\lfloor - \frac{\langle m, u_\rho \rangle}{\ell} \right\rfloor D_\rho \right).
    \end{align*}
In other words, $\Frob_\ell(D)$ coincides with the image of $M_{\ZZ[1/\ell]}/M$ under the map $M_\RR/M \to \Pic(X)$ given by
$$ m \mapsto \mathcal{O}_{\X} \left(\sum_{\rho \in A} \left( \left\lceil - \langle \beta^*m, u_\rho \rangle \right\rceil -1 \right) D_\rho + \sum_{\rho \not\in A} \left\lfloor - \langle \beta^*m, u_\rho \rangle \right\rfloor D_\rho \right).$$
As $\mathcal{S}_{\Sigma}$ is given by finitely many rational hyperplanes, the image of $M_{\ZZ[1/\ell]}/M$ coincides with that of $M_\RR/M$ for large highly divisible $\ell$. Finally, observe that the latter is given by the divisor classes that are linear equivalent as $\RR$ divisors to an element of $\bigstar_{\mathfrak{C}}(D)$, and hence are the integral points in its image $\bigstar_Z ([D])$.

For the third claim, similarly, observe that 
$$C_p = \bigcup_{\ell \geq 1} \big( p + (\ell -1)(p - Z) \big) .$$
As a result,
\begin{align*}
[D] \in C_p &\iff [D] \in \big( p + (\ell -1)(p - Z) \big) \text{ for some } \ell \geq 1 \\
&\iff p \in \frac{\ell -1}{\ell} Z + \frac{1}{\ell}[D] \text{ for some } \ell \geq 1 \\
&\impliedby p \in \Frob_\ell(D) \text{ for some } \ell \geq 1
\end{align*}
where the final implication follows from \eqref{eq:achingerobs}. 
\end{proof}

\begin{rem}
    Note that \eqref{eq:achingerobs} is not always an equality even though $\Pic(X)$ is free. See \cref{ex:hirzebruchfrobenius} below. In fact, \cite[Example VII.8]{altmann2020immaculate} shows that there is a line bundle on a smooth projective toric variety lying in $C_{[0]}$ which has vanishing cohomology. In particular, $[0]$ lies in the set on the right-hand-side of \eqref{eq:achingerobs} in this example, but $\mathcal{O}$ is not an element of $\Frob(D)$. 
\end{rem}

\cref{prop:frobgeom} implies that the only sheaves with nonzero cohomology that lie in the preimage of $Z$ must lie over vertices of $Z$ as noted in \cite{achinger2015characterization}. Moreover, if $H^\bullet(\mathcal{O}_Y(\phi^*D)) \neq 0$ for a linear inclusion $\phi\colon Y \to X$, then $\Frob(D) \cap Z$ must contain a point that pulls back to a vertex of the zonotope in $\Pic(Y)_\RR$. Given \cref{prop:frobgeom}, \cref{conj:frobgen} is implied by the following conjecture concerning the finite set of line bundles in $\Pic(X)$ that lie in $Z$ of which many of their Frobenius pushforwards can be computed directly from the second claim of \cref{prop:frobgeom}.

\begin{conj} 
    Suppose that the subvarieties of positive dimension with linear inclusions into a smooth toric variety $X$ are $Y_0 = X, Y_1, \hdots, Y_k$. If $D_0, \hdots, D_k$ are divisors on $X$ such that $\mathcal{O}_X(D_j) \in Z$ and
        $$ \Hom_{D^b\Coh(X)}^\bullet( \mathcal{O}_{Y_j}, \mathcal{O}_X (D_j)) \neq 0 $$
    for all $j$, then 
        $$ \bigoplus_{j=0}^k \bigoplus_{\mathcal F \in \Frob(D_j)} \mathcal F $$
    is a classical generator of $D^b\Coh(X)$. 
\end{conj}

\subsection{Examples}

We conclude with some examples illustrating the behavior described in the previous section.

\begin{example} Consider $\PP^n$ with fan generated by $\alpha_0 = -e_1-\hdots-e_n, \alpha_1 = e_1, \hdots, \alpha_n = e_n$. Then, all $D_i$ are equivalent and the zonotope $Z$ is the interval $[-n-1, 0]$. It can be directly computed that
$$ \Frob(kD_0) \cap Z = \begin{cases} \{  \mathcal{O}, \mathcal{O}(-1), \hdots, \mathcal{O}(-n) \} & k \geq 0 \\
\{  \mathcal{O}(-1), \hdots, \mathcal{O}(-n) \} & -n-1 < k < 0 \\
\{ \mathcal{O}(-1), \hdots, \mathcal{O}(-n-1) \} & k \leq -n-1
\end{cases}$$
and thus generates the derived category if and only if $k \not\in (-n -1, 0)$.  Note that $C_{0} = [0, \infty)$, $C_{-n-1} = (-\infty, -3]$, and $C_{-1} = \hdots = C_{-n} = \RR$. See \cref{fig:pnfrobex} for the case of $n = 2$. Note that in this example the only linear inclusions are that of $X$ itself and the identity point of $\TT_N$. That is, nonvanishing cohomology is the only requirement for generation by some pushforward under toric Frobenius.
\end{example}

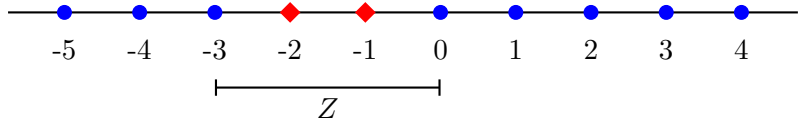
\begin{figure} 
\centering 
\begin{tikzpicture}
	\foreach \x in {-5,-4,...,4}
    {        
      \coordinate (A\x) at ($(\x,0)$) {};
      \node at ($(A\x)+(0,-3ex)$) {\x};
    }

    \draw[thick] ($(A-5)-(0.75,0)$) -- ($(A4)+(0.75,0)$);

    \draw[thick, |-|] ($(A-3) - (0,1)$) -- ($(A0) - (0,1)$) node [midway, below] {$Z$};
    
    \foreach \x in {-5,-4,-3}
    {        
      \node at ($(A\x)$) [circle, fill, blue, inner sep=2pt]{};
    }

    \foreach \x in {0,1,...,4}
    {        
      \node at ($(A\x)$) [circle, fill, blue, inner sep=2pt]{};
    }

    \foreach \x in {-2,-1}
    {        
      \node at ($(A\x)$) [shape=diamond, fill, red, inner sep=2pt]{};
    }

\end{tikzpicture} \caption{The real Picard group of $\PP^2$ with the zonotope $Z$. Line bundles such that $\Frob_\ell(D)$ generates $D^b\Coh(\PP^2)$ for some $\ell$ are blue dots, while those line bundles for which this is not true are red diamonds.} 
\label{fig:pnfrobex}
\end{figure}

\begin{example} \label{ex:hirzebruchfrobenius} The Hirzebruch surface $\mathbb{F}_b$ of degree $b$ is the smooth toric variety associated to the complete fan generated by $\alpha_1 = (1,0), \alpha_2 = (0,1), \alpha_3 = (-1, b)$, and $\alpha_4=(0,-1)$. Then, we have $D_2 + b D_3 - D_4 = D_1 - D_3 = 0$ up to linear equivalence, and thus, we will identify $\Pic(\mathbb{F}_b) \simeq \ZZ^2$ by choosing the generators $D_1, D_2$. Under this identification, we compute that the vertices of $\mathfrak{C}$ project to the points
\begin{align*} 
&(0,0), (-1,0), (-2, 0), (0,-1), (-1,-1), (-2, -1), \\
&(-b, -1), (-b, -2), (-b -1, -1), (-b-1, -2), (-b-2, -1), (-b-2, -2), 
\end{align*}
and as a result,
$$ (0,0), (-2,0), (0,-1), (-b, -2), (-b-2, -1), (-b-2,-2) $$
are the vertices of $Z$ for $b\geq 1$ (for $b =0$ there are only 4 vertices). Therefore, $Z$ has $8$ integer points on its boundary and $b+1$ in its interior. In particular, for $b \geq 4$, there are integer points in $Z$ that do not come from integer points of $\mathfrak{C}$. As a consequence, \eqref{eq:achingerobs} is not always an equality. For a concrete example, take $b=4, \ell =2$ and $D = D_1 + D_2$ and observe $2(-1,0) - (1,1) = (-3, -1) \in (2-1)Z = Z$ so $-D_1 \in \frac{1}{2} Z + \frac{1}{2} D$, but one can check directly using \eqref{eq:Frobfloor} that $\mathcal{O}_{\mathbb{F}_4}(-D_1)$ is not a summand of $(F_2)_* \mathcal{O}_{\mathbb{F}_4}(D_1 +D_2)$.

On the other hand, the translated cones determine exactly which line bundles on $\mathbb{F}_b$ have vanishing cohomology. That is, the line bundles with nonvanishing cohomology are exactly those lying in the translated cones over the vertices $(0,0), (-2,0), (-b, -2)$, and $(-b-2, -2)$. There is a non-trivial linear inclusion of $\PP^1$ into $\mathbb{F}_b$ given by $\ul\phi(u) = (0,u)$.
Choosing the basis $D_2$ for $\Pic(\PP^1)$, the pullback $\phi^* \colon  \Pic(X) \to \Pic(Y) $ becomes the projection $\ZZ^2 \to \ZZ$ onto the second factor. Thus, all four vertices of $Z$ with nonvanishing cohomology project to vertices of the zonotope of $\PP^1$ with nonvanishing cohomology. It follows that the only requirement for generation by $\Frob_\ell(D)$ for some $\ell$ is nonvanishing cohomology.

For $b=1$, the zonotope and some translated cones are depicted in \cref{fig:hirzfrobex}. 
\end{example}

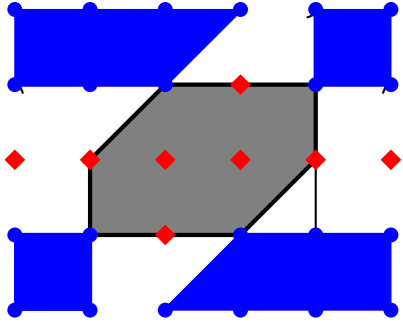
\begin{figure} 
\centering 
\begin{tikzpicture}
\draw[->, thick] (0,0) -- (1, 0);
\draw[->, thick] (0,0) -- (0,1);
\draw[->, thick] (0,0) -- (0,-3);
\draw[->, thick] (0,0) -- (-4, 0);

\draw[gray, fill = gray, opacity = 0.3] (0, 0) -- (-2,0) -- (-3, -1) -- (-3,-2) -- (-1,-2) -- (0,-1) -- (0,0);
\draw[black, ultra thick] (0, 0) -- (-2,0) -- (-3, -1) -- (-3,-2) -- (-1,-2) -- (0,-1) -- (0,0);

\draw[blue, fill = blue, opacity = 0.3] (0, 0) -- (1,0) -- (1, 1) -- (0,1) -- (0,0);
\draw[blue, ultra thick] (0, 1) -- (0,0) -- (1,0);

\draw[blue, fill = blue, opacity = 0.3] (-3, -2) -- (-4,-2) -- (-4, -3) -- (-3,-3) -- (-3,-2);
\draw[blue, ultra thick] (-4, -2) -- (-3,-2) -- (-3,-3);

\draw[blue, fill = blue, opacity = 0.3] (-2, 0) -- (-1,1) -- (-4, 1) -- (-4,0) -- (-2,0);
\draw[blue, ultra thick] (-1, 1) -- (-2,0) -- (-4,0);

\draw[blue, fill = blue, opacity = 0.3] (-1, -2) -- (-2,-3) -- (1, -3) -- (1,-2) -- (-1,-2);
\draw[blue, ultra thick] (-2, -3) -- (-1,-2) -- (1,-2);

\node at (0,0) [circle, fill, blue, inner sep=2pt]{};
\node at (-2,0) [circle, fill, blue, inner sep=2pt]{};
\node at (-3,-1) [shape=diamond, fill, red , inner sep=2pt] {};
\node at (-3,-2) [circle, fill, blue, inner sep=2pt]{};
\node at (-1,-2) [circle, fill, blue, inner sep=2pt]{};
\node at (0,-1) [shape=diamond, fill, red, inner sep=2pt]{};

\node at (-1,0) [shape=diamond, fill, red, inner sep=2pt]{};
\node at (-2,-2) [shape=diamond, fill, red, inner sep=2pt]{};

\node at (-1,-1) [shape=diamond, fill, red, inner sep=2pt]{};
\node at (-2,-1) [shape=diamond, fill, red, inner sep=2pt]{};

\node at (1,1) [circle, fill, blue, inner sep=2pt]{};
\node at (1,0) [circle, fill, blue, inner sep=2pt]{};
\node at (1,-1) [shape=diamond, fill, red, inner sep=2pt]{};
\node at (1,-2) [circle, fill, blue, inner sep=2pt]{};
\node at (0,-2) [circle, fill, blue, inner sep=2pt]{};
\node at (1,-3) [circle, fill, blue, inner sep=2pt]{};
\node at (0,-3) [circle, fill, blue, inner sep=2pt]{};
\node at (-1,-3) [circle, fill, blue, inner sep=2pt]{};
\node at (-2,-3) [circle, fill, blue, inner sep=2pt]{};
\node at (-3,-3) [circle, fill, blue, inner sep=2pt]{};
\node at (-4,-3) [circle, fill, blue, inner sep=2pt]{};
\node at (-4,-2) [circle, fill, blue, inner sep=2pt]{};
\node at (-4,-1) [shape=diamond, fill, red, inner sep=2pt]{};
\node at (-4,0) [circle, fill, blue, inner sep=2pt]{};
\node at (-3,0) [circle, fill, blue, inner sep=2pt]{};
\node at (-4,1) [circle, fill, blue, inner sep=2pt]{};
\node at (-3,1) [circle, fill, blue, inner sep=2pt]{};
\node at (-2,1) [circle, fill, blue, inner sep=2pt]{};
\node at (-1,1) [circle, fill, blue, inner sep=2pt]{};
\node at (0,1) [circle, fill, blue, inner sep=2pt]{};

\end{tikzpicture} \caption{The real Picard group of the blow-up of $\PP^2$ at a point with the zonotope $Z$ shaded in grey with solid black boundary. Line bundles such that $\Frob_\ell(D)$ generates $D^b\Coh(\PP^2)$ for some $\ell$ are blue dots, while those line bundles for which this is not true are red diamonds. The translated cones $C_p$ for vertices of $Z$ with nonzero cohomology are shaded in light blue and have a solid blue boundary.} 
\label{fig:hirzfrobex}
\end{figure}

\begin{example} \label{ex:twiceblowupfrobenius} Let $X$ be the blow-up at two points of $\PP^2$ with complete fan generated by 
    $$(\pm 1, 0), (0,1), (1,1) \text{ and } (-1,-1).$$ 
Let $D_1, D_2$ and $D_3$ be the divisors corresponding to the rays generated by $(1,0), (0,1),$ and $(-1,-1)$, respectively. The exceptional divisors $E_1,E_2$ correspond to the rays generated by $(1,1)$ and $(-1,0)$.  Then, we have that $D_1 - D_3 + E_1 - E_2 = D_2 + E_1 - D_3 = 0$ up to linear equivalence. Thus, we can identify $\Pic(X) \simeq \ZZ^3$ by choosing $D_1, D_2$ and $E_1$ as generators. 
The vertices of the zonotope $Z$ with nonzero cohomology are the points 
    \begin{align*} 
    &(0,0,0), (-1,-1,0), (-2,1,0), (0,-2,-1), (0,-1,-2), (-1,-1,-1), \\ &(-1,-2,-1), (-2,0,0), (-2,1,-1), (0,-2,-2), (-1,0,-2), \text{ and } (-2,-1, -2).
    \end{align*}
There are two linear inclusions of $\PP^1$ into $X$ corresponding to the lines $u_2 = 0$ and $u_1 = u_2$ in $N_\RR$. $D_1$ and $E_1$ form bases of the Picard groups of these linear inclusions so each pullback map is given by the projection $\ZZ^3 \to \ZZ$ onto the first or third factor. Thus, there are many vertices of $Z$ with non-vanishing cohomology whose Frobenius pushforwards do not generate $D^b\Coh(X)$. For example $-E_1 -E_2 = (-1,1-1)$ has nonzero cohomology and corresponds to the vertex $(-1, 1, -1)$ which projects to $-1$ in the zonotope of $\PP^1$ for both linear inclusions.
\end{example}

\appendix
\section{Discrete Morse cohomology for quivers}
\label{sec:quivers}
\label{app:morseForQuivers}
In this appendix, we state some results from discrete Morse theory used to prove \cref{lem:functorialityRestriction}. The results we state are more general than those in the literature (\cite{forman1998morse,skoldberg2006morse}). In particular, 
\begin{itemize}
\item \cite{skoldberg2006morse} constructs from a chain complex a quiver and then defines discrete Morse theory for that particular quiver. All quivers constructed this way will be digraphs (i.e., no multiple arrows between vertices). We instead provide our quivers with some extra structure so that we can associate a chain complex with the quiver.
\item \cite{skoldberg2006morse} constructs a Morse complex of $R$-modules. Our complexes instead take values in an abelian category.
\end{itemize}
The results of \cite{skoldberg2006morse} extend to our setting with only minor modifications. As such, we do not include their proofs here.

\subsection{Quivers from CW complexes and Cohomology}
Let $Q=(V(Q), E(Q))$ be a quiver. When the quiver is clear, we will simply write $Q=(V, E)$. To keep consistent notation with our application in \cref{lem:functorialityRestriction}, we will denote the elements of $E$ by $\delta, \gamma$ and the elements of $V$ by $\strata,\stratb,\stratc$. Let $P(\strata,\stratb)$ be the set of directed paths between $\strata,\stratb$.
\begin{df}\label{def:MorseQuiver}
    A \emph{Morse quiver} is a quiver $Q$ with a map $|-|\colon V\to \ZZ$ such that the following hold.
    \begin{itemize}
        \item  All edges $\gamma \colon \strata\to\stratb $ satisfy $|\strata|-|\stratb|=1$.
        \item For every $\strata,\stratb$ with $|\strata|-|\stratb|=2$, we have specified an involution 
    \[
        I_{\strata\stratb}:P(\strata, \stratb)\to P(\strata, \stratb)
    \]
    with no fixed points.
    \end{itemize}
\end{df}

Given a map $\sgn \colon  E\to \{\pm 1\}$, we can extend this to a map \[\sgn\colon P(\stratb, \strata)\to \{\pm 1\}\] via $\sgn(\gamma_1\cdots \gamma_k)=\sgn(\gamma_1)\cdots\sgn(\gamma_k)$. We say that $\sgn$ is an orientation if for all $\epsilon\in P(\strata,\stratb)$ with $|\strata|-|\stratb|=2$ the involution is anti-invariant with respect to our choice of signs $\sgn(\epsilon)=-\sgn(I_{\strata\stratb}(\epsilon))$. When $Q$ is a Morse quiver with a choice of orientation, we call $Q$ an \emph{oriented Morse quiver}.

\begin{example}
    \label{ex:orientedMorseQuiver}
    Let $X$ be a regular oriented CW complex. Then the Hasse diagram of the poset associated with $X$ is an example of an oriented Morse quiver.

    More generally, let $X$ be an oriented CW complex, whose attaching maps $\alpha_k \colon D^k\to X_{k-1} $ satisfy the following property: for all cells $\stratb=D^{k-1}\subset X_{k-1}$ and $x\in D^k$ with $\alpha(k)\in \stratb$, there exists a local chart $U\ni x$ so that $\stratb\subset \alpha(U)$ and $\alpha|_U$ is a homeomorphism onto its image. We associate to $X$ a quiver whose vertices are labeled by the cells of $X$, and whose edges $E(\strata,\stratb)$ denote the connected components of $\alpha_{\strata}^{-1}(\stratb)$. This quiver is an oriented Morse quiver.
    \label{ex:cwQuiver}
\end{example}

    A sheaf on a Morse quiver with values in an abelian category $\mathcal C$ is a functor $\mathcal F \colon  Q\to C $ with the property that for all $|\strata|-|\stratb|=2$ and $\gamma_1\cdot\gamma_2\in P(\strata, \stratb)$, we have 
    \[\mathcal F(\gamma_1\cdot\gamma_2)=\mathcal F(I_{\strata\stratb}(\gamma_1\cdot\gamma_2)).\]
Associated to the data of a sheaf on an oriented Morse quiver is a complex $C_\bullet(Q, \sgn, \mathcal F)$, where
\begin{align}
    \label{eq:morseComplex}
    C_k(Q, \sgn, \mathcal F)=& \bigoplus_{\strata\in Q, |\strata|=k}\mathcal F(\strata)\\
    d_{\strata}^\stratb=& \sum_{\gamma\in P(\strata,\stratb)}\sgn(\gamma)\cdot \mathcal F(\gamma)
\end{align}
\begin{example}
    When $X$ is a CW complex as in \cref{ex:cwQuiver}, then the constant sheaf $\ZZ(\sigma)=\ZZ, \ZZ(\gamma)=\id$ has the property that $C_k(Q, \ZZ)=C_k^{CW}(X, \ZZ)$.
\end{example}
Given different choices of sign function $\sgn, \sgn' \colon  E(Q)\to \{\pm 1\} $, the complexes $C_k(Q, \sgn, \mathcal F)$ and $C_k(Q, \sgn', \mathcal F)$ are isomorphic. For this reason, we will write $C_k(Q, \mathcal F)$ for the associated chain complex unless the choice of signs is important.

Discrete Morse theory is a method for finding a chain complex homotopic to $C_k(Q, \mathcal F)$ whose cochain groups have small rank.
\begin{df}
    Let $Q$ be a Morse quiver. An \emph{acyclic partial matching} of $Q$ is a partial matching $E^\ott\subset E(Q)$ of edges with the following property: The quiver $Q^{E^\ott}$, obtained from $Q$ by reversing all arrows in  $E^\ott$, has no directed cycles. 
\end{df}

Let $Q$ be a quiver and $E^\ott$ an acyclic partial matching. We denote by $E^\tto=E\setminus E^\ott$. The \emph{critical elements} of $Q$, denoted $\Crit(Q, E^\ott)$, are the $\strata\in Q$ which do not belong to an edge $\gamma\in E^\ott$.  
\subsection{Gradient flow lines}
 A \emph{gradient flow line} is a directed path in $Q^{E^\ott}$ which begins and ends with an element in $E^\tto$. 
For any vertices $\strata,\stratb$, let $\overline M(\strata,\stratb)$ denote the set of gradient flow lines starting at $\strata$ and ending at $\stratb$.
We write such paths as 
    \[\epsilon =\strata_1\tto\strata_2\ott\strata_3\tto\cdots \tto\strata_k\] 
We will use $\strata\tto\stratb$ when we know that the edge is contained in $E^\tto$ and  $\strata\ott\stratb$ when the edge is contained in $E^\ott$. 
Define $\ind(\epsilon)= |\strata|-|\stratb|-1$; this also counts the difference between the number of $\tto$ arrows and number of $\ott$ arrows. 
When $\ind(\epsilon)=1$, the path $\epsilon$ must be alternating between edges in $E^\tto$ and $E^\ott$.
Let $\ell(\epsilon)$ be the length of the path (as defined by the number of edges). 

\begin{example}
    If we take $E^{\ott}=\emptyset$, then $\Crit(Q, E^{\ott})=V$ and every edge of $E$ is a gradient flow line. 
\end{example}
The expectation from smooth Morse theory is that the downward flow spaces of the critical points provide a decomposition of $Q$ which can be used to compute $H_k(Q, \mathcal F)$. Additionally, the Morse-Smale-Witten approach to smooth Morse cohomology suggests that the quiver whose vertices are critical points and whose edges are gradient flow lines is another example of an oriented Morse quiver. This can be encapsulated in the following proposition. 

\begin{prop}\label{cor:involutionOnBrokenPaths}
Let $\strata$ and $\stratb$ be critical elements such that $|\strata|=|\stratb|+2$. There is a fixed-point free involution on the set $\bigcup_{\stratc\in \Crit(Q, E^\ott)}\overline{M}(\strata,\stratc)\times \overline{M}(\stratc,\stratb)$.
\end{prop}
In the smooth setting, this proposition is proved by studying the boundary compactification (by broken flow lines) of the moduli space of index 1 flow lines between $\strata, \stratb$.  
\subsection{The Morse Complex}
\begin{df}
    \label{def:MorseReduction}
Let $E^\ott$ be an acyclic matching on an oriented Morse quiver $Q$. The Morse reduction of $Q$ along $E^\ott$ is the Morse quiver $\tilde Q$ whose
\begin{itemize} 
    \item vertices are $\Crit(Q, E^\ott)$;
    \item edges between $\strata, \stratb$ with $|\strata|=|\stratb|+1$ are $\overline{M}(\strata,\stratb)$.
    \item linearization is given by restriction of $|-|\colon Q\to \ZZ$ to the critical elements
    \item orientation is given by $\widetilde{\sgn}(\epsilon)=(-1)^{(\ell(\epsilon)-1)/2}\prod_{i=1}^k \sgn(\gamma_i)$.
    \item involution is given by \cref{cor:involutionOnBrokenPaths}
\end{itemize}
\end{df}
An acyclic partial matching is said to respect a sheaf $\mathcal F \colon  Q\to \mathcal C $ if whenever $\delta\in E^\ott$, $\mathcal F(\delta)$ is an isomorphism.
Given an acyclic partial matching that respects $\mathcal F$, we can define a sheaf
$\widetilde {\mathcal F} \colon  \tilde Q\to \mathcal C $ that agrees with $\mathcal F$ on objects and takes an edge $\strata_1\xrightarrow{\gamma_1}\stratb_1\xleftarrow{\delta_1}\strata_2\xrightarrow{\gamma_2}\cdots \xrightarrow{\gamma_k}\stratb_k$ to 
\[\widetilde {\mathcal F} (\epsilon)=\mathcal F(\gamma_k)\circ (\mathcal F(\delta_{k-1}))^{-1}\circ\cdots\circ(\mathcal F(\delta_1))^{-1}\circ\mathcal F(\gamma_1).\]

In summary, $\tilde Q$ is a Morse quiver, and $\widetilde{\mathcal  F}$ is a sheaf on $\tilde Q$.
The complex $C_\bullet(\tilde Q, \widetilde{\mathcal F})$ is called the Morse complex of $Q$ with respect to the acyclic partial matching $E^\ott$.
The discussion in \cite[Theorem 8.10]{forman1998morse} shows that this complex (defined via counts of flow lines) agrees with the algebraically defined Morse complex from \cite[Equation 7.1]{forman1998morse}, and is homotopic to $C_\bullet(Q, \mathcal F)$. 

\begin{thm}[\cite{forman1998morse}, invariance of discrete Morse homology]
    Let $Q$ be an oriented Morse quiver, $\mathcal F \colon  Q\to \mathcal C $ a sheaf, and $E^\ott$ an acyclic partial matching respecting $\mathcal F$. Then  $C_\bullet(\tilde Q, \widetilde{\mathcal F})$ and $C_\bullet(Q, \mathcal F)$ are chain homotopic.\label{thm:MorseInvariance}
\end{thm}

\section{Quotient Stacks}
\label{subsec:stackBackground}
In this appendix, we collect some background material on quotient stacks, which are our chosen language to handle equivariant data for the abelian group actions in our constructions. In particular, we write an explicit formula in terms of characters for the pushforward of an equivariant sheaf under a finite quotient. 

Let $m:G\times G\to G$ be an abelian group, and let $X$ be a variety equipped with a $G$-action  $a\colon G\times X\to X$. Denote by $p\colon G\times X\to X$ the projection to $X$. A $G$-equivariant sheaf on $X$ is a sheaf $\mathcal F$ on $X$, together with an isomorphism of $\mathcal O_{G\times X}$-modules 
\[\theta\colon a^* \mathcal F\to p^*\mathcal F\]
satisfying the following associativity relation:
\[p^*\theta\circ (1\times a)^*\theta = (m\times 1)^*\theta.\]
A morphisms of equivariant sheaves $(\mathcal F, \theta_{\mathcal F})\to (\mathcal G, \theta_{\mathcal G})$ is a $G$-equivariant morphism of sheaves, that is, a morphism $s\colon \mathcal F\to \mathcal G$ satisfying  $p^*s \theta_{\mathcal F}= \theta_{\mathcal G} a^* s$. We will frequently think of $\theta$ as giving us a product on the sections of $\mathcal F$, and write $g\cdot_G(-)\cdot \mathcal F(U) \to \mathcal F(U)$ for this product where convenient.

We define the category of sheaves on the quotient stack $[X/G]$ to be the category of $G$-equivariant sheaves on $X$. The notation is inspired by the case where $G$ acts freely on $X$ where the categories $\Sh([X/G])$ and $\Sh(X/G)$ are equivalent. In general, these categories are not equivalent. When the quotient $X/G$ is algebraic, there is a pullback functor $\pi^*:\Sh(X/G)\to \Sh([X/G])$.

Given a group homomorphism $f\colon H\to G$, an $f$-morphism $\phi \colon [X/H]\to [Y/G] $ is a morphism $\Phi \colon  X\to Y $  so that the following diagram commutes
\[
	\begin{tikzcd}
		H\times X \arrow{d}{a_H} \arrow{r}{f\times \Phi} &  G\times Y \arrow{d}{a_G}\\
		X\arrow{r}{\Phi}& Y
	\end{tikzcd}
\] 
In this paper, we generally use upper-case characters to denote morphisms between varieties and lower-case characters to denote morphisms between their quotient stacks. We will be interested in explicit descriptions of $\phi_* \colon  \Sh([X/H])\to \Sh([Y/G]) $ and $\phi^* \colon  \Sh([Y/G])\to \Sh([X/H]) $ in the following two cases:
\begin{itemize}
	\item The map $f\colon G\to G$ is a group isomorphism. Then, the definitions of pushforward and pullback are similar to that in the non-equivariant case.
	\item The map $f\colon H\to G$ is a group homomorphism, but the map between varieties $\Phi \colon  X\to X $ is the identity. The constructions of $\phi_*$ and $\phi^*$ arise from representation theory. We are particularly interested in the special case where $H$ is the trivial group.
\end{itemize}

Given any $f$-morphism, we can define the inverse image functor $\phi^* \colon  \Sh([Y/G])\to \Sh([X/H]) $ via $\phi^*(\mathcal F, \theta)= (\Phi^*\mathcal F, (f\times \Phi)^*\theta)$.
In the setting where $\Phi=\id \colon  X\to X $, then the functor $\phi \colon  [X/H]\to [X/G] $ has a simpler expression. Namely, given $(\mathcal F, \theta_F)\in \Sh([X/G])$, the sheaf $\phi^*\mathcal (F, \theta_F)$ is the pair $(\mathcal F, \theta_F(f, -))$. 
\subsection{Pushforward along an \texorpdfstring{$f$}{f}-morphism}
We now focus on the construction of the pushforward when $\Phi=\id$. All of the constructions we consider are variations on the following example.

\begin{example}
It is already interesting to consider the pushforward where $f\colon H \to G$ is an arbitrary group homomorphism, $\Phi:X \to X$ is the identity, and $X$ is a point. A sheaf on $[X/H]$ is a vector space $V$ with a representation $\rho \colon  H\to End(V) $. That is, $\Sh([X/G])=Rep(G)$ and $\Sh([X/H])=Rep(H)$. Pushforward and pullback are then described by a pair of adjoint functors, the forgetful functor $\phi^* \colon Rep(G)\to Rep(H) $ and the coinduction functor $\phi_* \colon Rep(G)\to Rep(H)$.
\end{example}

Now, consider a group homomorphism $f\colon H\to G$ a group homomorphism and $\Phi \colon  X\to X $ the identity on any variety $X$. We denote by $\phi \colon  [X/H]\to [X/G]$ the map on quotient stacks. 
Given a $G$-sheaf $\mathcal F$, we have already defined the pullback $\phi^*\mathcal F$.
We now prove that the right adjoint is given by the coinduced representation
\[\phi_*\mathcal G :=\hom_{\Sh([X/H])}(\KK[G],\mathcal G^K).\]
where 
\begin{itemize}
    \item $K$ is the kernel of $f$
    \item $\mathcal G^K$ is the subsheaf of $K$-invariant sections
    \item $\KK[G]$ has the structure of an $H$-equivariant sheaf via the pushforward action $h\cdot_Hg = f(h)\cdot_G g$
    \item $G$ acts on $x\in \hom_{\Sh([X/H])}(\KK[G],\mathcal G^K)$ by $g\cdot_G x( - ) = x(g\cdot_G (-))$
\end{itemize}
This functor maps morphisms $s\in \hom_{\Sh([X/H])}(\mathcal F,\mathcal G)$  to the postcomposition 
\[\hom_{\Sh([X/H])}(\KK[G],\mathcal F^K)\xrightarrow{\phi_*(s):=s\circ}\hom_{\Sh([X/H])}(\KK[G],\mathcal G^K).\]
We now exhibit the adjunction 
\[\widehat{(-)}: \hom_{[X/H]}(\phi^* \mathcal F, \mathcal G)\leftrightarrow \hom_{[X/G]}(\mathcal F,\phi_* \mathcal G):\widecheck{(-)}\]
Given $s\in \hom_{[X/H]} (\phi^* \mathcal F, \mathcal G)$,
we produce a morphism in $\hat s \in \hom_{[X/G]}(\mathcal F, \phi_* \mathcal G)$ as follows.
For each open $U$, and section $x\in \mathcal F(U)$, we define $\hat s (x)=(g \mapsto s(g\cdot x))$.
The map $\hat s(x) \colon \KK[G]\to \mathcal G^K $ is $H$-equivariant by the following computation:
\begin{align*}
	\hat s(h\cdot_H x)(g)= \hat s(x)(h\cdot_H g) =  s(h\cdot_H g\cdot_G x)= h\cdot_H s(g\cdot_G x) = h\cdot_H \hat s (x)(g)
\end{align*}
The computation also shows that $\hat s(x)(g)$ is $K$-invariant.
We confirm that the morphism $\hat s$ is $G$-equivariant 
\[\hat s(g'\cdot_G x) =(g\mapsto s(g\cdot_G g'\cdot_G x))=g\cdot_G(\hat s(x)). \]
Similarly, for every $t\in \hom_{[X/G]}(\mathcal F, \phi_*\mathcal G)$, we produce a morphism $\check t \in   \hom_{[X/H]} (\phi^* \mathcal F, \mathcal G)$ which sends $y\in \phi^*\mathcal F$ to $\check t(y) = t(y)(1)$.
We also need to verify that the morphism $\check t$ is $H$-equivariant :
\begin{align*}
	\hat t(h\cdot_H y)=& t(h \cdot_H y)(1) = t(f(h)\cdot_G y)(1)\\
  =&f(h)\cdot_G t(y) (1) = t(y)(f(h)\cdot_G 1) = t(y)(h\cdot_H 1)=
  h\cdot_H t(y)(1)
\end{align*}
	We now show that $\widehat{(-)}, \widecheck{(-)}$  are inverses. Let $x \in \mathcal F(U)$, and $t\in \hom_{[X/G]}(\mathcal F, \phi_*\mathcal G)$. 
Then,
\begin{align*}
	\hat{\check t}(x)=(g\mapsto \check t(g\cdot_G x)) = g\mapsto t(g\cdot_G x)(1)=g\mapsto t(x)(g)= t(x)
 \intertext{and}
	\check{\hat s}(y)= \hat s(y)(1) = (g\mapsto s(g\cdot y))(1) = s(y).
\end{align*}
\subsubsection{Pushforward via Characters}
\label{subsubsec:characterPushforward}
We now make further assumptions to obtain a more computable formula for the pushforward. We assume $\kk$ is algebraically closed, the short exact sequence
\[ 0 \to H\to G\to \coker(f) \to 0 \]
splits, so for any $g \in G$, we can uniquely write $g=(h, c)$ for some $h \in H$ and $c \in \coker(f)$, and that $|\coker(f)|\in \KK^\times$. In particular, $\coker(f)$ is finite. Let $\chi_1, \ldots , \chi_k$ denote the characters of $\coker(f)$.
Define the character functor $\phi_c \colon \Sh([X/H])\to \Sh([X/G]) $ to act on objects by
 \[\phi_c\mathcal G:=\bigoplus_{i=1}^k \mathcal G\]
 The $G$-action on $\phi_c\mathcal G$ is defined using the decomposition $g=(h, c)$
 \[
	   (x_1, \ldots, x_k) \mapsto (\chi_1(c)  hx_1, \ldots, \chi_k(c) h x_k).
 \]
 For this definition to be well-defined, the characters must take value in $\kk$. Here we assume that $\kk$ is algebraically closed so that this automatically holds.
On morphisms, our functor sends $s\in \hom_{\Sh([X/H])}  (\mathcal G, \mathcal F)$ to
 \[\phi_c s (x_1, \ldots, x_k)= (s(x_1), \ldots, s(x_k)).\]
 We check that the morphism $\phi_c s$ is $G$-equivariant:
 \[\phi_c(s)(g\cdot_G(x_i)_{i=1}^k)= s(\chi_i(c)h\cdot_Hx_i)_{i=1}^k = (\chi_i(c)h\cdot_Hs(x_i))_{i=1}^k = g\cdot_G\phi_c(s)((x_i)_{i=1}^k)\]
 We prove that the functor $\phi_c$ agrees with $\phi_*$ by exhibiting identifications
\begin{align*}
	\mathfrak{F}:\hom_{\Sh([X/H])}(\kk[G],\mathcal G)\leftrightarrow&  \bigoplus_{i=1}^k\mathcal G: \hat{\mathfrak F}
	\\
\end{align*}
which can be thought of as finite Fourier transforms.
The forward transform is defined via  $\mathfrak F(x)= \frac{1}{|\coker(f)|} \sum_{c\in G}(\chi_i(c^{-1})  x(c))_{i=1}^k$.
We check that $\mathfrak F(x)$ is $G$-equivariant in the $x$-argument
\begin{align*}
	\mathfrak{F}(g'\cdot x)=& \frac{1}{|\coker(f)|} \sum_{c\in \coker(f)}\left(\chi_i(c^{-1})  x(c\cdot g')\right)_{i=1}^k\\
    =& \frac{1}{|\coker(f)|} \sum_{g\in \coker(f)}\left(\chi_i(c^{-1})  (h')^{-1} x(c\cdot c')\right)_{i=1}^k
	\intertext{Reindexing the sum over $c''=c\cdot c'$}
	=& \frac{1}{|\coker(f)|} \sum_{g''\in  \coker(f)}\left(\chi_i((c'')^{-1} c')(h')^{-1}\cdot_H x(c'') \right)_{i=1}^k
	= g'\cdot_G \mathfrak{F}(x)
 \end{align*}

The inverse transformation is defined via 
$\check{\mathfrak F}(y)=\left(g\mapsto \sum_{i=1}^k \chi_i(c) h\cdot y_i\right)$.
We now check that $\check{\mathfrak F}(y)$ is $G$-equivariant in the $y$-argument. 
\begin{align*}
	\check{\mathfrak F}(g'\cdot y) =&\left( g\mapsto\left(  \sum_{i=1}^k \chi_i(c)h\cdot_H \chi_i(c')h'  y_i\right)\right)\\
	=&\left( g\mapsto\left(  \sum_{i=1}^k \chi_i(cc') (hh')y_i\right)\right)= g'\cdot \mathfrak F ( y)
\end{align*}

These maps are inverses by Schur's orthogonality relations. 
\begin{align*}
	\mathfrak F\circ \hat{\mathfrak F}(y)=& \frac{1}{|\coker(f)|} \sum_{c\in \coker(f)}(\chi_i(c^{-1}) \hat{\mathfrak F}(y)(c))_{i=1}^k\\
	=& \frac{1}{|\coker(f)|}\sum_{c\in \coker(f)}\left(\chi_i(c^{-1})\sum_{j=1}^k\chi_j (c)y_j\right)_{i=1}^k\\
	=& \left(\sum_{j=1}^k \langle \chi_i, \chi_j\rangle y_j \right)_{i=1}^k\\
	=& y
\end{align*}
\begin{align*}
	\hat{\mathfrak F} \circ \mathfrak F(x)(g)=& \left(\sum_{i=1}^k \chi_i(c)h \mathfrak F(x)_i\right)\\
	=&\frac{1}{|\coker(f)|}\sum_{i=1}^k \chi_i(c )h\left(\sum_{c'\in \coker(f)} \chi((c')^{-1}) x(c')\right)\\
	=&\frac{1}{|\coker(f)|}\sum_{c'\in \coker(f)} \left(\sum_{i=1}^k \chi_i(c^{-1}c')\right) h x(c')\\
	=&\frac{1}{|\coker(f)|} \sum_{c'\in \coker(f)} \delta_{c, c'} h x(c')
	=hx(c)=x(g)
\end{align*}
In this fashion, the transformations $\mathfrak F, \check{\mathfrak F}$ give an equivalence of functors $\phi_*$ and $\phi_c$. 
\section{Generation and Rouquier dimension} \label{subsec:Rouquierbackground} 
  Here, we summarize some of the results of \cite{ballard2012hochschild} that explain how to deduce \cref{maincor:rdim} from \cref{maincor:diagonal}. We first provide some background on the Rouquier dimension of a triangulated category. Then, we discuss how it can be bounded by the generation time of the diagonal by product objects. 

Let $\mathcal C$ be a triangulated category over a field $\kk$. Given a full subcategory $\mathcal G\subset \mathcal C$, let $\langle \mathcal G \rangle = \langle \mathcal G \rangle_0$ be the smallest full subcategory of $\mathcal C$ containing $\mathcal G$ which is closed under isomorphisms, shifts, finite direct coproducts, and summands.  Inductively define    $$\langle \mathcal G \rangle_k = \langle \langle \mathcal{G} \rangle * \langle \mathcal{G} \rangle_{k-1} \rangle$$ 
where $\langle \mathcal{G} \rangle * \langle \mathcal{G} \rangle_{k-1}$ is the full subcategory of $\mathcal C$ whose objects are isomorphic to cones between objects of $\langle \mathcal G\rangle_0$ and $\langle \mathcal G \rangle_{k-1}$.
Given $\mathcal E\subset \mathcal C$ a full subcategory, and $\mathcal G\subset \mathcal C$ the generation\footnote{We allow for taking direct summands in generation; in some literature this is distinguished from generation and called split-generation.} time of $\mathcal E$ by $\mathcal G$ is 
\[
    \gentime_{\mathcal G}(\mathcal E):=\min\{k\in \NN, \infty \;|\; \mathcal E\subset \langle \mathcal G\rangle_k \} .
\]
Given $G,E$ objects of $\mathcal C$, we will abuse notation and write $G, E$ for their endomorphism algebras considered as full subcategories of $\mathcal C$.
 The \emph{Rouquier dimension} of $\mathcal C$ is defined by 
 \[
     \Rdim(\mathcal C):=\min\{k \;|\; \exists G\in \Ob(\mathcal C) \text{ such that } \gentime_G(\mathcal C)=k\}. 
 \]
Given objects $G_0, \ldots, G_k$, we write $[G_2\xrightarrow{f_1}G_1]$ for the mapping cone of $f_1\colon G_2\to G_1$. We additionally employ the notation  
\[[G_k\xrightarrow{f_{k-1}}G_{k-1}\to \cdots \xrightarrow{f_1} G_1]:=[G_k\xrightarrow{f_{k-1}}[G_{k-1} \cdots\xrightarrow{f_2}  [G_2\xrightarrow{f_1}G_1]\cdots]]]\]
for an iterated mapping cone. If there is an object $G$ of $\mathcal C$ such that for any object $E$ there exist objects $G_i\in \langle G\rangle_0$ such that $E\cong[G_n\to \cdots \to G_1]$, then $\Rdim(\mathcal C)\leq n$.

The main tool that we use for computing the Rouquier dimension of the derived category of coherent sheaves is the  Fourier-Mukai transform. 
\begin{df}(\cite[\href{https://stacks.math.columbia.edu/tag/0FYU}{Tag 0FYU}]{stacks-project})
    Let $X_1, X_2$ be schemes of finite type, and additionally assume that $X_1$ is proper. A \emph{Fourier-Mukai kernel} is an object $F\in D^b\Coh(X_1\times X_2)$. The associated Fourier-Mukai transform is the functor 
    \begin{align*}
        \Phi_F: D^b\Coh (X_1)\rightarrow& D^b\Coh(X_2)\\
        E\mapsto& (\pi_2)_*(F\tensor \pi_1^*(E)).
    \end{align*}
\end{df}
All functors are implicitly derived. Many functors between derived categories of coherent sheaves can be written as Fourier-Mukai transforms. A key example comes from morphisms $f\colon X_1\to X_2$. Let $\Gamma(f)\subset X_1\times X_2$ be the graph of $\Gamma$. Then,
\[\Phi_{\mathcal O_{\Gamma(f)}}=f_*.\]
In particular, $\Phi_{\mathcal O_\Delta}$ is the identity functor.

Fourier-Mukai transforms are triangulated functors. In particular, they send iterated cone decompositions to iterated cone decompositions. They are also triangulated in the argument of the kernel.
\begin{prop}
    Let $F, F'\in D^b\Coh(X_1\times X_2)$ be Fourier-Mukai kernels, and let $F''=\cone(f\colon F\to F')$.
    For every $E\in D^b\Coh(X_1)$, there exists a decomposition
    \[\Phi_{F''}(E)= \cone(\Phi_{F}(E)\to \Phi_{F'}(E)).\]
    \label{lem:FMofCone}
\end{prop}
\begin{prop}
    Suppose that $F=\pi_1^*F_1\tensor \pi_2^*F_2$. Then 
    \[\Phi_F(E)= F_2\tensor \Gamma(F_1\tensor E)\tensor \Gamma(\mathcal O_{X_1}),\]
    where $\Gamma$ is the global sections functor.
    \label{prop:FMofExternalTensor}
\end{prop}
From here on, given sheaves $E_1\in D^b\Coh(X_1), E_2\in D^b\Coh(X_2)$, we write $E_1\boxtimes E_2:= \pi_1^*E_1\otimes \pi_2^*E_2$.
\begin{prop}[Lemma 2.16, \cite{ballard2012hochschild}]
    Consider a proper smooth variety $X$. Suppose that there exists a subcategory $\mathcal G=\langle G'_1\boxtimes G_1, \ldots G'_n\boxtimes G_n\rangle_0 \subset D^b\Coh(X\times X)$, where $G_i, G_i'\in D^b\Coh(X)$ and $G_i$ are irreducible. If $\gentime_{\mathcal G}(\mathcal O_\Delta)=k$ then $\Rdim(D^b\Coh(X))\leq k$.
\end{prop}

\begin{proof}
     We prove that  $\mathcal G_X:=\langle G_1', \ldots, G_n'\rangle_0$ generates any object $E \in D^b\Coh(X)$ in time $k$. 
     By assumption, we can write $\mathcal O_\Delta\cong [\hat F_k\to \cdots \to \hat F_0]$, where each $\hat F_k=\bigoplus_l F_{kl}\boxtimes F_{kl}'$ is a shift of a direct summand of one of the objects $G_i\tensor G_i'$. 
    By \cref{lem:FMofCone}, we can re-express
     \begin{align*}
        E=\Phi_{\mathcal O_\Delta}(E)=& [\Phi_{\hat F_n}\circ E \to \cdots \to \Phi_{\hat F_1}\circ E].
    \intertext{At each term in the iterated mapping cone we apply \cref{prop:FMofExternalTensor} to obtain}
        \Phi_{\hat F_k}(E)=&\bigoplus_l\Gamma(\mathcal O_X)\tensor  \Gamma(G'_{kl}\tensor F)\tensor \Gamma(\mathcal O_X)\tensor G'_{kl}
        \end{align*}
    from which we conclude that $\Phi_{F_k}(E)\in \langle G_1', \ldots, G_n'\rangle_0$. From this, it follows that $E$ is generated in time $k$ by $\mathcal G_X$. 
\end{proof}
\begin{rem}
    Since the only constructions used in the proof are the exactness of the six operations, the statement also holds for other examples (e.g., smooth proper quotient stacks). Additionally, whenever $X$ admits a smooth compactification $\bar X$, the statement holds for $X$ as $D^b\Coh(X)$ is a localization of $D^b\Coh(\bar X)$ and localization does not increase Rouquier dimension. 
\end{rem}
Every line bundle $E$ on the product of smooth toric stacks $\X \times \X$ is of the form $F'\boxtimes F$, and the diagonal is a toric substack. By applying \cref{thm:resolutionOfSubstacks}, we obtain:
\begin{cor}
        If $\X$ is a toric stack covered by smooth stacky charts, $\Rdim(D^b\Coh(\X))\leq \dim(\X)$.
\end{cor} 

\printbibliography 

\Addresses

\end{document}